\documentclass[reqno,11pt]{article}
\usepackage{mathrsfs}
\usepackage{amssymb}
\usepackage{amsthm}
\usepackage{amsmath}
\usepackage{amsfonts}
\usepackage{mathtools}
\usepackage{enumerate}
\usepackage{bm}

\usepackage{graphicx}

\usepackage{bbm}
\usepackage{mathrsfs}
\usepackage{amssymb}
\usepackage[thicklines]{cancel}

\usepackage[usenames,dvipsnames]{xcolor}
\usepackage{soul}

\usepackage{float}
\let\origfigure\figure
\let\endorigfigure\endfigure
\renewenvironment{figure}[1][H]{%
  \origfigure[H]%
}{%
  \endorigfigure%
}

\usepackage{txfonts}

\usepackage{psfrag}
\usepackage[numbers,sort&compress]{natbib}
\usepackage{setspace,color,wrapfig}
\usepackage[colorlinks,linkcolor=blue,citecolor=cyan]{hyperref}

\numberwithin{equation}{section}
\newtheorem{theorem}{Theorem}[section]

\newtheorem{proposition}[theorem]{Proposition}

\theoremstyle{definition}
\newtheorem{definition}[theorem]{Definition}

\newtheorem{remark}[theorem]{Remark}
\theoremstyle{remark}

\begin{document}
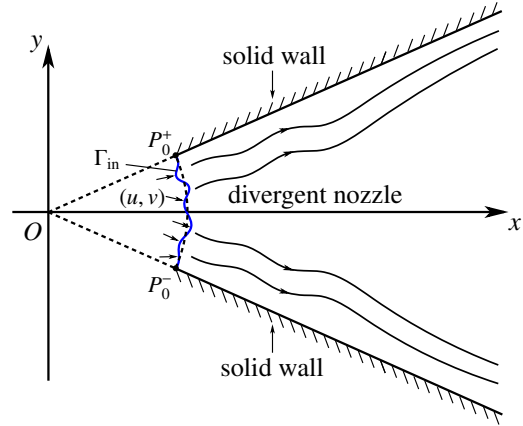
\title{ Continuous   sonic-supersonic  flows in two-dimensional infinitely long divergent nozzles }
\author{Lili Qian\thanks{School of Mathematics, Jilin University, Changchun, Jilin Province, China, 130012.  Email: qianll24@mails.jlu.edu.cn}\and
Chunpeng Wang\thanks{School of Mathematics, Jilin University, Changchun, Jilin Province, China, 130012. Email: wangcp@jlu.edu.cn}\and
Zihao Zhang\thanks{School of Mathematics, Jilin University, Changchun, Jilin Province, China, 130012. Email: zhangzihao@jlu.edu.cn}}
\date{}
\maketitle
\newcommand{\de}{{\mathrm{d}}}
\def\div{{\rm div\,}}
\def\curl{{\rm curl\,}}
\def\th{\theta}
\newcommand{\ro}{{\rm rot}}
\newcommand{\sr}{{\rm supp}}
\newcommand{\sa}{{\rm sup}}
\newcommand{\va}{{\varphi}}
\newcommand{\me}{\mathcal{M}}
\newcommand{\ml}{\mathcal{V}}
\newcommand{\mn}{\mathcal{N}}
\newcommand{\md}{\mathcal{D}}
\newcommand{\mg}{\mathcal{G}}
\newcommand{\mh}{\mathcal{H}}
\newcommand{\mf}{\mathcal{F}}
\newcommand{\ms}{\mathcal{S}}
\newcommand{\mt}{\mathcal{T}}
\newcommand{\ma}{\mathcal{L}}
\newcommand{\mb}{\mathcal{B}}
\newcommand{\mc}{\mathcal{C}}
\newcommand{\mr}{\mathcal{R}}
\newcommand{\mq}{\mathcal{Q}}
\newcommand{\mj}{\mathcal{J}}
\newcommand{\mw}{\mathcal{W}}
\newcommand{\my}{\mathcal{U}}
\newcommand{\n}{\nabla}
\newcommand{\m}{\Omega}
 \newcommand{\q}{{\rm R}}
\newcommand{\p}{{\partial}}
\newcommand{\ld}{{\tilde}}
\newcommand{\h}{\hat}
\pagestyle{myheadings} \markboth{\hfill     2-D continuous   sonic-supersonic  flows  \hfill}{\hfill 2-D continuous   sonic-supersonic  flows  \hfill}\maketitle
\begin{abstract}
This paper concerns the   structural stability of radially symmetric sonic-supersonic  flows in a two-dimensional infinitely long divergent nozzle with straight solid walls. For a given inlet that is a small perturbation of an arc centered at the vertex of the nozzle,
        we establish the global existence  and uniqueness of a continuous  sonic-supersonic potential flow which is close to the radially symmetric  sonic-supersonic  flow. The flow is sonic at the inlet, with its velocity along the normal direction, and becomes supersonic in the nozzle.
Such a sonic-supersonic model is governed by a quasilinear nonstrictly hyperbolic equation with degeneracy at the inlet, and the degeneracy is weak in the sense that there exist exactly two distinct characteristics from each sonic point, except for those at the walls, pointing into the supersonic region. One of the crucial ingredients of
the analysis is to  obtain the precise asymptotic behavior near the sonic inlet by the  method of characteristics, which is essential for establishing the local existence of continuous sonic-supersonic flows. Another one is  to extend the local solution globally and derive uniform estimates in the supersonic region.
\end{abstract}

\begin{center}
\begin{minipage}{5.5in}
Mathematics Subject Classifications 2020:  35Q31, 35L80, 76J20, 76N10.\\
Key words: sonic-supersonic flows, quasilinear nonstrictly hyperbolic equation, degeneracy.
\end{minipage}
\end{center}
\tableofcontents
\section{   Introduction  and main results}\noindent
\par In this paper, we focus on continuous   sonic-supersonic  flows in a two-dimensional infinitely long divergent nozzle. Compressible flow problems naturally arise in physical experiments and engineering designs, and there are many experiments and numerical simulations and rigorous theories involved in this field \cite{B58,CF48}. Two-dimensional steady compressible flows are governed by the following Euler system:
\begin{align}
&\frac{\partial}{\partial x}(\rho u)+\frac{\partial}{\partial y}(\rho v)=0,\label{1-eq1.1}\\
&\frac{\partial}{\partial x}(P+\rho u^{2})+\frac{\partial}{\partial y}(\rho uv)=0,\label{1-eq1.2}\\
&\frac{\partial}{\partial x}(\rho uv)+\frac{\partial}{\partial y}(P+\rho v^{2})=0,\label{1-eq1.3}
\end{align}
where $(u,v)$, $ \rho$ and $P$ represent the velocity, the density and the pressure, respectively. The flow is assumed to be isentropic so that $P=P(\rho)$ is a smooth function. In particular, for the polytropic gas with the adiabatic exponent $\gamma>1$,
\begin{align}\label{1-eq1.4}
   P(\rho)=\frac{1}{\gamma}\rho^\gamma
\end{align}
is the normalized pressure. Assume further that the flow is irrotational, i.e.
\begin{equation}\label{1-eq1.5}
  \frac{\partial u}{\partial y}=\frac{\partial v}{\partial x}.
\end{equation}
Then the density $\rho$ is expressed in terms of the speed $q$ according to the Bernoulli law
\cite{CF48}
\begin{equation}\label{1-eq1.6}
  \rho(q^2)=\left(1-\frac{\gamma-1}{2}q^2\right)^{1/(\gamma-1)},\quad q=\sqrt{u^2+v^2},\quad0\leq q\leq c^*=\sqrt{2/(\gamma-1)}.
\end{equation}
The sound speed $c$ is defined as $c^2=P'(\rho)$. At the sonic state, the speed is
$c_*=\sqrt{2/(\gamma+1)}$, which is critical in the sense that the flow is subsonic when
 $q<c_*$, sonic when $q=c_*$, and supersonic when $q>c_*$.
  The Euler system \eqref{1-eq1.1}--\eqref{1-eq1.5} can be transformed into the full potential equation
  \begin{equation}\label{1-eq1.7}
  \div(\rho(|\n \varphi|^2)\n \varphi)=0,
  \end{equation}
  where $\varphi$ is a velocity potential with $\n\varphi=(u,v)$, and $\rho$  is  given by \eqref{1-eq1.6}.
 It is noted that \eqref{1-eq1.7} is elliptic in the subsonic region, hyperbolic in the supersonic region,
 and degenerate at the sonic state.
 \par For sonic-supersonic flows, it was Bers \cite{B50} who first investigated the extension of a
 subsonic-sonic flow across its sonic curve. He showed  that a subsonic-sonic flow can be locally and uniquely extended across the sonic curve as a supersonic flow without discontinuity, provided no exceptional points exist on the sonic curve. The sonic-supersonic flow was formulated as a Cauchy problem for the linear Chaplygin equations in the hodograph plane, with no boundary conditions prescribed. This continuation was also established in \cite{ZZ14} by using the coordinate system $(\sqrt{q^2-c^2}, \varphi)$.
  The authors in \cite{WX16} further characterized the structure of
 the sonic curve of a smooth transonic flow. It was shown that for a smooth transonic flow, the set of exceptional points is either empty, a single point, or a closed line segment on which the velocity potential is identically constant. At any non-exceptional point, there exist exactly two distinct characteristics that point into the supersonic region. In contrast, at an interior point of an exceptional segment, the two characteristics coincide with the segment itself and do not extend locally into the supersonic region. Furthermore, it was proved in \cite{WX19} that if the throat is suitably flat, then there exists a unique smooth transonic irrotational flow of Meyer type of a de Laval nozzle, with all sonic points being exceptional. The flow constructed in \cite{WX19} is supersonic at the outlet, and its supersonic extension was studied in \cite{WX15}. It was shown that this local flow can be extended to a global smooth transonic flow of Meyer type in an infinitely long nozzle, provided that the upper wall beyond the local flow region is convex. Furthermore, the author in \cite{W22} investigated the global existence of smooth sonic-supersonic potential flows in a two-dimensional expanding nozzle with the critical geometry at the inlet. For the full steady Euler system, the authors established the local existence of classical sonic-supersonic solutions in \cite{HL20} and further obtained global smooth supersonic-sonic solutions and analyzed their behavior near the sonic curve in \cite{HL21}. The   local existence of sonic-supersonic jet
flows in two-dimensional nozzles was studied in \cite{L23}.   There are also several works  focusing on sonic-supersonic flows for the two-dimensional pseudo steady Euler equations, see \cite{HL19,ZZ17} and references therein.

\par  For finitely long straight convergent nozzles, there exist radially symmetric continuous subsonic-sonic flows. The structural stability of these flows was first established in \cite{WX13}. More precisely, for a prescribed incoming mass flux and an inlet that is a small perturbation of a circular arc centered at the nozzle vertex, the authors  proved that there exists an open interval, depending only on the adiabatic exponent and the arc length, such that if the incoming mass flux belongs to this interval and the inlet perturbation is sufficiently small, then there exists a unique subsonic-sonic flow, which is a perturbation of the corresponding radially symmetric subsonic-sonic flow with the same incoming mass flux.
Later on, the authors in \cite{NW16,GNW21,NW25} established the unique existence of continuous subsonic-sonic flows under perturbations of the inlet together with the incoming flow angle, the upper wall, and the inlet together with the upper wall, respectively. Recently,  the structural stability of radially symmetric
  subsonic-sonic flows under perturbations of the inlet, the upper wall, and the incoming flow angle was investigated in \cite{NZ26}.
 \par For   infinitely long divergent nozzles, the global existence of smooth three-dimensional supersonic Euler flows was established in \cite{XY21}, while the existence of global continuous two-dimensional sonic-supersonic flows was proved in \cite{L21}.  On the other hand, radially symmetric continuous sonic-supersonic flows also exist in such nozzles.  However, the structural stability of these flows has not yet been investigated.  Consequently,  for a given inlet that is a small perturbation of an
arc centered at the vertex of the nozzle and a given incoming mass flux, we will establish  the existence and stability of two-dimensional continuous sonic-supersonic flows  whose velocity is orthogonal to the inlet.
\subsection{Governing equations}\noindent
\par Define a velocity potential $\varphi$ and a stream function $\psi$ by
$$\dfrac{\partial\varphi}{\partial x}=q\cos\theta,\quad\dfrac{\partial\varphi}{\partial y}=q\sin\theta,\quad\dfrac{\partial\psi}{\partial x}=-\rho q\sin\theta,\quad\dfrac{\partial\psi}{\partial y}=\rho q\cos\theta, $$
 where $\theta$ is the angle of the velocity inclination to the $x$-axis. Then the potential equation \eqref{1-eq1.7} can be reduced to the following Chaplygin equations (\cite{B58}):
\begin{align}\label{1-eq2.6}
\frac{\partial\theta}{\partial\psi}+\frac{\rho(q^2)+2q^2\rho'(q^2)}{q\rho^2(q^2)}\frac{\partial q}{\partial\varphi}=0,\quad\frac{1}{q}\frac{\partial q}{\partial\psi}-\frac{1}{\rho(q^2)}\frac{\partial\theta}{\partial\varphi}=0
\end{align}
in the potential-stream coordinates $(\varphi,\psi).$ The coordinates transformation between the two coordinate systems is valid at least in the absence of stagnation and vacuum points
since $\small{\frac {\partial ( \varphi , \psi ) }{\partial ( x, y) }= \rho q^{2}}$. Eliminating $\theta$ from \eqref{1-eq2.6} yields the following second-order quasi-linear equation:
\begin{equation}\label{1-eq2.7}
  \frac{\partial^2A(q)}{\partial\varphi^2}+\frac{\partial^2B(q)}{\partial\psi^2}=0,
\end{equation}
where
$$A(q)=\int_{c_*}^q\frac{\rho(s^2)+2s^2\rho'(s^2)}{s\rho^2(s^2)}ds,\quad B(q)=\int_{c_*}^q\frac{\rho(s^2)}{s}ds,\quad0<q<c^*.$$
Here, $B(\cdot)$ is strictly increasing in $(0,c^*)$, while $A(\cdot)$ is strictly increasing in $(0,c_*]$ and strictly decreasing in $[c_*,c^*)$. It can be checked easily that both \eqref{1-eq1.7} and \eqref{1-eq2.7}
are elliptic in the subsonic region $(q<c_*)$ and hyperbolic in the supersonic region
$(q>c_{*})$, while singular at the sonic state $(q=c_*)$ and the vacuum $(q=c^*).$
Furthermore, as shown in \cite{WX19}, in the supersonic region, \eqref{1-eq2.7} can be rewritten
as
$$Q_{\varphi\varphi}-(b(Q)Q_\psi)_\psi=0,$$
or equivalent to the following first-order hyperbolic system:
$$\begin{cases}
W_\varphi+b^{1/2}(Q)W_\psi=\frac{1}{4}b^{-1}(Q)p(Q)W(W+Z),
\\[2ex]Z_\varphi-b^{1/2}(Q)Z_\psi=-\frac{1}{4}b^{-1}(Q)p(Q)Z(W+Z),
\end{cases}$$
where
\begin{equation*}
  Q=A(q),\quad W=Q_\varphi-b^{1/2}(Q)Q_\psi,\quad Z=-Q_\varphi-b^{1/2}(Q)Q_\psi,
\end{equation*}
and
$$
b(s)=\left.\left(\frac{\gamma+1}{2}q^{2}-1\right)^{-1}\left(1-\frac{\gamma-1}{2}q^{2}\right)^{2/(\gamma-1)+1}\right|_{q=A_{+}^{-1}(s)}>0,\quad s<0,
$$
$$
p(s)=\left(\gamma+1\right)q^{4}\left.\left(\frac{\gamma+1}{2}q^{2}-1\right)^{-3}\left(1-\frac{\gamma-1}{2}q^{2}\right)^{3/(\gamma-1)+1}\right|_{q=A_{+}^{-1}(s)}>0,\quad s<0,
$$
with $A_+^{-1}$ being the inverse function of $A(\cdot)$ lying in $[c_*,c^*)$.

\subsection{Mathematical formulation in the physical plane}\noindent
\par We first formulate the  problem in the physical plane.  Assume that the two walls of the nozzle are located symmetrically with respect to the $x$-axis
with the vertex being (0,0) and the angle at the vertex being $2\theta_0$. Let $P_{0}^{-} = (R_{0}\cos\theta_{0},-R_{0}\sin\theta_{0})$ and $P_{0}^{+} = (R_{0}\cos\theta_{0},R_{0}\sin\theta_{0})$ with $R_{0}> 0$
being two fixed points on the lower and upper walls respectively.

Give an inlet from $P_0^-$ to $P_0^+$
$$\Gamma_{\mathrm{in}}:x=g(y),\quad -R_0\sin\theta_0\leq y\leq R_0\sin\theta_0,$$
where  $g\in C^4([-R_0\sin\theta_0,R_0\sin\theta_0])$ satisfies
\begin{equation}\label{1-eq2.1}
g(\pm R_0\sin\theta_0)=R_0\cos\theta_0 \ \ {\rm{and}} \ \ g'(\pm R_0\sin\theta_0)=\mp\tan\theta_0,
\end{equation}
The angle $\Theta$ of the velocity at the inlet is orthogonal to the inlet, that is,
$$
\Theta(y)=-\arctan g'(y),\quad-R_0\sin\theta_0\leq y\leq R_0\sin\theta_0.
$$
Here $\Theta \in C^3([-R_0\sin\theta_0,R_0\sin\theta_0])$ satisfies
$$
\Theta(\pm R_0\sin\theta_0)=\pm \theta_0.
$$
In this paper, we  focus on  continuous sonic-supersonic flows in the two-dimensional infinite divergent nozzle $\Omega$ (See Figure \ref{1-pic1}), which is  bounded by the nozzle walls $y=\pm x\tan\theta_0$ and the given inlet $\Gamma_{\mathrm{in}}$.
 \input{exten1.TPX}
 On the nozzle walls, the flow satisfies the slip condition
$$
\frac{\partial \varphi}{\partial\nu}(x,y)=0,\quad y=\pm x\tan\theta_0,
$$
 where $\nu$ is the unit outer normal to the walls. At the inlet, since the angle of the velocity of the flow is orthogonal to
the inlet, the velocity potential is a constant. Without loss of generality, we assume that it is to be zero. That is to say,
the boundary condition at the inlet is
$$
\varphi\left(g(y),y\right)=0 \ \ {\rm{and}} \ \ |\nabla \varphi(g(y),y)|=c_*,\quad -R_0\sin\theta_0\leq y\leq R_0\sin\theta_0.
$$
Furthermore, the incoming mass flux of the flow is
\begin{align*}
&\int_{\Gamma_{\mathrm{in}}}\left|\nabla\varphi(x,y)\right|\rho\left(\left|\nabla\varphi(x,y)
\right|^2\right)ds\\
&=\int_{-R_0\sin\theta_0}^{R_0\sin\theta_0}
\frac{\left|\nabla\varphi(g(y),y)\right|\rho\left(\left|\nabla\varphi(g(y),y)
\right|^2\right)}{\cos\Theta(y)}dy\\
&=\int_{-R_0\sin\theta_0}^{R_0\sin\theta_0}\frac{c_*^{2/(\gamma-1)+1}}{\cos\Theta(y)}dy=m.
\end{align*}
Therefore, the problem in the physical plane can be formulated as
\begin{align}
&\div(\rho(|\nabla\varphi|^2)\nabla\varphi)=0, \quad&&(x,y)\in \Omega,\label{1-eq2.10}\\
&\varphi(g(y),y)=0, \quad |\nabla \varphi(g(y),y)|=c_*, \quad &&-R_0\sin\theta_0<y<R_0\sin\theta_0, \label{1-eq2.11}\\
&\frac{\partial \varphi}{\partial \nu}(x,x\tan\theta_0)=0, \quad &&R_0\cos\theta_0<x<+\infty, \label{1-eq2.12}\\
&\frac{\partial \varphi}{\partial \nu}(x,-x\tan\theta_0)=0, \quad &&R_0\cos\theta_0<x<+\infty, \label{1-eq2.13}\\
&|\nabla\varphi(x,y)|>c_*,\quad &&(x,y)\in\Omega.\label{1-eq2.15}
\end{align}
\begin{definition}
A function $\varphi$ is called a solution to the problem \eqref{1-eq2.10}--\eqref{1-eq2.15} if $\varphi \in C^1(\overline{\Omega})$ with
$$
c_* \leq |\nabla \varphi(x, y)| \leq c^*, \quad (x, y) \in \Omega,
$$
and $\varphi$ satisfies \eqref{1-eq2.10} in the distribution sense and satisfies \eqref{1-eq2.11}--\eqref{1-eq2.15} {pointwise}.
\end{definition}
\subsection{Mathematical formulation in the potential plane}\noindent
\par In the subsection, we formulate the problem in the potential plane. Assume that $P_0^-$ in physical coordinates
is transformed into the origin $(0,0)$ in the potential plane, and $P_0^+$ is transformed into $(0,m)$. Let the speed
of the flow at the upper wall be denoted by
$$
q(x,x\tan\theta_0)=\mathscr Q_{\mathrm{up}}(x),\quad R_0\cos\theta_0\leq x< +\infty.
$$
At the inlet, the stream function is
\begin{equation*}
  \psi\left(g(y),y\right)=\Psi_{\mathrm{in}}(y)=\int_{-R_0\sin\theta_0}^{y}\frac{c_*^{2/(\gamma-1)+1}}{\cos\Theta(y)}dy,\quad -R_0\sin\theta_0\leq y \leq R_0\sin\theta_0,
\end{equation*}
which satisfies $$\Psi_\mathrm{in}(-R_0\sin\theta_0)=0 \ \ {\rm{ and}}  \ \ \Psi_\mathrm{in}(R_0\sin\theta_0)=m.$$
Denote by $Y_{\mathrm{in}}$ the inverse function of $\Psi_\mathrm{in}$, i.e.
\begin{equation*}
 Y_\mathrm{in}(\psi)=\Psi_\mathrm{in}^{-1}(\psi),\quad 0\leq\psi\leq m.
\end{equation*}
The potential function at the upper wall is
\begin{equation*}
 \varphi\left(x,x\tan\theta_0\right)=\Phi_{\mathrm{up}}(x)=\int_{R_0\cos\theta_0}^{x}\frac{\mathscr Q_{\mathrm{up}}(x)}{\cos\theta_0}dx,\quad R_0\cos\theta_0\leq x < +\infty,
\end{equation*}
which satisfies $$\Phi_\mathrm{up}(R_0\cos\theta_0)=0.$$
Denote by $X_{\mathrm{up}}$ the inverse function of $\Phi_\mathrm{up}$, i.e.
\begin{equation*}
 X_\mathrm{up}(\varphi)=\Phi_\mathrm{up}^{-1}(\varphi),\quad0\leq\varphi< +\infty.
\end{equation*}
Therefore, the problem in the potential plane can be formulated as
 \begin{align}
&Q_{\varphi\varphi}-(b(Q)Q_{\psi})_{\psi}=0,&&(\varphi,\psi)\in(0,+\infty)\times(0,m),\label{1-eq2.22}  \\
&Q(0,\psi)=0,&&\psi\in(0,m), \label{1-eq2.23}\\
&Q_\varphi(0,\psi)=
-\frac{1}{c_*\rho(c_*^2)}\left.\frac{\mathrm{d}}{\mathrm{d}y}\left(\sin\Theta(y)\right)\right|_{y=Y_{\mathrm{in}}(\psi)}   ,&&\psi\in(0,m), \label{1-eq2.24} \\
&Q_\psi(\varphi,0)=0,&&\varphi\in(0,+\infty),\label{1-eq2.25}\\
&Q_{\psi}(\varphi,m)=0,&&\varphi\in(0,+\infty),\label{1-eq2.26}
\end{align}
which is equivalent to
\begin{align}
&W_{\varphi}+b^{1/2}(Q)W_{\psi}=\frac{1}{4}b^{-1}(Q)p(Q)W(W+Z),&& (\varphi,\psi)\in(0,+\infty)\times(0,m),\label{1-eq2.28} \\
&Z_{\varphi}-b^{1/2}(Q)Z_{\psi}=-\frac{1}{4}b^{-1}(Q)p(Q)Z(W+Z),&& (\varphi,\psi)\in(0,+\infty)\times(0,m),\label{1-eq2.29} \\
&W(0,\psi)=W_0(\psi)=
-\frac{1}{c_*\rho(c_*^2)}\left.\frac{\mathrm{d}}{\mathrm{d}y}\left(\sin\Theta(y)\right)\right|_{y=Y_{\mathrm{in}}(\psi)},&&\psi\in(0,m), \label{1-eq2.30} \\
&Z(0,\psi)=Z_0(\psi)=
\frac{1}{c_*\rho(c_*^2)}\left.\frac{\mathrm{d}}{\mathrm{d}y}\left(\sin\Theta(y)\right)\right|_{y=Y_{\mathrm{in}}(\psi)}, &&\psi\in(0,m), \label{1-eq2.31} \\
&W(\varphi,0)+Z(\varphi,0)=0,&&\varphi\in(0,+\infty), \label{1-eq2.32} \\
&W(\varphi,m)+Z(\varphi,m)=0,&& \varphi\in(0,+\infty),\label{1-eq2.33} \\
&Q_{\varphi}(\varphi,\psi)=\frac{1}{2}\left(W(\varphi,\psi)-Z(\varphi,\psi)\right),&& (\varphi,\psi)\in(0,+\infty)\times(0,m),\label{1-eq2.34}\\
&Q(0,\psi)=0,&&\psi\in(0,m).\label{1-eq2.35}
\end{align}
\begin{definition}
A function $Q$ is said to be a weak solution to the problem \eqref{1-eq2.22}--\eqref{1-eq2.26}, if  $Q\in C_{\rm{loc}}^{0,1}([0, +\infty) \times [0, m])$ satisfies \eqref{1-eq2.23} and
$$\begin{aligned}
&\int_{0}^{+\infty}\int_{0}^{m}\left(Q_{\varphi}(\varphi,\psi)\xi_{\varphi}(\varphi,\psi)-b(Q(\varphi,\psi))Q_{\psi}(\varphi,\psi)\xi_{\psi}(\varphi,\psi)\right)d\varphi d\psi\\
&\quad-\int_{0}^{m}\frac{1}{c_*\rho(c_*^2)}\left.\frac{\mathrm{d}}{\mathrm{d}y}
\left(\sin\Theta(y)\right)\right|_{y=Y_{\mathrm{in}}(\psi)}\xi(0,\psi)d\psi=0
\end{aligned}$$
for any $ \xi \in C^{1}([0,+\infty)\times[0,m])$ with compact support.
\end{definition}
\begin{definition}
 A triad of functions $(W,Z,Q)$ is said to be a weak solution to the problem \eqref{1-eq2.28}--\eqref{1-eq2.35} if
 $W,Z\in L_{loc}^{\infty}([0,+\infty)\times[0,m])$ and  $Q\in C_{\rm{loc}}^{0,1}([0,+\infty)\times[0,m])$ satisfy the following properties.
    \begin{enumerate}[ \rm (i)]
       \item  {For any $\xi,\eta \in C^{1}([0,+\infty)\times[0,m])$
with compact support satisfying
\begin{equation*}
\xi(\varphi,0)=-\eta(\varphi,0),\quad
\xi(\varphi,m)=-\eta(\varphi,m),\quad\varphi\in(0,+\infty),
\end{equation*}}
       it holds that \begin{equation*}
       \begin{aligned}
&\int_{0}^{+\infty}\int_{0}^{m}\bigg(W\xi_{\varphi}+b^{1/2}(Q)W\xi_{\psi}+\frac{1}{2}b^{-1/2}(Q)p(Q)Q_{\psi}W\xi
+\frac{1}{4}b^{-1}(Q)p(Q)W(W+Z)\xi\bigg)d\varphi d\psi \\
&\quad+\int_{0}^{+\infty}\int_{0}^{m}\bigg(Z\eta_{\varphi}-b^{1/2}(Q)Z\eta_{\psi}-\frac{1}{2}b^{-1/2}(Q)p(Q)Q_{\psi}Z\eta
-\frac{1}{4}b^{-1}(Q)p(Q)Z(W+Z)\eta\bigg)d\varphi d\psi \\
&\quad{+\int_0^m
\bigl(W_0(\psi)\xi(0,\psi)+Z_0(\psi)\eta(0,\psi)\bigr)d\psi}=0.
\end{aligned}\end{equation*}
\item \eqref{1-eq2.34} holds in the sense of distribution.
\item $Q$ satisfies \eqref{1-eq2.35}.
   \end{enumerate}
\end{definition}
\subsection{Background solutions and main results}\noindent
\par
 We first consider radially  symmetric continuous sonic-supersonic flows. Set
$$
\hat{\Gamma}_{\mathrm{in}}: x=g_0(y)=\sqrt{R_0^2-y^2},\quad-R_0\sin\theta_0\leq y\leq R_0\sin\theta_0,
$$
and
$$\hat{\Theta}(y)=-\arctan g_0'(y)=\arctan\frac{y}{\sqrt{R_0^2-y^2}},
\quad-R_0\sin\theta_0\leq y\leq R_0\sin\theta_0.$$
 Then \eqref{1-eq2.22}--\eqref{1-eq2.26} can be reduced to the following boundary problem:
\begin{align}
&(A(\hat{q}))^{\prime\prime}(\varphi)=0,\label{1-eq2.37}\\
&(A(\hat{q}))^{\prime}(0)=-\frac{1}{\hat{q}(0)\rho(\hat{q}^{2}(0))R_{0}},\label{1-eq2.38}\\
&\hat{q}(0)=c_{*},\label{1-eq2.39}\\
&2\hat{q}(0)\rho(\hat{q}^{2}(0))R_{0}\theta_{0}=m.\label{1-eq2.40}
\end{align}
By solving  \eqref{1-eq2.37}--\eqref{1-eq2.40}, the solution can be represented as
\begin{equation}\label{1-eq2.41}
A(\hat{q}(\varphi))=-\frac{\varphi}{c_*^{1+2/(\gamma-1)}R_0}{:=-\mu_0\varphi},
\end{equation} which will be called the{ background solution.}
Furthermore, one can get from \eqref{1-eq2.41} that
\begin{equation*}
{\hat{q}(\varphi)=A_+^{-1}(-\mu_0\varphi).}
\end{equation*}
\par
For the problem \eqref{1-eq2.22}--\eqref{1-eq2.26}, to guarantee the $C^2$ smoothness of solution near the corners $(0,0)$
and $(0,m)$, one should impose the compatibility conditions
$$
\frac{\mathrm{d}^2}{\mathrm{d}y^2} (\sin \Theta (\pm R_0 \sin \theta_0)) = 0,
$$
which is equivalent to
\begin{equation}\label{1-eq2.43}
\left.\left(\frac{g''(y)}{(1+(g'(y))^2)^{3/2}}\right)'\right|_{y = \pm R_0 \sin \theta_0} = 0.
\end{equation}
The main result  of this paper can be stated as follows.
\begin{theorem}
\label{1-th2.1}
Given   $g\in C^4([-R_0\sin\theta_0,R_0\sin\theta_0])$ satisfying \eqref{1-eq2.1} and \eqref{1-eq2.43}, define
\begin{align*}
 \sigma:=\bigg\|\frac{g^{\prime}(y)}{\sqrt{(g^{\prime}(y))^2+1}}-\frac{g_0^{\prime}(y)}
 {\sqrt{(g_0^{\prime}(y))^2+1}}\bigg\|_{C^3([-R_0\sin\theta_0,R_0\sin\theta_0])}.
 \end{align*}
 There exists a constant $\sigma_1^*>0$ depending only on $\gamma$ and $R_0$ such that if
\begin{align}\label{1-eq2.2}
 \sigma\leq{\sigma}_1^*,
\end{align}
{then the problem \eqref{1-eq2.22}--\eqref{1-eq2.26} admits at least one weak solution $Q\in C_{\rm{loc}}^{1,1}([0,+\infty)\times[0,m])$ satisfying
\begin{eqnarray}
 { -C_2\sigma\varphi\leq Q(\varphi, \psi) + \mu_0\varphi \le C_1\sigma\varphi, \  \ (\varphi, \psi) \in (0, +\infty) \times (0, m),}\label{1-eq2.44-a}\\
          -C_2\sigma\leq   Q_\varphi(\varphi, \psi) + \mu_0 \le C_1\sigma, \ \
             (\varphi, \psi) \in (0, +\infty) \times (0, m),\label{1-eq2.44}
        \end{eqnarray}
    \begin{equation}\label{1-eq2.45}
        |Q_\psi(\varphi, \psi)| \le C_3\sigma\varphi, \quad  (\varphi, \psi) \in (0, \zeta] \times (0, m),
    \end{equation}
    \begin{equation}\label{1-eq2.451}
        |Q_\psi(\varphi, \psi)| \le C_4\sigma\varphi^{(\gamma+1)/2}, \quad (\varphi, \psi) \in (\zeta, +\infty) \times (0, m),
    \end{equation}
 \begin{align}\label{1-eq2.46}
|Q_{\varphi\varphi}(\varphi,\psi)|\leq C_5\sigma\varphi^{1/2},\ \ |Q_{\varphi\psi}(\varphi,\psi)|\leq C_5\sigma,\ \ |Q_{\psi\psi}(\varphi,\psi)|\leq C_5\sigma\varphi,  \ \
(\varphi,\psi)\in(0,\zeta)\times(0,m),
\end{align}}
and
\begin{align}\label{1-eq2.461}
|Q_{\varphi\varphi}(\varphi,\psi)|\leq C_6\sigma\varphi^{-1},\  |Q_{\varphi\psi}(\varphi,\psi)|\leq C_6\sigma\varphi^{(\gamma-1)/2},\ |Q_{\psi\psi}(\varphi,\psi)|\leq C_6\sigma\varphi^{\gamma},  \
(\varphi,\psi)\in(\zeta,+\infty)\times(0,m).
\end{align}
Here $ \zeta\in(0,1) $ is a sufficiently small constant and  $C_i$ $(i=1,\ldots,6)$ $(C_1\leq C_2)$ are positive constants depending only on $\gamma$ and $R_0$.
\end{theorem}

The following theorem shows the uniqueness of weak solutions.
\begin{theorem}
\label{1-th2.2}
Given   $g\in C^4([-R_0\sin\theta_0,R_0\sin\theta_0])$ satisfying \eqref{1-eq2.1} and \eqref{1-eq2.43}. There exists a constant $\sigma_2^*\in(0,\sigma_1^*]$ depending only on $\gamma$, $m$ and $R_0$ such that if
\begin{align}\label{1-eq2.2fu}
 \sigma\leq \sigma_2^*,
\end{align} then the problem \eqref{1-eq2.22}--\eqref{1-eq2.26} admits at most one weak solution $Q\in C_{\rm{loc}}^{0,1}([0,+\infty)\times[0,m])$ satisfying \eqref{1-eq2.44}--\eqref{1-eq2.451} with {the constants $C_i$ $(i=1,2,3,4)$ given in Theorem \ref{1-th2.1}}.
\end{theorem}
In terms of the physical variables, the above theorems can be transformed as follows.
\begin{theorem}
\label{1-th2.1*}
Under the assumptions of Theorem \ref{1-th2.2}, the  problem \eqref{1-eq2.10}--\eqref{1-eq2.15} admits a unique smooth {solution $\varphi \in C^1(\overline{\Omega})\cap C_{\rm loc}^{2,1}(\Omega)$,} which satisfies \eqref{1-eq2.44}--\eqref{1-eq2.461} for $Q = A(|\nabla \varphi|)$ in $\Omega$.
\end{theorem}
\subsection{Key ideas for the proof}\label{pr}\noindent
\par  Now we give the key ideas for the proof of the main results in
this paper. The flow is sonic  with the velocity normal to the inlet, and becomes supersonic within the nozzle. Such a sonic-supersonic model is governed by a quasilinear
nonstrictly hyperbolic equation with degeneracy at the inlet. For the background solution, the sonic curve coincides with the inlet, which is a circular arc of radius \(R_0\). Thus the degeneracy is weak in the sense that  there exist exactly two distinct characteristics from
each sonic point, except for those on the walls, pointing into the supersonic region.  Consequently, it was shown in \cite{WX16}  that the Cauchy problem for the smooth supersonic continuation from such a sonic curve is well-posed, which is in contrast to the cases considered \cite{WX19,W22}. The degeneracy at the sonic curve in \cite{WX19,W22} is strong in the sense that  all characteristics from the sonic points coincide with the sonic curve and never approach
the supersonic region.
\par In order to establish the structural stability of radially symmetric sonic-supersonic flow problem under the perturbation of the inlet  in the potential plane, we divide the problem into two regions: the region near the sonic state ($\varphi\in[0,{\zeta}]$) and the region away from the sonic state ($\varphi\in [{\zeta},+\infty)$).  This decomposition is motivated by the different asymptotic behaviors of the flow in these two regions. In particular, it is essential to characterize both the singular behavior near the sonic curve ($\varphi=0$) and the decay behavior as $\varphi\to+\infty$  to establish uniform estimates.

\par Firstly, for \(\varphi \in [0,{\zeta}]\), the primary challenge lies in the singularity and degeneracy of the governing equations near the sonic curve. To resolve this sonic-supersonic flow problem, the key ingredient is to determine the precise asymptotic behavior of the flow speed near the sonic boundary. More precisely, the existence of a local sonic-supersonic flow is established by the following two steps:
\begin{enumerate}[\rm (i)]
 \item We first derive the asymptotic behavior of $b$, $p$, $p'$ and their combinations near the sonic curve directly from their explicit expressions. These properties  determine the behavior of the coefficients  in the characteristic equations and provide the precise pointwise estimates required to construct a suitable iteration set around the background flow. For any given function in this set, we establish the well-posedness and a priori estimates for the corresponding linearized problem via the method of characteristics. The local existence of a $C^{0,1}$ sonic-supersonic solution then follows from the Schauder fixed-point theorem.

 \item  To improve the regularity of the solution, we introduce suitable weighted variables involving $W_\varphi$ and $Z_\varphi$ and derive their governing characteristic equations. Combining the method of characteristics with the established asymptotic properties, we obtain precise estimates  for $Q_{\varphi\varphi}$, $Q_{\varphi\psi}$, and $Q_{\psi\psi}$. These estimates yield the desired $C^{1,1}$ regularity for the local sonic-supersonic flow.
\end{enumerate}

\par Secondly, for \(\varphi \in [{\zeta},+\infty)\), the governing equation becomes strictly hyperbolic. The local sonic-supersonic solution  provides the initial data at \(\varphi = {\zeta}\). We first verify that these initial data satisfy the matching condition established in \cite{WX15}. This condition guarantees that the flow remains away from both the sonic and vacuum states {in any bounded region} and
  prevents shock formation in the supersonic region. Therefore, by the global existence theory  developed in \cite{WX15}, the local solution can be extended to the entire supersonic region. It then remains to establish global regularity, which is achieved through the following two steps:
 \begin{enumerate}[ \rm (i)]
 \item We first establish uniform bounds for \(W\) and \(Z\) by combining a contradiction argument with the method of characteristics. These estimates yield the global \(C^{0,1}\) estimate for \(Q\).
 \item  We then utilize the asymptotic behavior of \(b\), \(p\), \(p'\) and their combinations as $\varphi\to+\infty$, together with the  obtained  $C^{0,1}$ estimates, to derive uniform estimates for suitable weighted variables involving $ W_\varphi$ and $ Z_\varphi$  by the method of characteristics. These estimates then yield the global $C^{1,1}$ estimates for $Q$.
\end{enumerate}
\par
 This paper is arranged as follows.  In Section \ref{2se}, we investigate the sonic-supersonic flow near the sonic curve and establish the existence and $C^{1,1}$ regularity of local solutions. In Section \ref{3glo},
 we extend the local solution to the entire supersonic region and establish the global $C^{0,1}$ and $C^{1,1}$ estimates. Finally, in Section \ref{un},
 we establish the uniqueness of the solution.

\section{Existence of sonic-supersonic flows near the sonic curve}\label{2se}\noindent
\par  Fix $\zeta\in(0,1)$ sufficiently small
depending only on $\gamma$ and $R_0$
to be specified below.
For $\varphi \in[0,{\zeta}]$, the problem \eqref{1-eq2.22}--\eqref{1-eq2.26} can be formulated as
\begin{align}
&Q_{\varphi\varphi}-(b(Q)Q_{\psi})_{\psi}=0,&&(\varphi,\psi)\in(0,{\zeta})\times(0,m),\label{1-eq2.221}  \\
&Q(0,\psi)=0,&&\psi\in(0,m), \label{1-eq2.222}\\
&Q_\varphi(0,\psi)=
-\frac{1}{c_*\rho(c_*^2)}\left.\frac{\mathrm{d}}{\mathrm{d}y}\left(\sin\Theta(y)\right)\right|_{y=Y_{\mathrm{in}}(\psi)}   ,&&\psi\in(0,m), \label{1-eq2.223} \\
&Q_\psi(\varphi,0)=0,&&\varphi\in(0,{\zeta}),\label{1-eq2.224}\\
&Q_{\psi}(\varphi,m)=0,&&\varphi\in(0,{\zeta}),\label{1-eq2.225}
\end{align}
 which will be solved by the   iteration method. {More precisely, for a given $\tilde Q$ belonging to some suitable set, define
 $$\tilde W=\tilde Q_\varphi-b^{1/2}(\tilde Q)\tilde Q_\psi \ \
{\rm{and}} \ \
\tilde Z=-\tilde Q_\varphi-b^{1/2}(\tilde Q)\tilde Q_\psi,  \ \ (\varphi,\psi)\in(0,\zeta)\times(0,m).$$
  We will solve the following  linearized problem:}
\begin{align}
&W_{\varphi}+b^{1/2}(\tilde{Q})W_{\psi}=-\frac{1}{2}b^{-1/2}(\tilde{Q})p(\tilde{Q})\tilde{Q}_{\psi}W,&&(\varphi,\psi)\in(0,{\zeta})\times(0,m),\label{1-eq3.7}\\
&Z_{\varphi}-b^{1/2}(\tilde{Q})Z_{\psi}=\frac{1}{2}b^{-1/2}(\tilde{Q})p(\tilde{Q})\tilde{Q}_{\psi}Z,&&(\varphi,\psi)\in(0,{\zeta})\times(0,m),\label{1-eq3.8}\\
&W(0,\psi)=W_0(\psi)=
-\frac{1}{c_*\rho(c_*^2)}\left.\frac{\mathrm{d}}{\mathrm{d}y}\left(\sin\Theta(y)\right)
\right|_{y=Y_{\mathrm{in}}(\psi)},&&\psi\in(0,m),\label{1-eq3.9}\\
&Z(0,\psi)=Z_0(\psi)=\frac{1}{c_*\rho(c_*^2)}\left.
\frac{\mathrm{d}}{\mathrm{d}y}\left(\sin\Theta(y)\right)\right|_{y=Y_{\mathrm{in}}(\psi)},&&\psi\in(0,m),\label{1-eq3.10}\\
&W(\varphi,0)+Z(\varphi,0)=0,&&\varphi\in(0,{\zeta}),\label{1-eq3.11}\\
&W(\varphi,m)+Z(\varphi,m)=0,&&\varphi\in(0,{\zeta}),\label{1-eq3.12}\\
&Q_{\varphi}(\varphi,\psi)=\frac{1}{2}\left(W(\varphi,\psi)-Z(\varphi,\psi)\right),&& (\varphi,\psi)\in(0,{\zeta})\times(0,m),\label{1-eq3.13}\\
&Q(0,\psi)=0,&&\psi\in(0,m).\label{1-eq3.14}
\end{align}
Then
a direct computation yields that
\begin{align}
&\lim_{s\to0^-}b(s)=\lim_{s\rightarrow0^-}p(s)=+\infty,\quad
\lim_{s\to0^-}(-s)^{1/4}b^{1/2}(s)=\hat{\mu},\quad
\lim_{s\to0^-}(-s)^{-1/4}b^{-1/2}(s)=\frac{1}{\hat{\mu}},\label{1-eq2.91}\\
&\lim_{s\to0^-}(-s)^{3/2}p(s)=\frac12\hat{\mu}^2,
\quad \lim_{s\to0^-}(-s)^{5/2}p'(s)=\frac34\hat{\mu}^2,\quad
\lim_{s\to0^-}(-s)b^{-1}(s)p(s)=\frac{1}{2}\label{1-eq2.9},
\end{align}
where $\hat{\mu}$ is a positive constant depending only on $\gamma$.

\begin{remark}
\label{1-rem Q0}
{\it {By the definition of $\sigma$ and $\mu_0$, it holds that
 \begin{align}\label{1-eq2.27}
 \bigg|Q_\varphi(0,\psi)+{\mu_0}\bigg| \leq \frac{\sigma}{c_*^{1+2/(\gamma-1)}}.
\end{align}
 Set \begin{align}\label{1-eq2.27mu}
\mu_1=\mu_0(1-R_0\sigma),\quad
\mu_2=\mu_0(1+R_0\sigma).
\end{align}}
Then if $\sigma \le \frac{\gamma - 1}{20 R_0(\gamma+1)}$, one has
\begin{align}\label{1-mu1}
\mu_1\geq \frac{19\gamma + 21}{20R_0(\gamma+1)c_*^{1+2/(\gamma-1)}}:=\tilde\mu_1,\quad\frac{\mu_2}{\mu_1}\leq\frac{21\gamma + 19}{19\gamma + 21}.
\end{align}
 It follows from   \eqref{1-eq3.9}, \eqref{1-eq3.10} and \eqref{1-eq2.27} that
\begin{equation}\label{1-eq2.36}
\mu_1\leq -W_0(\psi),Z_0(\psi)\leq \mu_2,\quad \psi \in(0,m).
\end{equation}}
\end{remark}
{\begin{remark}
\label{1-rem3.1}
{\it By \eqref{1-eq3.7}, \eqref{1-eq3.8} and \eqref{1-eq3.13}, the compatibility condition is of the form
\begin{equation}\label{1-eq3.15}\begin{aligned}
&Q_\psi(\varphi,\psi)
=-\frac12 b^{-1/2}(\tilde Q(\varphi,\psi))
\bigg(W(\varphi,\psi)+Z(\varphi,\psi)\bigg)\\
&\ \ -\frac14\int_0^\varphi
b^{-1}(\tilde Q(s,\psi))p(\tilde Q(s,\psi))
\tilde Q_\psi(s,\psi)
\bigg(W(s,\psi)-Z(s,\psi)\bigg)ds\\
&\ \ -\frac14\int_0^\varphi
b^{-3/2}(\tilde Q(s,\psi))p(\tilde Q(s,\psi))
\tilde Q_\varphi(s,\psi)
\bigg(W(s,\psi)+Z(s,\psi)\bigg)ds,\ \
(\varphi,\psi)\in(0,\zeta)\times(0,m).\end{aligned}
\end{equation}
At a fixed point $\tilde Q=Q$, the compatibility relation becomes
\begin{equation}\label{1-eq3.15-a}
Q_\psi(\varphi,\psi)
=-\frac12 b^{-1/2}(Q(\varphi,\psi))
\bigg(W(\varphi,\psi)+Z(\varphi,\psi)\bigg),
\quad(\varphi,\psi)\in(0,\zeta)\times(0,m).
\end{equation}
}
\end{remark}}

\par
In the following two subsections, we first establish the  well-posedness of the linearized problem \eqref{1-eq3.7}--\eqref{1-eq3.14}, and then solve the nonlinear problem \eqref{1-eq2.221}--\eqref{1-eq2.225}.
{To simplify the notation}, we use $\left|(W,Z)\right|$ to denote $\max\left\{\left|W\right|,\left|Z\right|\right\}$.
\subsection{The well-posedness of the linearized problem }\label{2-1-sonic} \noindent
\par Given $\tilde{Q}\in C^{0,1}([0,{\zeta}]\times[0,m])$ satisfying
\begin{align}
-\beta_2\varphi \leq &\tilde{Q}(\varphi,\psi)\leq -\beta_1\varphi, \quad(\varphi,\psi)\in(0,{\zeta})\times(0,m),\label{1-eq3.16}\\
-\beta_2\leq \tilde{Q}_{\varphi}(\varphi,\psi)&\leq -\beta_1,\quad \left|\tilde{Q}_{\psi}(\varphi,\psi)\right|\leq\beta_3\varphi,\quad (\varphi,\psi)\in(0,{\zeta})\times(0,m)\label{1-eq3.17}
\end{align}
for some positive constants $\beta_i$ $(i=1,2,3)$, where
$$
\beta_1\ge{\mu_0}-\tilde{\beta}_1\sigma,\ \ \beta_2
\le{\mu_0+\tilde{\beta}_2\sigma\leq2\mu_0},  \ \  \beta_3=\tilde{\beta}_3\sigma
$$
 and $\tilde{\beta}_i$ $(i=1,2,3)$ are positive constants to be specified below.
Then \eqref{1-eq3.16} gives
\[
0<\beta_1\varphi
\leq-\tilde Q(\varphi,\psi)
\leq\beta_2\varphi
\leq2\mu_0\zeta,
\qquad
(\varphi,\psi)\in(0,\zeta]\times[0,m].
\]
It follows from \eqref{1-eq3.16} and \eqref{1-eq3.17} that
\begin{align}
-\tilde{\beta}_2\sigma\varphi \leq \tilde{Q}(\varphi,\psi)+&{\mu_0\varphi}\leq \tilde{\beta}_1\sigma\varphi, \quad(\varphi,\psi)\in(0,{\zeta})\times(0,m),\label{1-eq3.16*}\\
-\tilde{\beta}_2\sigma\leq \tilde{Q}_{\varphi}(\varphi,\psi)+{\mu_0}&\leq \tilde{\beta}_1\sigma
,\quad \left|\tilde{Q}_{\psi}(\varphi,\psi)\right|\leq\tilde{\beta}_3\sigma\varphi,\quad (\varphi,\psi)\in(0,{\zeta})\times(0,m).\label{1-eq3.17*}
\end{align}
We choose  $\zeta\in(0,1)$ sufficiently small,
depending only on $\gamma$ and $R_0$, so that
the coefficient bounds obtained from
\eqref{1-eq2.91}--\eqref{1-eq2.9}
hold for $-2\mu_0\zeta\leq s<0$.
These bounds can be applied to $\tilde Q$
on $(0,\zeta]\times[0,m]$. Then combining \eqref{1-eq2.91}--\eqref{1-eq2.9}  with \eqref{1-eq3.16}--\eqref{1-eq3.17} leads to
\begin{align}
\left|\frac{1}{2}b^{-1/2}(\tilde{Q}(\varphi,\psi))p(\tilde{Q}(\varphi,\psi))\tilde{Q}_{\psi}(\varphi,\psi)\right|\leq \frac{\hat{\mu}\beta_3}{3\beta_1^{5/4}}\varphi^{-1/4},\quad(\varphi,\psi)\in(0,{\zeta})\times(0,m).\label{1-eq3.18}
\end{align}
 Next, we solve the problem \eqref{1-eq3.7}--\eqref{1-eq3.14}  under the assumptions \eqref{1-eq3.16}, \eqref{1-eq3.17} and \eqref{1-eq3.18}.
\begin{proposition}\label{1-pro3.1}
Given  $(g,\tilde{Q},W_0,Z_0) \in C^4([-R_0\sin\theta_0,R_0\sin\theta_0])\times C^{0,1}([0,{\zeta}]\times[0,m])\times L^{\infty}(0,m)\times L^{\infty}(0,m)$ satisfying \eqref{1-eq2.1}, \eqref{1-eq2.43}, \eqref{1-eq2.36} and \eqref{1-eq3.16}--\eqref{1-eq3.18}, the problem \eqref{1-eq3.7}--\eqref{1-eq3.14} admits a unique solution $(W,Z,Q)\in L^{\infty}((0,{\zeta})\times(0,m))\times L^{\infty}((0,{\zeta})\times(0,m))\times C^{0,1}([0,{\zeta}]\times[0,m])$ that satisfies
\begin{align}
&\mu_1\exp\bigg\{-\frac{4\hat{\mu}\beta_3{\zeta^{3/4}}}{9\beta_1^{5/4}}\bigg\}
\leq -W(\varphi,\psi),Z(\varphi,\psi)\leq
\mu_2\exp\bigg\{\frac{4\hat{\mu}\beta_3{\zeta^{3/4}}}{9\beta_1^{5/4}}\bigg\}, &&\quad(\varphi,\psi)\in(0,{\zeta})\times(0,m),\label{1-eq3.21}\\
&\left|W(\varphi,\psi)+Z(\varphi,\psi)\right|\leq \bigg(\frac{\mu_4\sigma}{\beta_1^{1/4}}
+\frac{8\hat{\mu}\mu_2\beta_3}{9\beta_1^{5/4}}\bigg)\exp\bigg\{\frac{4\hat{\mu}\beta_3{\zeta^{3/4}}}{9\beta_1^{5/4}}\bigg\}\varphi^{3/4},&&\quad(\varphi,\psi)\in(0,{\zeta})\times(0,m),\label{1-eq3.22}\\
&-\mu_2\exp\bigg\{\frac{4\hat{\mu}\beta_3{\zeta^{3/4}}}{9\beta_1^{5/4}}\bigg\}\varphi
  \leq Q(\varphi,\psi)\leq -\mu_1\exp\bigg\{-\frac{4\hat{\mu}\beta_3{\zeta^{3/4}}}{9\beta_1^{5/4}}\bigg\}\varphi,&&\quad(\varphi,\psi)\in(0,{\zeta})\times(0,m),\label{1-eq3.23}\\
&-\mu_2\exp\bigg\{\frac{4\hat{\mu}\beta_3{\zeta^{3/4}}}{9\beta_1^{5/4}}\bigg\}\leq Q_{\varphi}(\varphi,\psi)\leq
-\mu_1\exp\bigg\{-\frac{4\hat{\mu}\beta_3{\zeta^{3/4}}}{9\beta_1^{5/4}}\bigg\},&&\quad(\varphi,\psi)\in(0,{\zeta})\times(0,m),\label{1-eq3.24}\\
&|Q_{\psi}(\varphi,\psi)|
\leq{}
{
\bigg(
\mu_5\bigg(\frac{\beta_2^{1/4}}{\beta_1^{1/4}}
+\frac{4\beta_2}{15\beta_1}\bigg)\sigma
+\frac{\mu_2}{2\beta_1}
\bigg(
\frac{\beta_2^{1/4}}{\beta_1^{1/4}}
+\frac{8}{15}
}\nonumber\\
&{
\qquad\qquad\qquad
+\frac{4\beta_2}{15\beta_1}
\bigg)\beta_3\bigg)
\exp\bigg\{
\frac{4\hat{\mu}\beta_3{\zeta^{3/4}}}{9\beta_1^{5/4}}
\bigg\}\varphi
},
&&\quad
(\varphi,\psi)\in(0,{\zeta})\times(0,m),\label{1-eq3.25}
\end{align}
where $\mu_1$ and $\mu_2$ $(\mu_1\leq\mu_2)$ are   positive constants given by \eqref{1-eq2.27mu} depending only on   $\gamma$, $ R_0$ and $ \sigma$, while $\hat{\mu}$, $\mu_4$ and $\mu_5$ are positive constants depending only on $\gamma$.
\end{proposition}

\begin{proof}
{According to the analysis of Cauchy problems for degenerate-hyperbolic equations in \cite{B50}, the problem \eqref{1-eq3.7}--\eqref{1-eq3.12} admits a unique weak solution
$(W,Z)\in L^{\infty}((0,\zeta)\times(0,m))\times L^{\infty}((0,\zeta)\times(0,m))$. Next, we derive the estimates for $(W,Z)$.
Assume that
\begin{align*}
\Sigma_\pm&: \Psi_\pm'=\pm b^{1/2}(\tilde{Q}(\varphi,\Psi_\pm(\varphi))),  \\
0&<\Psi_\pm<m,\quad\hat{\varphi}_\pm\leq\varphi\leq\check{\varphi}_\pm\quad(0\leq\hat{\varphi}_\pm<\check{\varphi}_\pm\leq \zeta)
\end{align*}}
are positive and negative characteristics, respectively. On $\Sigma_+$, $W$ satisfies
\begin{align*}
\frac{d}{d\varphi}W(\varphi,\Psi_+(\varphi))=
-\frac{1}{2}b^{-1/2}(\tilde{Q}(\varphi,\Psi_+(\varphi)))p(\tilde{Q}(\varphi,\Psi_+(\varphi)))\tilde{Q}_{\psi}(\varphi,\Psi_+(\varphi))W(\varphi,\Psi_+(\varphi)),\quad\hat{\varphi}_+\leq\varphi\leq\check{\varphi}_+,
\end{align*}
which can be rewritten as
\begin{align*}
\bigg(W(\varphi,\Psi_+(\varphi))\exp\bigg\{\int_{\hat{\varphi}_+}^{\varphi}\frac{1}{2}b^{-1/2}(\tilde{Q}(s,\Psi_+(s)))p(\tilde{Q}(s,\Psi_+(s)))\tilde{Q}_{\psi}(s,\Psi_+(s))ds\bigg\}\bigg)^{\prime}=0,\quad\hat{\varphi}_+\leq\varphi\leq\check{\varphi}_+.
\end{align*}
Then for any $\hat{\varphi}_+\leq\varphi\leq\check{\varphi}_+$, it holds that
\begin{align}
&W(\check{\varphi}_+,\Psi_+(\check{\varphi}_+))\nonumber\\
=\,&W(\hat{\varphi}_+,\Psi_+(\hat{\varphi}_+))
\exp\bigg\{\int_{\hat{\varphi}_+}^{\check{\varphi}_+}-\frac{1}{2}b^{-1/2}\bigl(\tilde{Q}(s,\Psi_+(s))\bigr)p(\tilde{Q}(s,\Psi_+(s)))\tilde{Q}_{\psi}(s,\Psi_+(s))ds\bigg\}.\label{1-eq3.26}
\end{align}
Similarly, on $\Sigma_-$, Z satisfies
\begin{align*}
\frac{d}{d\varphi}Z(\varphi,\Psi_-(\varphi))=
\frac{1}{2}b^{-1/2}(\tilde{Q}(\varphi,\Psi_-(\varphi)))p(\tilde{Q}(\varphi,\Psi_-(\varphi)))\tilde{Q}_{\psi}(\varphi,\Psi_-(\varphi))
Z(\varphi,\Psi_-(\varphi)),\quad\hat{\varphi}_-\leq\varphi\leq\check{\varphi}_-,
\end{align*}
which can be expressed as
\begin{align*}
\bigg(Z(\varphi,\Psi_-(\varphi))\exp\bigg\{\int_{\hat{\varphi}_-}^{\varphi}-\frac{1}{2}b^{-1/2}(\tilde{Q}(s,\Psi_-(s)))p(\tilde{Q}(s,\Psi_-(s)))\tilde{Q}_{\psi}(s,\Psi_-(s))ds\bigg\}\bigg)^{'}=0,\quad\hat{\varphi}_-\leq\varphi\leq\check{\varphi}_-.
\end{align*}
Then for any $\hat{\varphi}_-\leq\varphi\leq\check{\varphi}_-$, it holds that
\begin{align}\label{1-eq3.27}
&Z(\check{\varphi}_-,\Psi_-(\check{\varphi}_-))\nonumber\\
=\,&Z(\hat{\varphi}_-,\Psi_-(\hat{\varphi}_-))
\exp\bigg\{\int_{\hat{\varphi}_-}^{\check{\varphi}_-}\frac12b^{-1/2}(\tilde{Q}(s,\Psi_-(s)))p(\tilde{Q}(s,\Psi_-(s)))\tilde{Q}_{\psi}(s,\Psi_-(s))ds\bigg\}.
\end{align}
\input{keven1.TpX}
\input{kodd1.TpX}

{In the following, we use the method of characteristics to estimate $W$ and $Z$. Fix $(\varphi_0,\psi_0)\in(0,\zeta)\times(0,m)$.
Let $\psi=\Psi_1(\varphi)$ be the positive characteristic starting from the point $(\varphi_0,\psi_0)$, which approaches either
$\{0\}\times[0,m]$ or $(0,\zeta)\times\{0\}$ at a point $(\varphi_1,\psi_1)$. If $\varphi_1>0$,
then there exists a negative characteristic $\psi=\Psi_2(\varphi)$ from $(\varphi_1,\psi_1)$,
which approaches either $\{0\}\times[0,m]$ or $(0,\zeta)\times\{m\}$ at a point $(\varphi_2,\psi_2)$. If $\varphi_2>0$,
then there exists a positive characteristic $\psi=\Psi_3(\varphi)$ from $(\varphi
_2,\psi_2)$, which approaches either $\{0\}\times[0,m]$ or $(0,\zeta)\times\{0\}$ at a point $(\varphi_3,\psi_3)$.
Therefore, there exists a nonnegative integer $k$ such that}
$$
\varphi_0>\varphi_1>\cdots>\varphi_k>\varphi_{k+1}=0,
$$
and
$$
\psi_j=\left\{
\begin{aligned}
&0,\quad1\le j\le k\mbox{ and $j$ is odd},
\\
&m,\quad1\le j\le k\mbox{ and $j$ is even},
\end{aligned}
\right.
\quad0\le\psi_{k+1}\le m.
$$
Define a function $r_1(s)$ $(0\leq s<\varphi)$ as follows:
\begin{equation*}
\begin{aligned}
r_1(s)=\left\{
\begin{aligned}
&-\dfrac{1}{2}b^{-1/2}(\tilde{Q}(s,\Psi_j(s)))p(\tilde{Q}(s,\Psi_j(s)))\tilde{Q}_{\psi}(s,\Psi_j(s)),
\quad
\\
&\quad\quad\quad\quad\quad\quad\varphi_{j}\le s<\varphi_{j-1},\, 1\le j\le k+1\mbox{ and $j$ is odd},
\\[2.5 mm]
&\dfrac{1}{2}b^{-1/2}(\tilde{Q}(s,\Psi_j(s)))p(\tilde{Q}(s,\Psi_j(s)))\tilde{Q}_{\psi}(s,\Psi_j(s)),
\quad
\\
&\quad\quad\quad\quad\quad\quad\varphi_{j}\le s<\varphi_{j-1},\,1\le j\le k+1\mbox{ and $j$ is even},
\end{aligned}
\right.
\end{aligned}
\end{equation*}
where
$$
\left\{
\begin{aligned}
&\Psi'_{j}(\varphi)=(-1)^{j-1}b^{1/2}(\tilde Q(\varphi,\Psi_{j}(\varphi))),
\quad\varphi_{j}<\varphi<\varphi_{j-1},
\\
&\Psi_{j}(\varphi_j)=\psi_{j},\quad \Psi_{j}(\varphi_{j-1})=\psi_{j-1},
\end{aligned}
\right.
\quad1\le j\le k+1.
$$
\par Similarly, let $\psi=\Psi_1^*(\varphi)$ be the negative characteristic starting from the point $(\varphi_0,\psi_0)$,  the points $(\varphi_j^*,\psi_j^*)$ $(1\le j\le {k^*}+1)$ can be defined in the same way.
Then there exists a nonnegative integer ${k^*}$ such that
$$
\varphi_0>\varphi_1^*>\cdots>\varphi_{{k^*}}^*>\varphi_{{k^*}+1}^*=0,
$$
and
$$
\psi_j^*=\left\{
\begin{aligned}
&m,\quad1\le j\le {k^*}\mbox{ and $j$ is odd},
\\
&0,\quad1\le j\le {k^*}\mbox{ and $j$ is even},
\end{aligned}
\right.
\quad0\le\psi_{{k^*}+1}^*\le m,
$$
Define a function $r_1^*(s)$ $(0\leq s<\varphi)$ as follows:
\begin{equation*}
\begin{aligned}
r_1^*(s)=\left\{
\begin{aligned}
&\dfrac{1}{2}b^{-1/2}(\tilde{Q}(s,\Psi_j^*(s)))p(\tilde{Q}(s,\Psi_j^*(s)))\tilde{Q}_{\psi}(s,\Psi_j^*(s)),
\quad
\\
&\quad\quad\quad\quad\quad\quad\varphi_{j}^*\le s<\varphi_{j-1}^*,\, 1\le j\le {k^*}+1\mbox{ and $j$ is odd},
\\[2.5 mm]
&-\dfrac{1}{2}b^{-1/2}(\tilde{Q}(s,\Psi_j^*(s)))p(\tilde{Q}(s,\Psi_j^*(s)))\tilde{Q}_{\psi}(s,\Psi_j^*(s)),
\quad
\\
&\quad\quad\quad\quad\quad\quad\varphi_{j}^*\le s<\varphi_{j-1}^*,\,1\le j\le {k^*}+1\mbox{ and $j$ is even},
\end{aligned}
\right.
\end{aligned}
\end{equation*}
where $$
\left\{
\begin{aligned}
&(\Psi_{j}^*)'(\varphi)=(-1)^{j}b^{1/2}(\tilde Q(\varphi,{\Psi_{j}^*(\varphi)})),
\quad\varphi_{j}^*<\varphi<\varphi_{j-1}^*,
\\
&\Psi_{j}^*(\varphi_j^*)=\psi_{j}^*,\quad \Psi_{j}^*(\varphi_{j-1}^*)=\psi_{j-1}^*,
\end{aligned}
\right.
\quad1\le j\le {k^*}+1.
$$
\input{kevenz1.TpX}
\input{koddz1.TpX}

\par For $1\leq j \leq k+1$, it follows from \eqref{1-eq3.26} and \eqref{1-eq3.27} that
\begin{align}\label{1-eq3.26*}
W(\varphi_{j-1},\psi_{j-1})=W(\varphi_j,\psi_j)
\exp\bigg\{\int_{\varphi_j}^{\varphi_{j-1}}r_1(s)ds\bigg\}\mbox{ if $j$ is odd}\nonumber,\\
Z(\varphi_{j-1},\psi_{j-1})=Z(\varphi_j,\psi_j)
\exp\bigg\{\int_{\varphi_j}^{\varphi_{j-1}}r_1(s)ds\bigg\}\mbox{ if $j$ is even}.
\end{align}
If $k$ is even, one can get from \eqref{1-eq3.9}--\eqref{1-eq3.12} and \eqref{1-eq3.26*} that
\begin{align}\label{1-eqw1}
W(\varphi_0,\psi_0)=\,&W(\varphi_1,\psi_1)\exp\bigg\{\int_{\varphi_1}^{\varphi_0}r_1(s)ds\bigg\}
=-Z(\varphi_1,\psi_1)\exp\bigg\{\int_{\varphi_1}^{\varphi_0}r_1(s)ds\bigg\}\nonumber\\
=&-Z(\varphi_2,\psi_2)\exp\bigg\{\int_{\varphi_2}^{\varphi_0}r_1(s)ds\bigg\}
=W(\varphi_2,\psi_2)\exp\bigg\{\int_{\varphi_2}^{\varphi_0}r_1(s)ds\bigg\}\nonumber\\
=\,&W_0(\psi_{k+1})\exp\bigg\{\int_{0}^{\varphi_0}r_1(s)ds\bigg\}
\end{align}
{and if $k^*$ is even},
\begin{align}\label{1-eqz1}
Z(\varphi_0,\psi_0)=\,&Z_0(\psi_{{k^*}+1}^*)\exp\bigg\{\int_{0}^{\varphi_0}r_1^*(s)ds\bigg\}.
\end{align}
Similarly, if $k$ {and $k^*$ are} odd, one derives
\begin{align}\label{1-eqw2}
W(\varphi_0,\psi_0)=-Z_0(\psi_{k+1})\exp\bigg\{\int_{0}^{\varphi_0}r_1(s)ds\bigg\},
\quad Z(\varphi_0,\psi_0)=-W_0(\psi_{{k^*}+1}^*)\exp\bigg\{\int_{0}^{\varphi_0}r_1^*(s)ds\bigg\}.
\end{align}
\begingroup
Since $Z_0=-W_0$, these representations can be written, independently
of the parities of $k$ and $k^*$, as
\begin{align*}
W(\varphi_0,\psi_0)
=W_0(\psi_{k+1})\exp\bigg\{\int_0^{\varphi_0}r_1(s)ds\bigg\},\ \
Z(\varphi_0,\psi_0)
=Z_0(\psi_{k^*+1}^*)\exp\bigg\{\int_0^{\varphi_0}r_1^*(s)ds\bigg\}.
\end{align*}
\endgroup
It follows from \eqref{1-eq2.36}, \eqref{1-eq3.18}, \eqref{1-eqw1} and \eqref{1-eqw2} that in both cases,
\begin{align*}
W(\varphi_0,\psi_0)
&\leq-\mu_1\exp\bigg\{-\frac{4\hat{\mu}\beta_3}{9\beta_1^{5/4}}\varphi_0^{3/4}\bigg\}
\leq-\mu_1\exp\bigg\{-\frac{4\hat{\mu}\beta_3{\zeta^{3/4}}}{9\beta_1^{5/4}}\bigg\}
\end{align*}
and
\begin{align*}
W(\varphi_0,\psi_0)
&\geq-\mu_2\exp\bigg\{\frac{4\hat{\mu}\beta_3}{9\beta_1^{5/4}}\varphi_0^{3/4}\bigg\}
\geq-\mu_2\exp\bigg\{\frac{4\hat{\mu}\beta_3{\zeta^{3/4}}}{9\beta_1^{5/4}}\bigg\}.
\end{align*}
Similarly, the estimate of $Z$ can be obtained as follows:
\begin{align*}
\mu_1\exp\bigg\{-\frac{4\hat{\mu}\beta_3{\zeta^{3/4}}}{9\beta_1^{5/4}}\bigg\}\leq Z(\varphi,\psi)
\leq \mu_2\exp\bigg\{\frac{4\hat{\mu}\beta_3{\zeta^{3/4}}}{9\beta_1^{5/4}}\bigg\},\quad
(\varphi,\psi)\in(0,{\zeta})\times(0,m).
\end{align*}
Thus the estimate \eqref{1-eq3.21} holds.

{We now turn to estimate $\left|W+Z\right|$. Note that
$\Psi_+(0)=\Psi_{k+1}(0)=\psi_{k+1}$
and
$\Psi_-(0)=\Psi_{k^*+1}^*(0)=\psi_{k^*+1}^*$. If $k$ {and $k^*$ are} even, it holds that
\begin{align*}
\psi_0 - \Psi_+(0) = \int_0^{\varphi_0} s_+(\varphi)b^{1/2}(\tilde Q(\varphi,{\Psi_+(\varphi)})) d\varphi,\quad\psi_0 - \Psi_-(0) = \int_0^{\varphi_0} s_-(\varphi)b^{1/2}(\tilde Q(\varphi,{\Psi_-(\varphi)})) d\varphi,
\end{align*}}
where  $s_{\pm}(\varphi)\in \{1, -1\}$ are piecewise constant sign functions defined as follows:
\begin{align*}
    s_+(\varphi) &= (-1)^j, \quad \text{for } \varphi \in (\varphi_{j+1}, \varphi_j], \quad j = 0, 1, 2, \ldots, k, \\
    s_-(\varphi) &= (-1)^{j+1}, \quad \text{for } \varphi \in (\varphi_{j+1}^*, \varphi_j^*], \quad j = 0, 1, 2, \ldots, {k^*}.
\end{align*}
Then it follows from \eqref{1-eq2.91} that
\begin{align}\label{1-eqvar}
\left|\Psi_+(0)-\Psi_-(0)\right|
&\leq \int_0^{\varphi_0} {\left|s_-(\varphi)b^{1/2}(\tilde Q(\varphi,\Psi_-(\varphi)))-s_+(\varphi)b^{1/2}(\tilde Q(\varphi,\Psi_+(\varphi)))\right|} d\varphi\nonumber\\
&\leq {2\int_{0}^{\varphi_0}\|b^{1/2}(\tilde Q)(\varphi,\cdot)\|_{L^\infty}d\varphi}
\leq\frac{\mu_3}{\beta_1^{1/4}}\varphi_0^{3/4},
\end{align}
where $\mu_3$ is a positive constant depending only on $\hat{\mu}$. This,  together with \eqref{1-eq3.9}, \eqref{1-eq3.10}, \eqref{1-eqw1} and \eqref{1-eqz1}, yields that
\begin{align}\label{1-eqwz1}
\left|W(\varphi_0,\psi_0)+Z(\varphi_0,\psi_0)\right|
=\,&\left|W_0(\psi_{k+1})\exp\bigg\{\int_{0}^{\varphi_0}r_1(s)ds\bigg\}+Z_0(\psi_{{k^*}+1}^*)\exp\bigg\{\int_{0}^{\varphi_0}r_1^*(s)ds\bigg\}\right|\nonumber\\
\leq\,&\left|W_0(\Psi_+(0))+Z_0(\Psi_-(0))\right|\exp\bigg\{\int_{0}^{\varphi_0}r_1(s)ds\bigg\}\nonumber\\
&\quad+\left|Z_0(\Psi_-(0))\right|\left|\exp\bigg\{\int_{0}^{\varphi_0}r_1^*(s)ds\bigg\}-\exp\bigg\{\int_{0}^{\varphi_0}r_1(s)ds\bigg\}\right|\nonumber\\
\leq\,&\left|\frac{1}{c_*^{2/(\gamma-1)+1}}\bigg(\left.\frac{d}{dy}\left(\sin\Theta(y)\right) \right|_{y=Y_{\mathrm{in}}(\psi)}\bigg)'(\Psi_+(0)-\Psi_-(0))\right|\exp\bigg
\{\frac{4\hat{\mu}\beta_3}{9\beta_1^{5/4}}\varphi_0^{3/4}\bigg\}\nonumber\\
&\quad+\frac{8\hat{\mu}\mu_2\beta_3}{9\beta_1^{5/4}}\exp\bigg\{\frac{4\hat{\mu}\beta_3}
{9\beta_1^{5/4}}\varphi_0^{3/4}\bigg\}\varphi_0^{3/4}\nonumber\\
\leq\,&\bigg(\frac{\mu_4\sigma}{\beta_1^{1/4}}
+\frac{8\hat{\mu}\mu_2\beta_3}{9\beta_1^{5/4}}\bigg)
\exp\bigg\{\frac{4\hat{\mu}\beta_3{\zeta^{3/4}}}{9\beta_1^{5/4}}\bigg\}\varphi_0^{3/4},
\end{align}
where $\mu_4$ is a positive constant depending only on $\mu_3$ and $\gamma$. {Similarly, when $k$ and $k^*$ are both odd, or $k$ is odd and $k^*$ is even, or $k$ is even and $k^*$ is odd, \eqref{1-eqvar} remains valid. Combining it with the corresponding characteristic representations and using $Z_0=-W_0$ gives the same estimate.}
Combining all four cases yields \eqref{1-eq3.22}.
 \par Next, one can get the weak solution $Q \in C^{0,1}([0,{\zeta}]\times[0,m])$ by solving the problem \eqref{1-eq3.13}--\eqref{1-eq3.14}. Thus $(W,Z,Q)\in L^{\infty}((0,{\zeta})\times(0,m))\times L^{\infty}((0,{\zeta})\times(0,m){)}\times C^{0,1}([0,{\zeta}]\times[0,m])$ is the unique weak solution to the problem \eqref{1-eq3.7}--\eqref{1-eq3.14} satisfying \eqref{1-eq3.21}, \eqref{1-eq3.22} and
\begin{align*}
-\mu_2\exp\bigg\{\frac{4\hat{\mu}\beta_3{\zeta^{3/4}}}{9\beta_1^{5/4}}\bigg\}
 \leq Q_{\varphi}(\varphi,\psi) \leq -\mu_1\exp\bigg\{-\frac{4\hat{\mu}\beta_3{\zeta^{3/4}}}{9\beta_1^{5/4}}\bigg\},\quad (\varphi,\psi)\in(0,{\zeta})\times(0,m), \\
-\mu_2\exp\bigg\{\frac{4\hat{\mu}\beta_3{\zeta^{3/4}}}{9\beta_1^{5/4}}\bigg\}\varphi
 \leq Q(\varphi,\psi) \leq -\mu_1\exp\bigg\{-\frac{4\hat{\mu}\beta_3{\zeta^{3/4}}}{9\beta_1^{5/4}}\bigg\}\varphi,\quad (\varphi,\psi)\in(0,{\zeta})\times(0,m).
\end{align*}
Furthermore, by \eqref{1-eq2.91} and \eqref{1-eqwz1}, it holds that
{\begin{align*}
|Q_\psi(\varphi,\psi)|
\leq\,&1/2b^{-1/2}(\tilde Q(\varphi,\psi))
|W(\varphi,\psi)+Z(\varphi,\psi)|\\
&\quad+1/4\int_0^\varphi
b^{-1}(\tilde Q(s,\psi))p(\tilde Q(s,\psi))
|\tilde Q_\psi(s,\psi)|
|W(s,\psi)-Z(s,\psi)|ds\\
&\quad+1/4\int_0^\varphi
b^{-3/2}(\tilde Q(s,\psi))p(\tilde Q(s,\psi))
|\tilde Q_\varphi(s,\psi)|
|W(s,\psi)+Z(s,\psi)|ds\\
\leq\,&\bigg(\frac{\mu_5\beta_2^{1/4}\sigma}{\beta_1^{1/4}}
+\frac{\mu_2\beta_2^{1/4}\beta_3}{2\beta_1^{5/4}}\bigg)
\exp\bigg\{\frac{4\hat{\mu}\beta_3{\zeta^{3/4}}}{9\beta_1^{5/4}}\bigg\}\varphi
+\frac{4\mu_2\beta_3}{15\beta_1}
\exp\bigg\{\frac{4\hat\mu\beta_3{\zeta^{3/4}}}{9\beta_1^{5/4}}\bigg\}
\varphi\\
&\quad+\bigg(
\frac{3\mu_4\beta_2}
{20\hat\mu\beta_1}\sigma
+\frac{2\mu_2\beta_2}
{15\beta_1^2}\beta_3
\bigg)
\exp\bigg\{\frac{4\hat\mu\beta_3{\zeta^{3/4}}}{9\beta_1^{5/4}}\bigg\}
\varphi\\
\leq\,&\bigg(
\mu_5\bigg(\frac{\beta_2^{1/4}}{\beta_1^{1/4}}
+\frac{4\beta_2}{15\beta_1}\bigg)\sigma+\frac{\mu_2}{2\beta_1}
\bigg(
\frac{\beta_2^{1/4}}{\beta_1^{1/4}}
+\frac{8}{15}
+\frac{4\beta_2}{15\beta_1}
\bigg)\beta_3
\bigg)
\exp\bigg\{\frac{4\hat\mu\beta_3{\zeta^{3/4}}}{9\beta_1^{5/4}}\bigg\}
\varphi,
\end{align*}
where $\mu_5\geq9\mu_4/(16\hat\mu)$ is a positive constant depending only on
$\mu_4$ and $\hat\mu$.}
Consequently,  \eqref{1-eq3.23}-\eqref{1-eq3.25} can be obtained.
\end{proof}
\subsection{Solving the nonlinear boundary value problem }\label{2-2-sonic}\noindent
\par
Based on the well-posedness of the linearized problem \eqref{1-eq3.7}--\eqref{1-eq3.14},
we are going to prove the existence of weak solutions to the nonlinear problem \eqref{1-eq2.221}--\eqref{1-eq2.225}.
According to \eqref{1-eq3.16}--\eqref{1-eq3.18} and Proposition \ref{1-pro3.1}, the problem \eqref{1-eq3.7}--\eqref{1-eq3.14} admits a unique weak solution $(W,Z,Q)\in L^{\infty}((0,{\zeta})\times(0,m))\times L^{\infty}((0,{\zeta})\times(0,m))\times C^{0,1}([0,{\zeta}]\times[0,m])$. Furthermore, $Q$ satisfies
\begin{align}
&-\mu_2\exp\bigg\{\frac{4\hat{\mu}\beta_3{\zeta^{3/4}}}{9\beta_1^{5/4}}\bigg\}
  \leq Q_{\varphi}(\varphi,\psi)
  \leq -\mu_1\exp\bigg\{-\frac{4\hat{\mu}\beta_3{\zeta^{3/4}}}{9\beta_1^{5/4}}\bigg\},
\quad &&(\varphi,\psi)\in(0,{\zeta})\times(0,m), \label{1-eq3.31} \\
&-\mu_2\exp\bigg\{\frac{4\hat{\mu}\beta_3{\zeta^{3/4}}}{9\beta_1^{5/4}}\bigg\}\varphi
  \leq Q(\varphi,\psi)
  \leq -\mu_1\exp\bigg\{-\frac{4\hat{\mu}\beta_3{\zeta^{3/4}}}{9\beta_1^{5/4}}\bigg\}\varphi,
\quad&&(\varphi,\psi)\in(0,{\zeta})\times(0,m), \label{1-eq3.32},\\
&|Q_{\psi}(\varphi,\psi)|
\leq{}
{
\bigg(
\mu_5\bigg(\frac{\beta_2^{1/4}}{\beta_1^{1/4}}
+\frac{4\beta_2}{15\beta_1}\bigg)\sigma
+\frac{\mu_2}{2\beta_1}
\bigg(
\frac{\beta_2^{1/4}}{\beta_1^{1/4}}
+\frac{8}{15}
}\nonumber\\
&{
\qquad\qquad\qquad
+\frac{4\beta_2}{15\beta_1}
\bigg)\beta_3\bigg)
\exp\bigg\{
\frac{4\hat{\mu}\beta_3{\zeta^{3/4}}}{9\beta_1^{5/4}}
\bigg\}\varphi
},
\quad
&&(\varphi,\psi)\in(0,{\zeta})\times(0,m), \label{1-eq3.33}
\end{align}
where $\mu_i$ $(i=1,2,5)$ $(\mu_1\leq\mu_2)$ and $\hat{\mu}$ from the estimates \eqref{1-eq3.21}--\eqref{1-eq3.25} are positive constants. Set
\begin{equation}\label{1-eqe1}
{\beta_1\geq \frac{3\mu_1}{4},\quad \frac{\beta_2}{\beta_1}\leq\frac{21}{20},\quad \mu_6 > \frac{4^{9/4}\hat{\mu}}{9(3\tilde\mu_1)^{5/4}}.}
\end{equation}
Then we consider the function
{$$ f_1(t)=\exp\bigg\{\frac{4^{9/4}\hat\mu t}{9(3\tilde\mu_1)^{5/4}}\bigg\}-1-\mu_6t,\ \ t>0.$$
Since $f_1(0)=0$ and $f_1'(0)<0$, there exists $N_1>0$, independent of $\sigma$, such that $f_1(t)\leq0$ for $0\leq t\leq N_1$. Using $\mu_1\geq\tilde\mu_1$ in \eqref{1-mu1}, if $\beta_3\leq N_1$, the following estimate holds:}
{\begin{align}\label{1-eqe}
1-\exp\bigg\{-\frac{4\hat{\mu}\beta_3\zeta^{3/4}}{9\beta_1^{5/4}}\bigg\}\leq \exp\bigg\{\frac{4\hat{\mu}\beta_3\zeta^{3/4}}{9\beta_1^{5/4}}\bigg\}-1\leq
\exp\bigg\{\frac{4^{9/4}\hat{\mu}\beta_3}{9(3\tilde{\mu}_1)^{5/4}}\bigg\}-1
\le \mu_6\beta_3.
\end{align}}
{
Here we used $1-e^{-t}\leq e^t-1$ for $t>0$.
We further require $\mu_2\leq\beta_2$, which will be
ensured by the choice of $\beta_2$ below. By \eqref{1-eq3.33}--\eqref{1-eqe},
\begin{align}\label{1-eqe-aa}
|Q_\psi(\varphi,\psi)|
\leq\,&\bigg(
\mu_5\bigg(
\frac{\beta_2^{1/4}}{\beta_1^{1/4}}
+\frac{4\beta_2}{15\beta_1}\bigg)\sigma
+\frac{\mu_2}{2\beta_1}
\bigg(
\frac{\beta_2^{1/4}}{\beta_1^{1/4}}
+\frac{8}{15}
+\frac{4\beta_2}{15\beta_1}\bigg)\beta_3
\bigg)
\exp\bigg\{\frac{4\hat\mu\beta_3{\zeta^{3/4}}}
{9\beta_1^{5/4}}\bigg\}\varphi\nonumber\\
\leq\,&\bigg(
\mu_5\bigg(
\bigg(\frac{21}{20}\bigg)^{1/4}
+\frac{7}{25}\bigg)\sigma
+\frac{21}{40}\bigg(
\frac{81}{80}+\frac{8}{15}
+\frac{7}{25}\bigg)\beta_3
\bigg)(1+\mu_6\beta_3)\varphi\nonumber\\
\leq\,&\bigg(
\frac43\bigg(\frac{21}{20}\bigg)^{1/4}
\mu_5\sigma+\frac{39}{40}\beta_3
\bigg)(1+\mu_6\beta_3)\varphi\nonumber\\
=\,&(\mu_7\sigma+\mu_8\beta_3)
(1+\mu_6\beta_3)\varphi,
\end{align}}
Define{
\begin{equation*}
\mu_7=\frac43\bigg(\frac{21}{20}\bigg)^{1/4}\mu_5,
\quad0<\mu_8=\frac{39}{40}<1.
\end{equation*}}
If
\begin{equation}\label{1-eqe13}
\mu_7\sigma+\mu_8\beta_3\leq\frac{1-\mu_8}{2\mu_6},
\end{equation}
it follows from \eqref{1-eq3.33}--\eqref{1-eqe13} that
\begin{align*}
\left|Q_{\psi}(\varphi,\psi)\right|
  \leq\,&(\mu_7\sigma+\mu_8\beta_3)(1+\mu_6\beta_3)\varphi\nonumber\\
  \leq\,&\bigg((\mu_7\sigma+\mu_8\beta_3)+\mu_6(\mu_7\sigma+\mu_8\beta_3)\beta_3\bigg)\varphi\nonumber\\
  \le\,&\bigg(\mu_7\sigma+\frac{(1+\mu_8)\beta_3}{2}\bigg)\varphi.
\end{align*}
\par {Next, using $\mu_0$ defined in \eqref{1-eq2.41}}, we  choose
{\begin{align}\label{1-beta*}
\beta_1=\mu_1(1-\mu_6\beta_3),\quad
\beta_2=\mu_2(1+\mu_6\beta_3),\quad\beta_3=\mu_9\sigma \ \ {\rm{with}} \ \ \mu_9=\frac{2\mu_7}{1-\mu_8},
\end{align}
\begin{align}\label{1-eq3.34}
\tilde{\beta}_1=\mu_0(R_0+\mu_6\mu_9),
\quad\tilde{\beta}_2= \mu_0(R_0+2\mu_6\mu_9),
\quad\tilde{\beta}_3=\mu_9,
\end{align}
and
\begin{align}\label{1-eq3.35}
\sigma_1=\min\left\{
\frac{1}{100(\mu_6\mu_9+R_0)},\frac{N_1}{\mu_9},\frac{1-\mu_8}{2\mu_6(\mu_7+\mu_8\mu_9)},\frac{\gamma - 1}{20 R_0 (\gamma + 1)}\right\}.
\end{align}
{For $0<\sigma\leq\sigma_1$, it follows from
\eqref{1-mu1}, \eqref{1-beta*}--\eqref{1-eq3.35} that
\begin{align*}
\beta_2
=\mu_0(1+R_0\sigma)(1+\mu_6\mu_9\sigma)\leq\mu_0+\tilde\beta_2\sigma
\leq\frac{51}{50}\mu_0<2\mu_0.
\end{align*}
Moreover, \eqref{1-eq3.31}--\eqref{1-eq3.35} yields}
\begin{align}
-\beta_2\varphi \leq &Q(\varphi,\psi)\leq -\beta_1\varphi,
\quad(\varphi,\psi)\in(0,\zeta)\times(0,m),\label{1-qq1}\\
-\beta_2\leq Q_{\varphi}(\varphi,\psi)&\leq -\beta_1,
\quad \left|Q_{\psi}(\varphi,\psi)\right|\leq\beta_3\varphi,
\quad(\varphi,\psi)\in(0,\zeta)\times(0,m),\label{1-qq2}
\end{align}
{Consequently, using the above relations between
$\beta_i$ and $\tilde\beta_i$ $(i=1,2,3)$, we obtain}
\begin{align}
-\tilde{\beta}_2\sigma\varphi
\leq Q(\varphi,\psi)+&\mu_0\varphi
\leq \tilde{\beta}_1\sigma\varphi,
\quad(\varphi,\psi)\in(0,\zeta)\times(0,m),\label{1-eq3.36}\\
-\tilde{\beta}_2\sigma
\leq Q_{\varphi}(\varphi,\psi)+\mu_0&
\leq \tilde{\beta}_1\sigma,
\quad \left|Q_{\psi}(\varphi,\psi)\right|
\leq\tilde{\beta}_3\sigma\varphi,
\quad(\varphi,\psi)\in(0,\zeta)\times(0,m).\label{1-eq3.38}
\end{align}
Define
\begin{align*}
{\mathscr{S}_1 = \left\{\tilde{Q}\in C^{0,1}([0,{\zeta}]\times[0,m]):
\tilde{Q}\text{ satisfies }\eqref{1-eq3.16}\text{--}\eqref{1-eq3.17}
\text{ with }\eqref{1-beta*}\right\}.}
\end{align*}
with the norm
$$
\|\tilde{Q}\|_{{\mathscr{S}}_1}=\|\tilde{Q}\|_{L^{\infty}((0,{\zeta})\times(0,m))},\quad\tilde{Q}\in {\mathscr{S}}_1.
$$
Based on the above analysis, we define a mapping $J$ from $\mathscr{S}_1$ to itself as follows:
\begin{align}\label{1-eq3.43}
J(\tilde{Q})=Q,\quad\tilde{Q}\in\mathscr{S}_1.
\end{align}
Then we have the following theorem.
\begin{theorem}
\label{1-th3.1}
Given  $g\in C^4([-R_0\sin\theta_0,R_0\sin\theta_0])$ satisfying \eqref{1-eq2.1} and \eqref{1-eq2.43}. Then if $\sigma\leq\sigma_1$  with $\sigma_1$ given by \eqref{1-eq3.35}, the problem \eqref{1-eq2.221}--\eqref{1-eq2.225} admits at least one weak solution $Q\in C^{0,1}([0,{\zeta}]\times[0,m])$ satisfying
\begin{align}\label{1-eq3.44}
-\tilde{\beta}_2\sigma\varphi \leq Q(\varphi,\psi)+&{\mu_0\varphi}\leq \tilde{\beta}_1\sigma\varphi,\quad&&(\varphi,\psi)\in(0,{\zeta})\times(0,m),\nonumber\\
-\tilde{\beta}_2\sigma\leq Q_{\varphi}(\varphi,\psi)+{{\mu_0}}&\leq \tilde{\beta}_1\sigma,\quad \left|Q_{\psi}(\varphi,\psi)\right|\leq\tilde{\beta}_3\sigma\varphi,\quad
&&(\varphi,\psi)\in(0,{\zeta})\times(0,m),
\end{align}
where $\tilde{\beta} _i$ $(i= 1,2,3)$ are positive constants given by \eqref{1-eq3.34}.
\end{theorem}

\begingroup
\begin{proof}
As mentioned above, the mapping $J$ defined by
\eqref{1-eq3.43} maps $\mathscr S_1$ into itself.
The set $\mathscr S_1$ is nonempty, closed and convex
in $C([0,\zeta]\times[0,m])$.
In view of \eqref{1-eq3.17*}, the set
$J(\mathscr S_1)$ is uniformly bounded in
$C^{0,1}([0,\zeta]\times[0,m])$.
Since $C^{0,1}([0,\zeta]\times[0,m])$ is
compactly embedded into
$C([0,\zeta]\times[0,m])$,
it follows that $J$ is compact.
Therefore, by \eqref{1-eq3.36},
\eqref{1-eq3.38}, and the Schauder fixed
point theorem, it remains only to prove that
$J$ is continuous, which then yields
Theorem \ref{1-th3.1}.

\par Assume that
$\{\tilde Q^{(n)}\}_{n=0}^{\infty}
\subset\mathscr S_1$
satisfies
\begin{align}\label{1-eq3.45}
\lim_{n\to\infty}
\|\tilde Q^{(n)}-\tilde Q^{(0)}\|_
{L^\infty((0,\zeta)\times(0,m))}=0.
\end{align}
Furthermore, the definition of $J$ gives that
\[
J(\tilde Q^{(n)})=Q^{(n)}.
\]
Let
$(W^{(n)},Z^{(n)},Q^{(n)})$
be the unique weak solution of the following
problem:
\begin{align}
&W^{(n)}_\varphi
+b^{1/2}(\tilde Q^{(n)})W^{(n)}_\psi
=-\frac12b^{-1/2}(\tilde Q^{(n)})
p(\tilde Q^{(n)})\tilde Q^{(n)}_\psi W^{(n)},
&&(\varphi,\psi)\in(0,\zeta)\times(0,m),
\label{1-eq3.47}\\
&Z^{(n)}_\varphi
-b^{1/2}(\tilde Q^{(n)})Z^{(n)}_\psi
=\frac12b^{-1/2}(\tilde Q^{(n)})
p(\tilde Q^{(n)})\tilde Q^{(n)}_\psi Z^{(n)},
&&(\varphi,\psi)\in(0,\zeta)\times(0,m),
\label{1-eq3.48}\\
&W^{(n)}(0,\psi)=W_0(\psi),
&&\psi\in(0,m),
\label{1-eq3.49-a}\\
&Z^{(n)}(0,\psi)=Z_0(\psi),
&&\psi\in(0,m),
\label{1-eq3.49}\\
&W^{(n)}(\varphi,0)+Z^{(n)}(\varphi,0)=0,
&&\varphi\in(0,\zeta),
\label{1-eq3.50}\\
&W^{(n)}(\varphi,m)+Z^{(n)}(\varphi,m)=0,
&&\varphi\in(0,\zeta),
\label{1-eq3.51}\\
&Q^{(n)}_\varphi
=\frac12(W^{(n)}-Z^{(n)}),
&&(\varphi,\psi)\in(0,\zeta)\times(0,m),
\label{1-eq3.51aa}\\
&Q^{(n)}(0,\psi)=0,
&&\psi\in(0,m).
\label{1-eq3.51bb}
\end{align}
To prove continuity, it suffices to show
\begin{align}\label{1-eq3.52}
\lim_{n\to\infty}
\|Q^{(n)}-Q^{(0)}\|_
{L^\infty((0,\zeta)\times(0,m))}=0.
\end{align}
We will prove \eqref{1-eq3.52} by contradiction.
Otherwise, there exist a positive constant
$\varepsilon_0$ and a subsequence of
$\{Q^{(n)}\}_{n=1}^{\infty}$, still denoted
by itself, such that
\begin{align}\label{1-eq3.53}
\|Q^{(n)}-Q^{(0)}\|_
{L^\infty((0,\zeta)\times(0,m))}
\geq\varepsilon_0,
\qquad n=1,2,\ldots.
\end{align}
By Proposition \ref{1-pro3.1}, $Q^{(n)}$ satisfies
\eqref{1-qq1}--\eqref{1-qq2} for each $n\geq1$.
The Arzel$\mathrm{\grave{a}}$-Ascoli theorem and weak-star compactness
then yield a  subsequence, still denoted by itself,
and functions $W^*,Z^*,Q^*$ such that
\begin{align}
&Q^{(n)}\longrightarrow Q^*
\quad\text{in }
L^\infty((0,\zeta)\times(0,m))
\quad\text{as }n\to\infty,
\label{1-eq3.54}\\
&Q_\varphi^{(n)}
\rightharpoonup^* Q_\varphi^*,
\qquad
Q_\psi^{(n)}
\rightharpoonup^* Q_\psi^*
\quad\text{in }
L^\infty((0,\zeta)\times(0,m))
\quad\text{as }n\to\infty,
\label{1-eq3.55}\\
&W^{(n)}\rightharpoonup^*W^*,
\qquad
Z^{(n)}\rightharpoonup^*Z^*
\quad\text{in }
L^\infty((0,\zeta)\times(0,m))
\quad\text{as }n\to\infty.
\label{1-eq3.55-aa}
\end{align}
By \eqref{1-eq2.91} and \eqref{1-eq3.45},
\begin{align*}
\int_0^\zeta
&\|b^{1/2}(\tilde Q^{(n)})(s,\cdot)
-b^{1/2}(\tilde Q^{(0)})(s,\cdot)
\|_{L^\infty(0,m)}\,ds
\longrightarrow0.
\end{align*}
This convergence is uniform away
from $s=0$, while the coefficient bounds
near $s=0$ are integrable.
For smooth test functions $\eta_\pm$
satisfying
$\eta_\pm(\zeta,\cdot)=0$ and
$\eta_++\eta_-=0$ on $\psi=0 $ and $\psi=m$,
\eqref{1-eq3.47}--\eqref{1-eq3.51}
yield
\begin{align*}
&\int_0^\zeta\int_0^m
\bigg(
W^{(n)}
\bigg(\eta_{+,\varphi}
+b^{1/2}(\tilde Q^{(n)})
\eta_{+,\psi}\bigg)
+Z^{(n)}
\bigg(\eta_{-,\varphi}
-b^{1/2}(\tilde Q^{(n)})
\eta_{-,\psi}\bigg)
\bigg)d\psi d\varphi\\
&\quad+\int_0^m
\bigg(
W_0(\psi)\eta_+(0,\psi)
+Z_0(\psi)\eta_-(0,\psi)
\bigg)d\psi=0.
\end{align*}
Passing to the limit as $n\to\infty$
in the above weak formulation and using
\eqref{1-eq3.45} and
\eqref{1-eq3.54}--\eqref{1-eq3.55-aa},
we obtain the following problem:
\begin{align*}
&W^*_\varphi
+b^{1/2}(\tilde Q^{(0)})W^*_\psi
=-\frac12b^{-1/2}(\tilde Q^{(0)})
p(\tilde Q^{(0)})
\tilde Q^{(0)}_\psi W^*,
&&(\varphi,\psi)\in(0,\zeta)\times(0,m),\\
&Z^*_\varphi
-b^{1/2}(\tilde Q^{(0)})Z^*_\psi
=\frac12b^{-1/2}(\tilde Q^{(0)})
p(\tilde Q^{(0)})
\tilde Q^{(0)}_\psi Z^*,
&&(\varphi,\psi)\in(0,\zeta)\times(0,m),\\
&W^*(0,\psi)=W_0(\psi),
&&\psi\in(0,m),\\
&Z^*(0,\psi)=Z_0(\psi),
&&\psi\in(0,m),\\
&W^*(\varphi,0)+Z^*(\varphi,0)=0,
&&\varphi\in(0,\zeta),\\
&W^*(\varphi,m)+Z^*(\varphi,m)=0,
&&\varphi\in(0,\zeta),\\
&Q^*_\varphi=\frac12(W^*-Z^*),
&&(\varphi,\psi)\in(0,\zeta)\times(0,m),\\
&Q^*(0,\psi)=0,
&&\psi\in(0,m).
\end{align*}
The boundary conditions on the entrance and nozzle wall  for
$W^*$ and $Z^*$  hold in the sense of the limiting weak formulation. Furthermore, the entrance
condition for $Q^*$ follows from
\eqref{1-eq3.54}.
By uniqueness in Proposition
\ref{1-pro3.1},
$
(W^*,Z^*,Q^*)=(W^{(0)},Z^{(0)},Q^{(0)}).
$
This contradicts \eqref{1-eq3.53}
and \eqref{1-eq3.54}.
Thus \eqref{1-eq3.52} holds,
and $J$ is continuous. Then The Schauder fixed point theorem gives
$Q=J(Q)\in\mathscr S_1$.
By \eqref{1-eq3.15-a} in Remark \ref{1-rem3.1},
\begin{align*}
Q_{\varphi\varphi}
=\frac{1}{2}(W_\varphi-Z_\varphi)
=-\frac{1}{2}\bigg(b^{1/2}(Q)(W+Z)\bigg)_\psi
=(b(Q)Q_\psi)_\psi.
\end{align*}
The entrance and wall conditions follow from
\eqref{1-eq3.9}--\eqref{1-eq3.14}.
Finally, \eqref{1-qq1}--\eqref{1-qq2} imply
\eqref{1-eq3.44}, which proves the theorem.
\end{proof}
\endgroup
\subsection{Continuous sonic-supersonic flows }\label{2-3-sonic}\noindent
\par
In this subsection, we establish the existence of $C^{1,1}$ sonic-supersonic flows in $[0,{\zeta}]$. Assume that
\begin{equation}\label{1-eq3.61}
  |\tilde{Q}_{\varphi\varphi}(\varphi,\psi)|\leq \beta_4\sigma\varphi^{1/2},\ \
  |\tilde{Q}_{\varphi\psi}(\varphi,\psi)|\leq \beta_5\sigma,\ \
  |\tilde{Q}_{\psi\psi}(\varphi,\psi)|\leq \beta_6\sigma\varphi,\ \ (\varphi,\psi)\in(0,{\zeta})\times(0,m)
\end{equation} for some positive constants $\beta_i$ $(i=4,5,6)$. Then the following proposition holds.
\begin{proposition}\label{1-pro3.2}
Let the assumptions of Proposition \ref{1-pro3.1} hold.
Assume \eqref{1-eq3.61}  and
$$
\tilde Q_\psi(\varphi,0)=\tilde Q_\psi(\varphi,m)=0,\qquad0<\varphi\leq \zeta
$$
hold.
{Let} $(W,Z,Q)$ be the unique weak solution to the problem \eqref{1-eq3.7}--\eqref{1-eq3.14}, then $(W,Z,Q)\in L^{\infty}((0,{\zeta})\times(0,m))\times L^{\infty}((0,{\zeta})\times(0,m))\times C^{1,1}([0,{\zeta}]\times[0,m])$ satisfies
\begin{align}
&\left|b^{-1/2}(\tilde{Q})(\varphi,\psi)W_{\varphi}(\varphi,\psi)\right|\leq
\left(M_2\beta_6+M_3\right)\sigma,\ \ &&(\varphi,\psi)\in(0,{\zeta})\times(0,m),\label{1-eq3.62}\\
&\left|b^{-1/2}(\tilde{Q})(\varphi,\psi)Z_{\varphi}(\varphi,\psi)\right|\leq
\left(M_2\beta_6+M_3\right)\sigma,\ \ &&(\varphi,\psi)\in(0,{\zeta})\times(0,m),\label{1-eq3.621}\\
&\left|b^{-1/2}(\tilde{Q})(\varphi,\psi)\left(W_{\varphi}(\varphi,\psi)-Z_{\varphi}(\varphi,\psi)\right)\right|
\leq\bigg(M_4\beta_6+M_5\bigg)\sigma\varphi^{3/4},\ \ &&(\varphi,\psi)\in(0,{\zeta})\times(0,m),\label{1-eq3.63}\\
&\left|Q_{\varphi\psi}(\varphi,\psi)\right|\leq
\left(M_6\beta_6+M_7\right)\sigma,\ \ &&(\varphi,\psi)\in(0,{\zeta})\times(0,m),\label{1-eq3.66}\\
&\left|Q_{\varphi\varphi}(\varphi,\psi)\right|\leq
\bigg(M_8\beta_6+M_9\bigg)\sigma\varphi^{1/2},\ \ &&(\varphi,\psi)\in(0,{\zeta})\times(0,m),\label{1-eq3.64}\\
&\left|Q_{\psi\psi}(\varphi,\psi)\right|\leq
\bigg(M_{10}\beta_6+M_{11}\bigg)\sigma\varphi,\ \ &&(\varphi,\psi)\in(0,{\zeta})\times(0,m),\label{1-eq3.65}
\end{align}
where the positive constants  $M_i$ $(i=2,\cdots,5,10)$ depend only on $\sigma$, $R_0$, $\gamma$  and $M_j$ $(j=6,\cdots,9,11)$  depend only on $R_0$ and $\gamma$.
\end{proposition}

\begin{proof}
 Set $$
 (U_s,V_s)(\varphi,\psi)=\left(b^{-1/2}(\tilde{Q})W_\varphi,b^{-1/2}(\tilde{Q})Z_\varphi\right)(\varphi,\psi),\quad (\varphi,\psi)\in (0,{\zeta})\times(0,m).
 $$ Then $(U_s,V_s)$ solves the following problem:
\begin{align}
 &\partial_\varphi U_s + b^{1/2}(\tilde{Q}) \partial_\psi U_s=-R_s U_s-W\partial_\varphi F_s,\quad&&(\varphi,\psi)\in (0,{\zeta})\times(0,m),\label{1-eq3.67}\\
 & \partial_\varphi V_s - b^{1/2}(\tilde{Q}) \partial_\psi V_s =R_s V_s + Z \partial_\varphi F_s,\quad&&(\varphi,\psi)\in (0,{\zeta})\times(0,m),\label{1-eq3.68}\\
 &{U_s}(0,\psi)={U_{s,0}}(\psi),\quad&&\psi\in (0,m),\label{1-eq3.69}\\
 &{V_s}(0,\psi)={V_{s,0}}(\psi),\quad&&\psi\in (0,m),\label{1-eq3.70}\\
 &U_s(\varphi,0)+{V_s}(\varphi,0)=0,\quad&&\varphi\in (0,{\zeta}),\label{1-eq3.71}\\
  &{U_s}(\varphi,m)+{V_s}(\varphi,m)=0,\quad&&\varphi\in (0,{\zeta}),\label{1-eq3.72}
\end{align}
where
\begin{align*}
  &R_s=b^{-1/2}(\tilde{Q}(\varphi,\psi))p(\tilde{Q}(\varphi,\psi))\tilde{Q}_\psi(\varphi,\psi), \\
 &F_s=\frac12b^{-1}(\tilde{Q}(\varphi,\psi))p(\tilde{Q}(\varphi,\psi))\tilde{Q}_\psi(\varphi,\psi), \\
&{U_{s,0}=-F_s(0,\psi)W_0(\psi)-W_{0,\psi}(\psi)},\\
&W_{\psi}(0,\psi)=-\frac{1}{c_*^{2+4/(\gamma-1)}}\left(\sin\Theta(y)\right)''\cos\Theta(y),\\
&{V_{s,0}=F_s(0,\psi)Z_0(\psi)+Z_{0,\psi}(\psi)},\\
&Z_{\psi}(0,\psi)=\frac{1}{c_*^{2+4/(\gamma-1)}}\left(\sin\Theta(y)\right)''\cos\Theta(y).
\end{align*}
It follows from \eqref{1-eq2.91}, \eqref{1-eq2.9}, \eqref{1-eq2.36},
 \eqref{1-eq3.16} and \eqref{1-eq3.17} that
\begin{align}
    &|R_s(\varphi,\psi)|\leq \frac{\mu_{10}\beta_3}{\beta_1^{5/4}}\varphi^{-1/4}, \ \
    |F_s(\varphi,\psi)|\leq \frac{\mu_{10}\beta_3}{\beta_1},\ \ (\varphi,\psi)\in (0,{\zeta})\times(0,m), \label{1-eqr}\\
    &|W_{\psi}(0,\psi),Z_{\psi}(0,\psi)|\leq\mu_{10}\sigma,\ \
    |{U_{s,0}}(\psi),{V_{s,0}}(\psi)|\leq\frac{\mu_{10}(\beta_2\beta_3+\beta_1\sigma)}{\beta_1},\ \ \psi\in (0,m),\label{1-equ}
\end{align}
where  ${\beta} _i$ $(i= 1,2,3)$ are positive constants given by  \eqref{1-beta*}  and  $\mu_{10}$ is a positive constant depending only on $\hat{\mu}$ and $\gamma$. We first estimate $|U_s|$ and $|V_s|$. Assume that
\begin{align*}
\Sigma_\pm&: \Psi_\pm^{\prime}=\pm b^{1/2}(\tilde{Q}(\varphi,\Psi_\pm(\varphi))),  \\
0&<\Psi_\pm<m,\quad\hat{\varphi}_\pm<\varphi<\check{\varphi}_\pm\quad(0\leq\hat{\varphi}_\pm<\check{\varphi}_\pm\leq {\zeta})
\end{align*}
are positive and negative characteristics. On $\Sigma_+$, if $\hat{\varphi}_+\leq\varphi\leq\check{\varphi}_+$, $U_s$ satisfies
\begin{align*}
\frac{d}{d\varphi}U_s(\varphi,\Psi_+(\varphi))=&
-R_s(\varphi,\Psi_+(\varphi))U_s(\varphi,\Psi_+(\varphi))-W(\varphi,\Psi_+(\varphi))\partial_\varphi F_s(\varphi,\Psi_+(\varphi))\\
=&-R_s(\varphi,\Psi_+(\varphi))U_s(\varphi,\Psi_+(\varphi))-W(\varphi,\Psi_+(\varphi))\frac{d}{d\varphi}F_s(\varphi,\Psi_+(\varphi))\\
&\quad+b^{1/2}(\tilde{Q}(\varphi,\Psi_+(\varphi)))W(\varphi,\Psi_+(\varphi))\partial_\psi F_s(\varphi,\Psi_+(\varphi)),\\
\end{align*}
which is equivalent to
\begin{align*}
&\bigg(U_s(\varphi,\Psi_+(\varphi))\exp\bigg\{\int_{\hat{\varphi}_+}^{\varphi}R_s(s,\Psi_+(s))ds\bigg\}\bigg)'\\
=\,&\bigg(-W(\varphi,\Psi_+(\varphi))\frac{d}{d\varphi}F_s(\varphi,\Psi_+(\varphi))
+b^{1/2}(\tilde{Q}(\varphi,\Psi_+(\varphi)))W(\varphi,\Psi_+(\varphi))
\partial_\psi F_s(\varphi,\Psi_+(\varphi))\bigg)
\\
&\quad\times\exp\bigg\{\int_{\hat{\varphi}_+}^{\varphi}R_s(s,\Psi_+(s))ds\bigg\}.
\end{align*}
Then for any $\hat{\varphi}_+\leq\varphi\leq\check{\varphi}_+$, it holds that
\begin{align}\label{1-eq3.73}
&U_s(\check{\varphi}_+,\Psi_+(\check{\varphi}_+))\exp\bigg\{\int_{\hat{\varphi}_+}^{\check{\varphi}_+}R_s(s,\Psi_+(s))ds\bigg\}
- U_s(\hat{\varphi}_+,\Psi_+(\hat{\varphi}_+))
\nonumber\\
=\,& -W(\check{\varphi}_+,\Psi_+(\check{\varphi}_+))F_s(\check{\varphi}_+,\Psi_+(\check{\varphi}_+))
\exp\bigg\{\int_{\hat{\varphi}_+}^{\check{\varphi}_+}R_s(s,\Psi_+(s))ds\bigg\}\nonumber\\
&\quad
+ W(\hat{\varphi}_+,\Psi_+(\hat{\varphi}_+))F_s(\hat{\varphi}_+,\Psi_+(\hat{\varphi}_+))
+\int_{\hat{\varphi}_+}^{\check{\varphi}_+}I_1\exp\bigg\{\int_{\hat{\varphi}_+}^{\varphi}R_s(s,\Psi_+(s))ds\bigg\}d\varphi,
\end{align}
where
\begin{align*}
I_1=W(\varphi,\Psi_+(\varphi))\bigg(\frac12R_s(\varphi,\Psi_+(\varphi))F_s(\varphi,\Psi_+(\varphi))+b^{1/2}(\tilde{Q}(\varphi,\Psi_+(\varphi)))\partial_\psi F_s(\varphi,\Psi_+(\varphi))\bigg).
\end{align*}
{Similarly, on $\Sigma_-$, if $\hat{\varphi}_-\leq\varphi\leq\check{\varphi}_-$,  $V_s$ satisfies}
\begin{align*}
\frac{d}{d\varphi}V_s(\varphi,\Psi_-(\varphi))=&
R_s(\varphi,\Psi_-(\varphi))V_s(\varphi,\Psi_-(\varphi))+Z(\varphi,\Psi_-(\varphi))\partial_\varphi F_s(\varphi,\Psi_-(\varphi))\\
=&R_s(\varphi,\Psi_-(\varphi))V_s(\varphi,\Psi_-(\varphi))+Z(\varphi,\Psi_-(\varphi))\frac{d}{d\varphi}F_s(\varphi,\Psi_-(\varphi))\\
&\quad+b^{1/2}(\tilde{Q}(\varphi,\Psi_-(\varphi)))Z(\varphi,\Psi_-(\varphi))\partial_\psi F_s(\varphi,\Psi_-(\varphi)),\\
\end{align*}
which is simplified to
\begin{align*}
&\bigg(V_s(\varphi,\Psi_-(\varphi))\exp\bigg\{\int_{\hat{\varphi}_-}^{\varphi}-R_s(s,\Psi_-(s))ds\bigg\}\bigg)'\\
=\,&\bigg(Z(\varphi,\Psi_-(\varphi))\frac{d}{d\varphi}F_s(\varphi,\Psi_-(\varphi))+b^{1/2}(\tilde{Q}(\varphi,\Psi_-(\varphi)))
Z(\varphi,\Psi_-(\varphi))\partial_\psi F_s(\varphi,\Psi_-(\varphi))\bigg)\\
&\quad\times\exp\bigg\{\int_{\hat{\varphi}_-}^{\varphi}-R_s(s,\Psi_-(s))ds\bigg\}.
\end{align*}
Then for any $\hat{\varphi}_-\leq\varphi\leq\check{\varphi}_-$, one has
\begin{align}\label{1-eq3.74}
&V_s(\check{\varphi}_-,\Psi_-(\check{\varphi}_-))\exp\bigg\{\int_{\hat{\varphi}_-}^{\check{\varphi}_-}-R_s(s,\Psi_-(s))ds\bigg\}-V_s(\hat{\varphi}_-,\Psi_-(\hat{\varphi}_-))\nonumber\\
=\,&Z(\check{\varphi}_-,\Psi_-(\check{\varphi}_-))F_s(\check{\varphi}_-,\Psi_-(\check{\varphi}_-))\exp\bigg\{\int_{\hat{\varphi}_-}^{\check{\varphi}_-}-R_s(s,\Psi_-(s))ds\bigg\}
\nonumber\\
&\quad-Z(\hat{\varphi}_-,\Psi_-(\hat{\varphi}_-))F_s(\hat{\varphi}_-,\Psi_-(\hat{\varphi}_-))
+\int_{\hat{\varphi}_-}^{\check{\varphi}_-}I_2\exp\bigg\{\int_{\hat{\varphi}_-}^{\varphi}-R_s(s,\Psi_-(s))ds\bigg\}d\varphi,
\end{align}
where
\begin{align*}
I_2=Z(\varphi,\Psi_-(\varphi))\bigg(\frac12R_s(\varphi,\Psi_-(\varphi))F_s(\varphi,\Psi_-(\varphi))+b^{1/2}(\tilde{Q}(\varphi,\Psi_-(\varphi)))\partial_\psi F_s(\varphi,\Psi_-(\varphi))\bigg).
\end{align*}
It follows from \eqref{1-eq2.91}, \eqref{1-eq2.9}, \eqref{1-eq3.16}, \eqref{1-eq3.17}, \eqref{1-eq3.61} and \eqref{1-eqr} that
\begin{align*}
\left|\frac12R_sF_s+b^{1/2}(\tilde{Q})\partial_\psi F_s\right|
=\,&\left|\frac{1}{2}R_sF_s+b^{1/2}(\tilde{Q})\left(-\frac{1}{2}b^{-2}(\tilde{Q})p^2(\tilde{Q})\tilde{Q}_{\psi}^2\right.\right.\\
&\quad+\left.\left.\frac{1}{2}b^{-1}(\tilde{Q})p'(\tilde{Q})\tilde{Q}_{\psi}^2+\frac{1}{2}b^{-1}(\tilde{Q})p(\tilde{Q})\tilde{Q}_{\psi\psi}\right)\right|\\
=\,&\left|-\frac{1}{2}R_sF_s+\frac{1}{2}b^{-1/2}p'\tilde{Q}_{\psi}^2+\frac{1}{2}b^{-1/2}p\tilde{Q}_{\psi\psi}\right|\\
\leq\,&\bigg(\frac{\mu_{11}\beta_3^2}{\beta_1^{9/4}}+\frac{3\hat{\mu}\beta_6\sigma}{8\beta_1^{5/4}}\bigg)\varphi^{-1/4},
\end{align*}
where $\mu_{11}$ is a positive constant depending only on $\mu_{10}$ and $\hat{\mu}$. Combining the above inequality with \eqref{1-eq3.21}, \eqref{1-eqe} and \eqref{1-beta*}, we obtain that
\begin{align}\label{1-M_1}
\left|I_1,I_2\right|\leq \bigg(\frac{\mu_{11}\beta_2\beta_3^2}{\beta_1^{9/4}}
+\frac{3\hat{\mu}\beta_2\beta_6\sigma}{8\beta_1^{5/4}}\bigg)
\varphi^{-1/4}:=M_1\varphi^{-1/4}.
\end{align}
 Thus it follows from \eqref{1-eq3.73} and \eqref{1-eq3.74} that
\begin{align*}
&U_s(\check{\varphi}_+,\Psi_+(\check{\varphi}_+))\exp\bigg\{\int_{\hat{\varphi}_+}^{\check{\varphi}_+}R_s(s,\Psi_+(s))ds\bigg\}
-U_s(\hat{\varphi}_+,\Psi_+(\hat{\varphi}_+))
\nonumber\\
\leq\,&-W(\check{\varphi}_+,\Psi_+(\check{\varphi}_+))F_s(\check{\varphi}_+,\Psi_+(\check{\varphi}_+))
\exp\bigg\{\int_{\hat{\varphi}_+}^{\check{\varphi}_+}R_s(s,\Psi_+(s))ds\bigg\}\nonumber\\
&\quad
+W(\hat{\varphi}_+,\Psi_+(\hat{\varphi}_+))F_s(\hat{\varphi}_+,\Psi_+(\hat{\varphi}_+))
+M_1\int_{\hat{\varphi}_+}^{\check{\varphi}_+}\varphi^{-1/4}\exp\bigg\{\int_{\hat{\varphi}_+}^{\varphi}R_s(s,\Psi_+(s))ds\bigg\}d\varphi
\end{align*}
and
\begin{align*}
&V_s(\check{\varphi}_-,\Psi_-(\check{\varphi}_-))\exp\bigg\{\int_{\hat{\varphi}_-}^{\check{\varphi}_-}-R_s(s,\Psi_-(s))ds\bigg\}-V_s(\hat{\varphi}_-,\Psi_-(\hat{\varphi}_-))\nonumber\\
\geq\,&Z(\check{\varphi}_-,\Psi_-(\check{\varphi}_-))F_s(\check{\varphi}_-,\Psi_-(\check{\varphi}_-))
\exp\bigg\{\int_{\hat{\varphi}_-}^{\check{\varphi}_-}-R_s(s,\Psi_-(s))ds\bigg\}\nonumber\\
&\quad
-Z(\hat{\varphi}_-,\Psi_-(\hat{\varphi}_-))F_s(\hat{\varphi}_-,\Psi_-(\hat{\varphi}_-))
-M_1\int_{\hat{\varphi}_-}^{\check{\varphi}_-}\varphi^{-1/4}\exp\bigg\{\int_{\hat{\varphi}_-}^{\varphi}-R_s(s,\Psi_-(s))ds\bigg\}d\varphi.
\end{align*}
That is
\begin{align}\label{1-eq3.75}
&U_s(\check{\varphi}_+,\Psi_+(\check{\varphi}_+))-M_1\int_{0}^{\check{\varphi}_+}\varphi^{-1/4}\exp\bigg\{\int_{\varphi}^{\check{\varphi}_+}-R_s(s,\Psi_+(s))ds\bigg\}d\varphi \nonumber\\
\leq\,&\bigg(U_s(\hat{\varphi}_+,\Psi_+(\hat{\varphi}_+))-M_1\int_{0}^{\hat{\varphi}_+}\varphi^{-1/4}\exp\bigg\{\int_{\varphi}^{\hat{\varphi}_+}-R_s(s,\Psi_+(s))ds\bigg\}d\varphi\bigg)\exp\bigg\{\int_{\hat{\varphi}_+}^{\check{\varphi}_+}-R_s(s,\Psi_+(s))ds\bigg\}\nonumber\\
&\quad+W(\hat{\varphi}_+,\Psi_+(\hat{\varphi}_+))F_s(\hat{\varphi}_+,\Psi_+(\hat{\varphi}_+))\exp\bigg\{\int_{\hat{\varphi}_+}^{\check{\varphi}_+}-R_s(s,\Psi_+(s))ds\bigg\} \nonumber\\
&\quad-W(\check{\varphi}_+,\Psi_+(\check{\varphi}_+))F_s(\check{\varphi}_+,\Psi_+(\check{\varphi}_+))
\end{align}
and
\begin{align}\label{1-eq3.76}
&-V_s(\check{\varphi}_-,\Psi_-(\check{\varphi}_-))-M_1\int_{0}^{\check{\varphi}_-}\varphi^{-1/4}\exp\bigg\{\int_{\varphi}^{\check{\varphi}_-}R_s(s,\Psi_-(s))ds\bigg\}d\varphi \nonumber\\
\leq\,&\bigg(-V_s(\hat{\varphi}_-,\Psi_-(\hat{\varphi}_-))-M_1\int_{0}^{\hat{\varphi}_-}\varphi^{-1/4}\exp\bigg\{\int_{\varphi}^{\hat{\varphi}_-}R_s(s,\Psi_-(s))ds\bigg\}d\varphi\bigg)\exp\bigg\{\int_{\hat{\varphi}_-}^{\check{\varphi}_-}R_s(s,\Psi_-(s))ds\bigg\}\nonumber\\
&\quad+Z(\hat{\varphi}_-,\Psi_-(\hat{\varphi}_-))F_s(\hat{\varphi}_-,\Psi_-(\hat{\varphi}_-))\exp\bigg\{\int_{\hat{\varphi}_-}^{\check{\varphi}_-}R_s(s,\Psi_-(s))ds\bigg\} \nonumber\\
&\quad-Z(\check{\varphi}_-,\Psi_-(\check{\varphi}_-))F_s(\check{\varphi}_-,\Psi_-(\check{\varphi}_-)).
\end{align}
{Similarly, it follows from \eqref{1-eq3.73} and \eqref{1-eq3.74} that
\begin{align*}
&-U_s(\check{\varphi}_+,\Psi_+(\check{\varphi}_+))
\exp\bigg\{\int_{\hat{\varphi}_+}^{\check{\varphi}_+}
R_s(s,\Psi_+(s))ds\bigg\}
+U_s(\hat{\varphi}_+,\Psi_+(\hat{\varphi}_+))
\nonumber\\
\leq\,&W(\check{\varphi}_+,\Psi_+(\check{\varphi}_+))
F_s(\check{\varphi}_+,\Psi_+(\check{\varphi}_+))
\exp\bigg\{\int_{\hat{\varphi}_+}^{\check{\varphi}_+}
R_s(s,\Psi_+(s))ds\bigg\}\nonumber\\
&\quad
-W(\hat{\varphi}_+,\Psi_+(\hat{\varphi}_+))
F_s(\hat{\varphi}_+,\Psi_+(\hat{\varphi}_+))
+M_1\int_{\hat{\varphi}_+}^{\check{\varphi}_+}
\varphi^{-1/4}
\exp\bigg\{\int_{\hat{\varphi}_+}^{\varphi}
R_s(s,\Psi_+(s))ds\bigg\}d\varphi
\end{align*}
and
\begin{align*}
&V_s(\check{\varphi}_-,\Psi_-(\check{\varphi}_-))
\exp\bigg\{\int_{\hat{\varphi}_-}^{\check{\varphi}_-}
-R_s(s,\Psi_-(s))ds\bigg\}
-V_s(\hat{\varphi}_-,\Psi_-(\hat{\varphi}_-))
\nonumber\\
\leq\,&Z(\check{\varphi}_-,\Psi_-(\check{\varphi}_-))
F_s(\check{\varphi}_-,\Psi_-(\check{\varphi}_-))
\exp\bigg\{\int_{\hat{\varphi}_-}^{\check{\varphi}_-}
-R_s(s,\Psi_-(s))ds\bigg\}\nonumber\\
&\quad
-Z(\hat{\varphi}_-,\Psi_-(\hat{\varphi}_-))
F_s(\hat{\varphi}_-,\Psi_-(\hat{\varphi}_-))
+M_1\int_{\hat{\varphi}_-}^{\check{\varphi}_-}
\varphi^{-1/4}
\exp\bigg\{\int_{\hat{\varphi}_-}^{\varphi}
-R_s(s,\Psi_-(s))ds\bigg\}d\varphi.
\end{align*}
That is
\begin{align}\label{1-eq3.75minus}
&-U_s(\check{\varphi}_+,\Psi_+(\check{\varphi}_+))
-M_1\int_{0}^{\check{\varphi}_+}\varphi^{-1/4}
\exp\bigg\{\int_{\varphi}^{\check{\varphi}_+}
-R_s(s,\Psi_+(s))ds\bigg\}d\varphi
\nonumber\\
\leq\,&\bigg(
-U_s(\hat{\varphi}_+,\Psi_+(\hat{\varphi}_+))
-M_1\int_{0}^{\hat{\varphi}_+}\varphi^{-1/4}
\exp\bigg\{\int_{\varphi}^{\hat{\varphi}_+}
-R_s(s,\Psi_+(s))ds\bigg\}d\varphi
\bigg)
\exp\bigg\{\int_{\hat{\varphi}_+}^{\check{\varphi}_+}
-R_s(s,\Psi_+(s))ds\bigg\}
\nonumber\\
&\quad
-W(\hat{\varphi}_+,\Psi_+(\hat{\varphi}_+))
F_s(\hat{\varphi}_+,\Psi_+(\hat{\varphi}_+))
\exp\bigg\{\int_{\hat{\varphi}_+}^{\check{\varphi}_+}
-R_s(s,\Psi_+(s))ds\bigg\}
\nonumber\\
&\quad
+W(\check{\varphi}_+,\Psi_+(\check{\varphi}_+))
F_s(\check{\varphi}_+,\Psi_+(\check{\varphi}_+))
\end{align}
and
\begin{align}\label{1-eq3.76plus}
&V_s(\check{\varphi}_-,\Psi_-(\check{\varphi}_-))
-M_1\int_{0}^{\check{\varphi}_-}\varphi^{-1/4}
\exp\bigg\{\int_{\varphi}^{\check{\varphi}_-}
R_s(s,\Psi_-(s))ds\bigg\}d\varphi
\nonumber\\
\leq\,&\bigg(
V_s(\hat{\varphi}_-,\Psi_-(\hat{\varphi}_-))
-M_1\int_{0}^{\hat{\varphi}_-}\varphi^{-1/4}
\exp\bigg\{\int_{\varphi}^{\hat{\varphi}_-}
R_s(s,\Psi_-(s))ds\bigg\}d\varphi
\bigg)
\exp\bigg\{\int_{\hat{\varphi}_-}^{\check{\varphi}_-}
R_s(s,\Psi_-(s))ds\bigg\}
\nonumber\\
&\quad
-Z(\hat{\varphi}_-,\Psi_-(\hat{\varphi}_-))
F_s(\hat{\varphi}_-,\Psi_-(\hat{\varphi}_-))
\exp\bigg\{\int_{\hat{\varphi}_-}^{\check{\varphi}_-}
R_s(s,\Psi_-(s))ds\bigg\}
\nonumber\\
&\quad
+Z(\check{\varphi}_-,\Psi_-(\check{\varphi}_-))
F_s(\check{\varphi}_-,\Psi_-(\check{\varphi}_-)).
\end{align}}
\par Now, we estimate $U_s$ by the method of characteristics. Define functions $r_2(s)$ $(0\leq s<\varphi)$ and $r_2^*(s)$ $(0\leq s<\varphi)$ {as follows:}
\begin{equation*}
\begin{aligned}
r_2(s)=\left\{
\begin{aligned}
&-R_s(s,\Psi_j(s)),
\quad
\varphi_{j}\le s<\varphi_{j-1},\, 1\le j\le k+1\mbox{ and $j$ is odd},
\\[2.5 mm]
&R_s(s,\Psi_j(s)),
\quad
\varphi_{j}\le s<\varphi_{j-1},\,1\le j\le k+1\mbox{ and $j$ is even}.
\end{aligned}
\right.\\
r_2^*(s)=\left\{
\begin{aligned}
&R_s(s,\Psi_j^*(s)),
\quad
\varphi_{j}^*\le s<\varphi_{j-1}^*,\, 1\le j\le {k^*}+1\mbox{ and $j$ is odd},
\\[2.5 mm]
&-R_s(s,\Psi_j^*(s)),
\quad
\varphi_{j}^*\le s<\varphi_{j-1}^*,\,1\le j\le {k^*}+1\mbox{ and $j$ is even},
\end{aligned}
\right.
\end{aligned}
\end{equation*}
 where $(\varphi_j,\Psi_j)$ and  $(\varphi_j^*,\Psi_j^*)$ are defined in Figures \ref{1-pic2}--\ref{1-pic5}. For $1\le j\le k+1$, it follows from \eqref{1-eq3.75} and \eqref{1-eq3.76} that
\begin{align}\label{1-eq3.755}
&U_s(\varphi_{j-1},\psi_{j-1})-M_1\int_{0}^{\varphi_{j-1}}\varphi^{-1/4}\exp\bigg\{\int_{\varphi}^{\varphi_{j-1}}r_2(s)ds\bigg\}d\varphi \nonumber\\
\leq\,&\bigg(U_s(\varphi_j,\psi_j)-M_1\int_{0}^{\varphi_j}\varphi^{-1/4}\exp\bigg\{\int_{\varphi}^{\varphi_j}r_2(s)ds\bigg\}d\varphi\bigg)
\exp\bigg\{\int_{\varphi_j}^{\varphi_{j-1}}r_2(s)ds\bigg\}\nonumber\\
&\quad+W(\varphi_j,\psi_j)F_s(\varphi_j,\psi_j)\exp\bigg\{\int_{\varphi_j}^{\varphi_{j-1}}r_2(s)ds\bigg\} -W(\varphi_{j-1},\psi_{j-1})F_s(\varphi_{j-1},\psi_{j-1}),\mbox{ if $j$ is odd},
\end{align} and
\begin{align}\label{1-eq3.766}
&-V_s(\varphi_{j-1},\psi_{j-1})-M_1\int_{0}^{\varphi_{j-1}}\varphi^{-1/4}\exp\bigg\{\int_{\varphi}^{\varphi_{j-1}}r_2(s)ds\bigg\}d\varphi \nonumber\\
\leq\,&\bigg(-V_s(\varphi_j,\psi_j)-M_1\int_{0}^{\varphi_j}\varphi^{-1/4}\exp\bigg\{\int_{\varphi}^{\varphi_j}r_2(s)ds\bigg\}d\varphi\bigg)
\exp\bigg\{\int_{\varphi_j}^{\varphi_{j-1}}r_2(s)ds\bigg\}\nonumber\\
&\quad+Z(\varphi_j,\psi_j)F_s(\varphi_j,\psi_j)\exp\bigg\{\int_{\varphi_j}
^{\varphi_{j-1}}r_2(s)ds\bigg\}
-Z(\varphi_{j-1},\psi_{j-1})F_s(\varphi_{j-1},\psi_{j-1}),\mbox{ if $j$ is even}.
\end{align}
{Similarly, for $1\le j\le k+1$, it follows from
\eqref{1-eq3.75minus} and \eqref{1-eq3.76plus} that
\begin{align}\label{1-eq3.755minus}
&-U_s(\varphi_{j-1},\psi_{j-1})
-M_1\int_{0}^{\varphi_{j-1}}\varphi^{-1/4}
\exp\bigg\{\int_{\varphi}^{\varphi_{j-1}}r_2(s)ds\bigg\}d\varphi
\nonumber\\
\leq\,&\bigg(
-U_s(\varphi_j,\psi_j)
-M_1\int_{0}^{\varphi_j}\varphi^{-1/4}
\exp\bigg\{\int_{\varphi}^{\varphi_j}r_2(s)ds\bigg\}d\varphi
\bigg)
\exp\bigg\{\int_{\varphi_j}^{\varphi_{j-1}}r_2(s)ds\bigg\}
\nonumber\\
&\quad
-W(\varphi_j,\psi_j)F_s(\varphi_j,\psi_j)
\exp\bigg\{\int_{\varphi_j}^{\varphi_{j-1}}r_2(s)ds\bigg\}
+W(\varphi_{j-1},\psi_{j-1})F_s(\varphi_{j-1},\psi_{j-1}),
\mbox{ if $j$ is odd},
\end{align}
and
\begin{align}\label{1-eq3.766plus}
&V_s(\varphi_{j-1},\psi_{j-1})
-M_1\int_{0}^{\varphi_{j-1}}\varphi^{-1/4}
\exp\bigg\{\int_{\varphi}^{\varphi_{j-1}}r_2(s)ds\bigg\}d\varphi
\nonumber\\
\leq\,&\bigg(
V_s(\varphi_j,\psi_j)
-M_1\int_{0}^{\varphi_j}\varphi^{-1/4}
\exp\bigg\{\int_{\varphi}^{\varphi_j}r_2(s)ds\bigg\}d\varphi
\bigg)
\exp\bigg\{\int_{\varphi_j}^{\varphi_{j-1}}r_2(s)ds\bigg\}
\nonumber\\
&\quad
-Z(\varphi_j,\psi_j)F_s(\varphi_j,\psi_j)
\exp\bigg\{\int_{\varphi_j}^{\varphi_{j-1}}r_2(s)ds\bigg\}
+Z(\varphi_{j-1},\psi_{j-1})F_s(\varphi_{j-1},\psi_{j-1}),
\mbox{ if $j$ is even}.
\end{align}}
If $k$ is even, one can get from \eqref{1-eq3.11}, \eqref{1-eq3.12}, \eqref{1-eq3.69}--\eqref{1-eq3.72}, \eqref{1-eq3.755} and \eqref{1-eq3.766} together with  $F_s(\varphi _j,\psi_j)={0}$ $(1 \le j \le k)$ {and $F_s(\varphi _j^*,\psi_j^*)=0$ $(1 \le j \le k^*)$} that
\begin{align}\label{1-eq3.79}
&U_s(\varphi_0,\psi_0)-M_1\int_{0}^{\varphi_0}\varphi^{-1/4}\exp\bigg\{\int_{\varphi}^{\varphi_0}r_2(s)ds\bigg\}d\varphi\nonumber\\
\leq\,&\bigg(U_s(\varphi_1,\psi_1)-M_1\int_{0}^{\varphi_1}\varphi^{-1/4}\exp\bigg\{\int_{\varphi}^{\varphi_1}r_2(s)ds\bigg\}d\varphi\bigg)
\exp\bigg\{\int_{\varphi_1}^{\varphi_0}r_2(s)ds\bigg\}\nonumber\\
&\quad+W(\varphi_1,\psi_1)F_s(\varphi_1,\psi_1)\exp\bigg\{\int_{\varphi_1}^{\varphi_0}r_2(s)ds\bigg\} -W(\varphi_0,\psi_0)F_s(\varphi_0,\psi_0)\nonumber\\
\leq\,&\bigg({U_{s,0}}(\psi_{k+1})+W_0(\psi_{k+1})F_s(0,\psi_{k+1})\bigg)\exp\bigg\{\int_{0}^{\varphi_0}r_2(s)ds\bigg\}\nonumber\\
&\quad+2\sum_{j=1}^{k}(-1)^{j+1}W(\varphi_j,\psi_j)F_s(\varphi_j,\psi_j)\exp\bigg
\{\int_{\varphi_j}^{\varphi_0}r_2(s)ds\bigg\}-W(\varphi_0,\psi_0)F_s(\varphi_0,\psi_0)\nonumber\\
=\,&\bigg({U_{s,0}}(\psi_{k+1})+W_0(\psi_{k+1})F_s(0,\psi_{k+1})\bigg)\exp
\bigg\{\int_{0}^{\varphi_0}r_2(s)ds\bigg\}-W(\varphi_0,\psi_0)F_s(\varphi_0,\psi_0)
\end{align}
{and if $k^*$ is even},
\begin{align}\label{1-eq3.80}
&-V_s(\varphi_0,\psi_0)-M_1\int_{0}^{\varphi_0}\varphi^{-1/4}\exp\bigg\{\int_{\varphi}^{\varphi_0}r_2^*(s)ds\bigg\}d\varphi\nonumber\\
\leq\,&\bigg(-V_s(\varphi_1^*,\psi_1^*)-M_1\int_{0}^{\varphi_1^*}\varphi^{-1/4}\exp\bigg\{\int_{\varphi}^{\varphi_1^*}r_2^*(s)ds\bigg\}d\varphi\bigg)
\exp\bigg\{\int_{\varphi_1^*}^{\varphi_0}r_2^*(s)ds\bigg\}\nonumber\\
&\quad+Z(\varphi_1^*,\psi_1^*)F_s(\varphi_1^*,\psi_1^*)\exp
\bigg\{\int_{\varphi_1^*}^{\varphi_0}r_2^*(s)ds\bigg\}-Z(\varphi_0,\psi_0)
F_s(\varphi_0,\psi_0)\nonumber\\
\leq\,&\bigg(-{V_{s,0}}(\psi_{{k^*}+1}^*)+Z_0(\psi_{{k^*}+1}^*)F_s(0,\psi_{{k^*}+1}^*)\bigg)\exp\bigg\{\int_{0}^{\varphi_0}r_2^*(s)ds\bigg\}\nonumber\\
&\quad+2\sum_{j=1}^{{k^*}}(-1)^{j+1}Z(\varphi_j^*,\psi_j^*)F_s(\varphi_j^*,\psi_j^*)\exp\bigg\{\int_{\varphi_j^*}^{\varphi_0}r_2^*(s)ds\bigg\}
-Z(\varphi_0,\psi_0)F_s(\varphi_0,\psi_0)\nonumber\\
=\,&\bigg(-{V_{s,0}}(\psi_{{k^*}+1}^*)+Z_0(\psi_{{k^*}+1}^*)
F_s(0,\psi_{{k^*}+1}^*)\bigg)\exp\bigg\{\int_{0}^{\varphi_0}r_2^*(s)ds\bigg\}-Z(\varphi_0,\psi_0)F_s(\varphi_0,\psi_0).
\end{align}
Similarly, if $k$ is odd, one can obtain that
\begin{align}\label{1-eq3.79*}
&U_s(\varphi_0,\psi_0)-M_1\int_{0}^{\varphi_0}\varphi^{-1/4}\exp\bigg\{\int_{\varphi}^{\varphi_0}r_2(s)ds\bigg\}d\varphi\nonumber\\
\leq\,&\bigg(-{V_{s,0}}(\psi_{k+1})+Z_0(\psi_{k+1})F_s(0,\psi_{k+1})\bigg)
\exp\bigg\{\int_{0}^{\varphi_0}r_2(s)ds\bigg\}
-W(\varphi_0,\psi_0)F_s(\varphi_0,\psi_0)
\end{align}
{and if $k^*$ is odd},
\begin{align}\label{1-eq3.80*}
&-V_s(\varphi_0,\psi_0)-M_1\int_{0}^{\varphi_0}\varphi^{-1/4}\exp\bigg\{\int_{\varphi}^{\varphi_0}r_2^*(s)ds\bigg\}d\varphi\nonumber\\
\leq\,&\bigg({U_{s,0}}(\psi_{{k^*}+1}^*)+W_0(\psi_{{k^*}+1}^*)F_s(0,\psi_{{k^*}+1}^*)\bigg)\exp\bigg\{\int_{0}^{\varphi_0}r_2^*(s)ds\bigg\}
-Z(\varphi_0,\psi_0)F_s(\varphi_0,\psi_0).
\end{align}
{For $-U_s$ and $V_s$, the same iteration gives the following
estimates. If $k$ is even, one obtains that
\begin{align}\label{1-eq3.79minus}
&-U_s(\varphi_0,\psi_0)
-M_1\int_{0}^{\varphi_0}\varphi^{-1/4}
\exp\bigg\{\int_{\varphi}^{\varphi_0}r_2(s)ds\bigg\}d\varphi
\nonumber\\
\leq\,&\bigg(
-U_s(\varphi_1,\psi_1)
-M_1\int_{0}^{\varphi_1}\varphi^{-1/4}
\exp\bigg\{\int_{\varphi}^{\varphi_1}r_2(s)ds\bigg\}d\varphi
\bigg)
\exp\bigg\{\int_{\varphi_1}^{\varphi_0}r_2(s)ds\bigg\}
\nonumber\\
&\quad
-W(\varphi_1,\psi_1)F_s(\varphi_1,\psi_1)
\exp\bigg\{\int_{\varphi_1}^{\varphi_0}r_2(s)ds\bigg\}
+W(\varphi_0,\psi_0)F_s(\varphi_0,\psi_0)
\nonumber\\
\leq\,&\bigg(
-{U_{s,0}}(\psi_{k+1})
-W_0(\psi_{k+1})F_s(0,\psi_{k+1})
\bigg)
\exp\bigg\{\int_{0}^{\varphi_0}r_2(s)ds\bigg\}
\nonumber\\
&\quad
-2\sum_{j=1}^{k}(-1)^{j+1}
W(\varphi_j,\psi_j)F_s(\varphi_j,\psi_j)
\exp\bigg\{\int_{\varphi_j}^{\varphi_0}r_2(s)ds\bigg\}
+W(\varphi_0,\psi_0)F_s(\varphi_0,\psi_0)
\nonumber\\
=\,&\bigg(
-{U_{s,0}}(\psi_{k+1})
-W_0(\psi_{k+1})F_s(0,\psi_{k+1})
\bigg)
\exp\bigg\{\int_{0}^{\varphi_0}r_2(s)ds\bigg\}
+W(\varphi_0,\psi_0)F_s(\varphi_0,\psi_0)
\end{align}
{and if $k^*$ is even},
\begin{align}\label{1-eq3.80plus}
&V_s(\varphi_0,\psi_0)
-M_1\int_{0}^{\varphi_0}\varphi^{-1/4}
\exp\bigg\{\int_{\varphi}^{\varphi_0}r_2^*(s)ds\bigg\}d\varphi
\nonumber\\
\leq\,&\bigg(
V_s(\varphi_1^*,\psi_1^*)
-M_1\int_{0}^{\varphi_1^*}\varphi^{-1/4}
\exp\bigg\{\int_{\varphi}^{\varphi_1^*}r_2^*(s)ds\bigg\}d\varphi
\bigg)
\exp\bigg\{\int_{\varphi_1^*}^{\varphi_0}r_2^*(s)ds\bigg\}
\nonumber\\
&\quad
-Z(\varphi_1^*,\psi_1^*)F_s(\varphi_1^*,\psi_1^*)
\exp\bigg\{\int_{\varphi_1^*}^{\varphi_0}r_2^*(s)ds\bigg\}
+Z(\varphi_0,\psi_0)F_s(\varphi_0,\psi_0)
\nonumber\\
\leq\,&\bigg(
{V_{s,0}}(\psi_{{k^*}+1}^*)
-Z_0(\psi_{{k^*}+1}^*)F_s(0,\psi_{{k^*}+1}^*)
\bigg)
\exp\bigg\{\int_{0}^{\varphi_0}r_2^*(s)ds\bigg\}
\nonumber\\
&\quad
-2\sum_{j=1}^{{k^*}}(-1)^{j+1}
Z(\varphi_j^*,\psi_j^*)F_s(\varphi_j^*,\psi_j^*)
\exp\bigg\{\int_{\varphi_j^*}^{\varphi_0}r_2^*(s)ds\bigg\}
+Z(\varphi_0,\psi_0)F_s(\varphi_0,\psi_0)
\nonumber\\
=\,&\bigg(
{V_{s,0}}(\psi_{{k^*}+1}^*)
-Z_0(\psi_{{k^*}+1}^*)F_s(0,\psi_{{k^*}+1}^*)
\bigg)
\exp\bigg\{\int_{0}^{\varphi_0}r_2^*(s)ds\bigg\}
+Z(\varphi_0,\psi_0)F_s(\varphi_0,\psi_0).
\end{align}
Similarly, if $k$ is odd, one can obtain that
\begin{align}\label{1-eq3.79minusodd}
&-U_s(\varphi_0,\psi_0)
-M_1\int_{0}^{\varphi_0}\varphi^{-1/4}
\exp\bigg\{\int_{\varphi}^{\varphi_0}r_2(s)ds\bigg\}d\varphi
\nonumber\\
\leq\,&\bigg(
{V_{s,0}}(\psi_{k+1})
-Z_0(\psi_{k+1})F_s(0,\psi_{k+1})
\bigg)
\exp\bigg\{\int_{0}^{\varphi_0}r_2(s)ds\bigg\}
+W(\varphi_0,\psi_0)F_s(\varphi_0,\psi_0)
\end{align}
{and if $k^*$ is odd},
\begin{align}\label{1-eq3.80plusodd}
&V_s(\varphi_0,\psi_0)
-M_1\int_{0}^{\varphi_0}\varphi^{-1/4}
\exp\bigg\{\int_{\varphi}^{\varphi_0}r_2^*(s)ds\bigg\}d\varphi
\nonumber\\
\leq\,&\bigg(
-{U_{s,0}}(\psi_{{k^*}+1}^*)
-W_0(\psi_{{k^*}+1}^*)F_s(0,\psi_{{k^*}+1}^*)
\bigg)
\exp\bigg\{\int_{0}^{\varphi_0}r_2^*(s)ds\bigg\}
+Z(\varphi_0,\psi_0)F_s(\varphi_0,\psi_0).
\end{align}
Then it follows from \eqref{1-eq3.79}, \eqref{1-eq3.79*},
\eqref{1-eq3.79minus} and \eqref{1-eq3.79minusodd} that
\begin{align}\label{1-eq3.81}
|U_s(\varphi_0,\psi_0)|\leq J_1+J_2+J_3.
\end{align}}
where
\begin{equation*}
 \begin{aligned} J_1=&\left(\big\||(U_{s,0},V_{s,0})|\big\|_{L^{\infty}((0,m))}+\big\||(W_{0},Z_{0})|\big\|_{L^{\infty}((0,m))}|F_s(0,\psi_{k+1})|\right)\exp\bigg\{\int_{0}^{\varphi_0}r_2(s)ds\bigg\},\nonumber\\
J_2=&|W(\varphi_0,\psi_0)F_s(\varphi_0,\psi_0)|,\nonumber\\
J_3=&M_1\int_{0}^{\varphi_0}\varphi^{-1/4}\exp\bigg\{\int_{\varphi}^{\varphi_0}r_2(s)ds\bigg\}d\varphi.
\end{aligned}\end{equation*}
Next, we estimate $J_i$ $(i=1,2,3)$.
 For $1\le j \le k+1$, it follows from \eqref{1-eq2.36}, \eqref{1-eq3.21}, \eqref{1-eqe}, \eqref{1-beta*} , \eqref{1-eqr} and \eqref{1-equ} that
\begin{equation*}
\begin{aligned}
J_1\leq\,&\bigg(\frac{\mu_{10}(\beta_2\beta_3+\beta_1\sigma)}{\beta_1}
+\frac{\mu_{10}\mu_2\beta_3}{\beta_1}\bigg)
\exp\bigg\{\int_{0}^{\varphi_0}\frac{\mu_{10}\beta_3}{\beta_1^{5/4}}s^{-1/4}ds\bigg\}\\
\leq\,&\frac{\mu_{10}(2\beta_2\beta_3+\beta_1\sigma)}{\beta_1}\exp\bigg\{\frac{\mu_{12}\beta_3{\zeta^{3/4}}}{\beta_1^{5/4}}\bigg\}\nonumber,\\
J_2
\leq\,&\frac{\mu_2(1+\mu_6\beta_3)\mu_{10}\beta_3}{\beta_1}
=\frac{\mu_{10}\beta_2\beta_3}{\beta_1}
\leq\frac{\mu_{10}\beta_2\beta_3}{\beta_1}\exp\bigg\{\frac{\mu_{12}\beta_3{\zeta^{3/4}}}{\beta_1^{5/4}}\bigg\},\\
J_3
\leq\,& M_1\exp\bigg\{\frac{\mu_{12}\beta_3{\zeta^{3/4}}}{\beta_1^{5/4}}\bigg\}\int_{0}^{\varphi_0}\varphi^{-1/4}d\varphi
\leq\frac{4M_1}{3}\exp\bigg\{\frac{\mu_{12}\beta_3{\zeta^{3/4}}}{\beta_1^{5/4}}\bigg\}\varphi_0^{3/4}\\
\leq\,&\bigg(\frac{4\mu_{11}\beta_2\beta_3^2}{3\beta_1^{9/4}} + \frac{\hat{\mu}\beta_2\beta_6\sigma}{ 2\beta_1^{5/4}} \bigg)
 \exp\bigg\{ \frac{\mu_{12} \beta_3{\zeta^{3/4}}}{\beta_1^{5/4}} \bigg\},
\end{aligned}
\end{equation*}
where $\mu_{12}$ is a positive constant depending only on $\mu_{10}$.
Then we obtain that
\begin{align}\label{1-eq3.81*}
|U_s(\varphi_0,\psi_0)|
\leq\,&\bigg(\frac{\mu_{10}\left(3\beta_2\beta_3+\beta_1\sigma\right)}{\beta_1}
+\frac{4\mu_{11}\beta_2\beta_3^2}{3\beta_1^{9/4}}
+ \frac{\hat{\mu} \beta_2\beta_6\sigma}{ 2\beta_1^{5/4}}\bigg)\exp\bigg\{\frac{\mu_{12}\beta_3{\zeta^{3/4}}}{\beta_1^{5/4}}\bigg\}
\nonumber\\
\leq\,&\mu_{13}\bigg(\frac{\beta_2\beta_3+\beta_1\sigma}{\beta_1}
+\frac{\beta_2\beta_3^2}{\beta_1^{9/4}}
+ \frac{\beta_2\beta_6\sigma}{ \beta_1^{5/4}}\bigg)\exp\bigg\{\frac{\mu_{12}\beta_3{\zeta^{3/4}}}{\beta_1^{5/4}}\bigg\},
\end{align}
where $\mu_{13}$ is a positive constant depending only on $\hat{\mu}$, $\mu_{10}$ and $\mu_{11}$. The estimate of $|V_s|$ can be derived similarly.
\par In the following, we turn to estimate $|U_s-V_s|$. If $k$ {and $k^*$ are} even, {one can get from \eqref{1-eq3.79}, \eqref{1-eq3.80}, \eqref{1-eq3.79minus} and \eqref{1-eq3.80plus} that}
\begin{align}\label{1-eq3.83}
\left|U_s(\varphi_0,\psi_0)-V_s(\varphi_0,\psi_0)\right|\leq\,&J_4+J_5+J_6,
\end{align}
where
\begin{equation*}
  \begin{aligned}
J_4=\,&\left|\bigg({U_{s,0}}(\psi_{k+1})+W_0(\psi_{k+1})F_s(0,\psi_{k+1})\bigg)
\exp\bigg\{\int_{0}^{\varphi_0}r_2(s)ds\bigg\}\right.\\
&\quad\left.+\bigg(-{V_{s,0}}(\psi_{{k^*}+1}^*)+Z_0(\psi_{{k^*}+1}^*)F_s(0,\psi_{{k^*}+1}^*)\bigg)
\exp\bigg\{\int_{0}^{\varphi_0}r_2^*(s)ds\bigg\}\right|,\\
J_5=\,&\left|\left(W(\varphi_0,\psi_0)+Z(\varphi_0,\psi_0)\right)F_s(\varphi_0,\psi_0)\right|,\\
J_6=\,&M_1\int_{0}^{\varphi_0}\varphi^{-1/4}\exp\bigg\{\int_{\varphi}^{\varphi_0}r_2(s)ds\bigg\}d\varphi
+M_1\int_{0}^{\varphi_0}\varphi^{-1/4}\exp\bigg\{\int_{\varphi}^{\varphi_0}r_2^*(s)ds\bigg\}d\varphi.
  \end{aligned}
\end{equation*}
For the first term $J_4$, it follows from \eqref{1-eqvar}, \eqref{1-eqr} and \eqref{1-equ} that
\begin{align*}
J_4
=\,&\left|W_{\psi}(0,\psi_{k+1})\exp\bigg\{\int_{0}^{\varphi_0}r_2(s)ds\bigg\}
+Z_{\psi}(0,\psi_{{k^*}+1}^*)\exp\bigg\{\int_{0}^{\varphi_0}r_2^*(s)ds\bigg\}\right|\\
{\leq}\,&\left|W_{\psi}(0,\Psi_+(0))+Z_{\psi}(0,\Psi_-(0))\right|\exp\bigg\{\int_{0}^{\varphi_0}r_2(s)ds\bigg\}\\
&\quad+\left|Z_{\psi}(0,\Psi_-(0))\right|\left|\exp\bigg\{\int_{0}^{\varphi_0}r_2^*(s)ds\bigg\}-\exp\bigg\{\int_{0}^{\varphi_0}r_2(s)ds\bigg\}\right|\\
\leq\,& \left|
  \frac{1}{c_*^{2+4/(\gamma-1)}}
  \left(
    \left(
      \sin\Theta(y)
    \right)''\cos\Theta(y)
  \right)'\frac{dy}{d\psi}
   (\Psi_+(0)-\Psi_-(0))
\right| \exp\bigg\{\frac{\mu_{12}\beta_3{\zeta^{3/4}}}{\beta_1^{5/4}}\bigg\} \\
&\quad+\frac{2\mu_{10}\mu_{12}\beta_3\sigma}{\beta_1^{5/4}}
\exp\bigg\{\frac{\mu_{12}\beta_3{\zeta^{3/4}}}{\beta_1^{5/4}}\bigg\}\varphi_0^{3/4}\\
\leq\,& \left|
  \frac{1}{c_*^{2+4/(\gamma-1)}}
    \bigg(\left(\sin\Theta(y)\right)'''\cos\Theta(y)\frac{dy}{d\psi}
    - \left(\sin\Theta(y)\right)''\sin\Theta(y)\Theta'(y)
  \frac{dy}{d\psi}\bigg)
  (\Psi_+(0)-\Psi_-(0))
\right|\\
&\quad\times\exp\bigg\{\frac{\mu_{12}\beta_3{\zeta^{3/4}}}{\beta_1^{5/4}}\bigg\}
+\frac{2\mu_{10}\mu_{12}\beta_3\sigma}{\beta_1^{5/4}}
\exp\bigg\{\frac{\mu_{12}\beta_3{\zeta^{3/4}}}{\beta_1^{5/4}}\bigg\}\varphi_0^{3/4}\\
\leq\,&\frac{\mu_{14}\sigma}{\beta_1^{5/4}}\bigg(\beta_1(1+\sigma)+\beta_3
\bigg)
\exp\bigg\{\frac{\mu_{12}\beta_3{\zeta^{3/4}}}{\beta_1^{5/4}}\bigg\}\varphi_0^{3/4},
\end{align*}
where $\mu_{14}$ is a positive constant depending only on $\mu_3$, $\mu_{10}$, $\mu_{12}$, $R_0$ and $\gamma$.
Next, for $J_5$ and $J_6$, it follows from \eqref{1-eq3.22}, \eqref{1-eqe}, \eqref{1-eqr} and $J_3$ that
\begin{equation*}
\begin{aligned}
J_5
\leq\,&\frac{\mu_{10}\beta_3}{\beta_1}\bigg(\frac{\mu_4\sigma}{\beta_1^{1/4}}
+\frac{8\hat{\mu}\mu_2\beta_3}{9\beta_1^{5/4}}\bigg)
(1+\mu_6\beta_3)\varphi_0^{3/4}
\leq\mu_{15} \bigg( \frac{\beta_3\sigma + \beta_3^2\sigma}{\beta_1^{5/4}} + \frac{\beta_3^2 + \beta_3^3}{\beta_1^{9/4}} \bigg) \varphi_0^{3/4}\\
\leq\,&\mu_{15} \bigg( \frac{\beta_3\sigma + \beta_3^2\sigma}{\beta_1^{5/4}} + \frac{\beta_3^2 + \beta_3^3}{\beta_1^{9/4}} \bigg)
\exp\bigg\{\frac{\mu_{12}\beta_3{\zeta^{3/4}}}{\beta_1^{5/4}}\bigg\}
\varphi_0^{3/4},\\
J_6
\leq\,&\bigg(\frac{8\mu_{11}\beta_2\beta_3^2}{ 3\beta_1^{9/4}} + \frac{\hat{\mu} \beta_2\beta_6\sigma}{ \beta_1^{5/4}} \bigg)
 \exp\bigg\{ \frac{\mu_{12} \beta_3{\zeta^{3/4}}}{\beta_1^{5/4}} \bigg\}\varphi_0^{3/4},
\end{aligned}
\end{equation*}
where $\mu_{15}$ is a positive constant depending only on $\gamma$, $R_0$, $\hat{\mu}$, $\mu_4$, $\mu_6$ and $\mu_{10}$. Collecting  the above estimates yields that
\begin{align}\label{1-eq3.84}
&\left|U_s(\varphi_0,\psi_0)-V_s(\varphi_0,\psi_0)\right|\nonumber\\
\leq\,&\mu_{16}\bigg(\frac{(1+\sigma)\sigma}{\beta_1^{1/4}}
+\frac{(\beta_3+\beta_3^2)\sigma}{\beta_1^{5/4}}
+\frac{(1+\beta_2)\beta_3^2+\beta_3^3}{ \beta_1^{9/4}} \bigg)
\exp\bigg\{ \frac{\mu_{12} \beta_3{\zeta^{3/4}}}{\beta_1^{5/4}} \bigg\}\varphi_0^{3/4}\nonumber\\
&\quad+ \frac{\hat{\mu} \beta_2\beta_6\sigma}{ \beta_1^{5/4}}
 \exp\bigg\{ \frac{\mu_{12} \beta_3{\zeta^{3/4}}}{\beta_1^{5/4}} \bigg\}\varphi_0^{3/4},
\end{align}
where $\mu_{16}$ is a positive constant depending only on $\mu_{11}$, $\mu_{14}$ and $\mu_{15}$.
{
Similarly, when $k$ and $k^*$ are both odd, one can derive the same estimates from \eqref{1-eq3.79*}, \eqref{1-eq3.80*}, \eqref{1-eq3.79minusodd} and \eqref{1-eq3.80plusodd}. When $k$ and $k^*$ have different parities, set
$$
\psi_b=
\begin{cases}
m, & \text{if $k$ is even and $k^*$ is odd},\\
0, & \text{if $k$ is odd and $k^*$ is even}.
\end{cases}
$$
By the compatibility condition \eqref{1-eq2.43},
$W_{0,\psi}(0)=W_{0,\psi}(m)=0$,
and hence $W_{0,\psi}(\psi_b)=0$. Then
\begin{align*}
J_4
={}&\left|
W_{0,\psi}(\psi_{k+1})
\exp\bigg\{\int_0^{\varphi_0}r_2(s)\,ds\bigg\}
+W_{0,\psi}(\psi_{k^*+1}^*)
\exp\bigg\{\int_0^{\varphi_0}r_2^*(s)\,ds\bigg\}
\right|\\
\leq{}&
\bigg(
\|W_{0,\psi\psi}\|_{L^\infty(0,m)}
\bigl(|\psi_{k+1}-\psi_b|
     +|\psi_{k^*+1}^*-\psi_b|\bigr)
+\frac{2\mu_{10}\mu_{12}\beta_3\sigma}
{\beta_1^{5/4}}\varphi_0^{3/4}
\bigg)
\exp\bigg\{
\frac{\mu_{12}\beta_3\zeta^{3/4}}{\beta_1^{5/4}}
\bigg\}\\
\leq{}&
\bigg(
2\|W_{0,\psi\psi}\|_{L^\infty(0,m)}
\int_0^{\varphi_0}
\|b^{1/2}(\tilde Q)(s,\cdot)\|_{L^\infty(0,m)}\,ds
+\frac{2\mu_{10}\mu_{12}\beta_3\sigma}
{\beta_1^{5/4}}\varphi_0^{3/4}
\bigg)
\exp\bigg\{
\frac{\mu_{12}\beta_3\zeta^{3/4}}{\beta_1^{5/4}}
\bigg\}\\
\leq{}&
\frac{\mu_{14}\sigma}{\beta_1^{5/4}}
\bigg(\beta_1(1+\sigma)+\beta_3\bigg)
\exp\bigg\{
\frac{\mu_{12}\beta_3\zeta^{3/4}}{\beta_1^{5/4}}
\bigg\}\varphi_0^{3/4}.
\end{align*}
While the estimates for $J_5$ and $J_6$ remain unchanged.
Combining all four cases yields \eqref{1-eq3.84}.}
\par
Under the assumptions of \eqref{1-eqe1}, we consider the function{
$$ f_2(t)=\exp\bigg\{\frac{4^{5/4}\mu_{12}t}{(3\tilde\mu_1)^{5/4}}\bigg\}-1-\mu_{17}t,\ \ t>0,$$
where $\mu_{17}>\frac{4^{5/4}\mu_{12}}{(3\tilde{\mu}_1)^{5/4}}$ is a positive constant depending only on $\gamma$, $R_0$, $\sigma_1$ and $\mu_{12}$. A direct computation yields that $f_2(0)=0$ and $ f_2'(0)=\frac{4^{5/4}\mu_{12}}{(3\tilde \mu_1)^{5/4}}-\mu_{17}<0$. Then  there exists $N_2>0$, independent of $\sigma$, such that $f_2(t)\leq0$ for $0\leq t\leq N_2$. Using $\mu_1\geq\tilde\mu_1$ in \eqref{1-mu1}, if $\beta_3\leq N_2$, the following estimate holds:}
{\begin{align}\label{1-eqe2}
\exp\bigg\{\frac{\mu_{12}\beta_3\zeta^{3/4}}{\beta_1^{5/4}}\bigg\}-1\leq
\exp\bigg\{\frac{4^{5/4}\mu_{12}\beta_3}{(3\tilde{\mu}_1)^{5/4}}\bigg\}-1
\le \mu_{17}\beta_3.
\end{align}}
 {Combining the four cases  with \eqref{1-beta*}, \eqref{1-eq3.81*}, \eqref{1-eq3.84} and \eqref{1-eqe2} yields that}
\begin{align}\label{1-equ1}
&|\left(U_s(\varphi_0,\psi_0),V_s(\varphi_0,\psi_0)\right)|\nonumber\\
\leq\,&\mu_{13}\bigg(\frac{\beta_2\beta_3+\beta_1\sigma}{\beta_1}
+\frac{\beta_2\beta_3^2}{\beta_1^{9/4}}
+ \frac{\beta_2\beta_6\sigma}{ \beta_1^{5/4}}\bigg)(1+\mu_9\mu_{17}\sigma)\nonumber\\
\leq\,&\left(M_2\beta_6+M_3\right)\sigma
\end{align}
and
\begin{align}\label{1-equv}
&\left|U_s(\varphi_0,\psi_0)-V_s(\varphi_0,\psi_0)\right|\nonumber\\
\leq\,&\mu_{16}\bigg(\frac{(1+\sigma)\sigma}{\beta_1^{1/4}}
+\frac{(\beta_3+\beta_3^2)\sigma}{\beta_1^{5/4}}
+\frac{(1+\beta_2)\beta_3^2+\beta_3^3}{ \beta_1^{9/4}} \bigg)
(1+\mu_9\mu_{17}\sigma)\varphi_0^{3/4}\nonumber\\
&\quad+ \frac{\hat{\mu} \beta_2\beta_6\sigma}{ \beta_1^{5/4}}
 (1+\mu_9\mu_{17}\sigma)\varphi_0^{3/4}
 \leq\bigg(M_4\beta_6
+M_5\bigg)\sigma\varphi_0^{3/4},
\end{align}
with
\begin{equation}\label{1-M2}
\begin{cases}
\begin{aligned}
&M_2=\frac{\mu_{13}\left(1+\mu_9\mu_{17}\sigma\right)\beta_2}{\beta_1^{5/4}},\ \
M_4=\frac{\hat{\mu}(1+\mu_9\mu_{17}\sigma)\beta_2}{\beta_1^{5/4}},\\
&M_3 =\mu_{13} \bigg( \frac{\mu_9\beta_2+\beta_1}{\beta_1} + \frac{\mu_9^2\beta_2\sigma}{\beta_1^{9/4}} \bigg)
\left(1 + \mu_9\mu_{17}\sigma\right),\\
&M_5=\mu_{16}\bigg(\frac{1+\sigma}{\beta_1^{1/4}}
+\frac{\mu_9\sigma+\mu_9^2\sigma^2}{\beta_1^{5/4}}
+\frac{\mu_9^2(1+\beta_2)\sigma+\mu_9^3\sigma^2}{ \beta_1^{9/4}} \bigg) \left(1 + \mu_9\mu_{17}\sigma\right).\end{aligned}
\end{cases}
\end{equation}
Thus \eqref{1-eq3.62}--\eqref{1-eq3.63} can be obtained.
 Furthermore, it follows from \eqref{1-eq3.34} and \eqref{1-eq3.35} that
\begin{align}\label{1-jie1}
\beta_1\ge  {\mu_0}-\tilde{\beta}_1\sigma_1:=\hat{\beta}_1,\quad
\beta_2\le  {\mu_0}+\tilde{\beta}_2\sigma_1:=\hat{\beta}_2.
 \end{align}
 Combining \eqref{1-qq2}, \eqref{1-eqr}, \eqref{1-equ1} and \eqref{1-jie1} yields that
\begin{align*}
|Q_{\varphi\psi}(\varphi,\psi)|
=\,&\bigg|-F_sQ_{\varphi}-\frac{1}{2}(U_s+V_s)\bigg|\leq |F_sQ_{\varphi}|+\frac{1}{2}|U_s+V_s|\\
\leq \,&\left(M_2\beta_6+M_3+\frac{\mu_9\mu_{10}\beta_2}{\beta_1}\right)\sigma\leq(M_6\beta_6+M_7)\sigma,
\end{align*}
with
\begin{equation}\label{1-M3} M_6=\frac{\mu_{13}\left(1+\mu_9\mu_{17}\sigma_1\right)\hat{\beta}_2}{\hat{\beta}_1^{5/4}},\quad M_7=\mu_{13} \bigg( 1+\frac{\mu_9\hat{\beta}_2}{\hat{\beta}_1} + \frac{\mu_9^2\hat{\beta}_2\sigma_1}{\hat{\beta}_1^{9/4}} \bigg)
\left(1 + \mu_9\mu_{17}\sigma_1\right)+\frac{\mu_9\mu_{10}\hat{\beta}_2}{\hat{\beta}_1}.
\end{equation}
{Finally, owing to \eqref{1-eq2.91}, \eqref{1-eq3.16}, \eqref{1-equv} and \eqref{1-jie1}, it holds that
\begin{equation*}
  \begin{aligned}
  |Q_{\varphi\varphi}(\varphi,\psi)|=\,&\frac{1}{2}\left|b^{1/2}(\tilde{Q})(U_s-V_s)\right|
\leq \bigg(M_8\beta_6+M_9\bigg)\sigma\varphi^{1/2},
  \end{aligned}
\end{equation*}
with
\begin{equation}\label{1-M4}
\begin{cases}
\begin{aligned}
&M_8=\frac{\mu_{18}\hat{\beta}_2}{\hat{\beta}_1^{3/2}}
\left(1+\mu_9\mu_{17}\sigma_1\right),\\
&M_9=\mu_{18}\bigg(\frac{1+\sigma_1}{\hat{\beta}_1^{1/2}}
+\frac{\mu_9\sigma_1+\mu_9^2\sigma_1^2}{\hat{\beta}_1^{3/2}}
+\frac{\mu_9^2(1+\hat{\beta}_2)\sigma_1+\mu_9^3\sigma_1^2}{ \hat{\beta}_1^{5/2}} \bigg) \left(1 + \mu_9\mu_{17}\sigma_1\right).\\
\end{aligned}\end{cases}\end{equation}
Here $\mu_{18}$ is a positive constant depending only on $\hat{\mu}$, $\mu_9$, $\mu_{17}$. It follows from \eqref{1-eq3.15} that
\begin{equation*}
\begin{aligned}
Q_{\psi\psi}
={}&\frac{1}{2}b^{-1/2}(\tilde Q)(U_s-V_s)
+\frac{1}{2}b^{-1/2}(\tilde Q)F_s(W+Z)\\
&\quad+\frac{1}{4}\int_0^\varphi
b^{-3/2}(\tilde Q)p(\tilde Q)
\tilde Q_\varphi(U_s-V_s)\,ds
+\frac{1}{4}\int_0^\varphi
b^{-3/2}(\tilde Q)p(\tilde Q)
\tilde Q_\varphi F_s(W+Z)\,ds\\
&\quad+\frac{1}{2}\int_0^\varphi
b^{-1}(\tilde Q)p(\tilde Q)
\tilde Q_\psi(U_s+V_s)\,ds
+\int_0^\varphi
b^{-1}(\tilde Q)p(\tilde Q)
\tilde Q_\psi F_sQ_\varphi\,ds\\
&\quad-\frac{1}{2}\int_0^\varphi
b^{-1}(\tilde Q)p(\tilde Q)
\tilde Q_{\psi\psi}Q_\varphi\,ds\\
&\quad-\frac{1}{2}\int_0^\varphi
\bigg(b^{-1}(\tilde Q)p'(\tilde Q)
-\frac{1}{2}b^{-2}(\tilde Q)p^2(\tilde Q)\bigg)
\tilde Q_\psi^2Q_\varphi\,ds.
\end{aligned}
\end{equation*}
Thus we obtain
\begin{align}\label{1-Qpsipsi-Jsum}
|Q_{\psi\psi}|\leq J_7+J_8+J_9+J_{10},
\end{align}
where
\begin{equation*}
\begin{aligned}
J_7={}&\frac{1}{2}b^{-1/2}(\tilde Q)|U_s-V_s|
+\frac{1}{4}\int_0^\varphi
b^{-3/2}(\tilde Q)p(\tilde Q)
|\tilde Q_\varphi||U_s-V_s|\,ds,\\
J_8={}&\frac{1}{2}\int_0^\varphi
b^{-1}(\tilde Q)p(\tilde Q)
|\tilde Q_{\psi\psi}||Q_\varphi|\,ds,
\end{aligned}
\end{equation*}
\begin{equation*}
\begin{aligned}
J_9={}&\frac{1}{2}\int_0^\varphi
b^{-1}(\tilde Q)p(\tilde Q)
|\tilde Q_\psi||U_s+V_s|\,ds,\\
J_{10}={}&\frac{1}{2}b^{-1/2}(\tilde Q)|F_s||W+Z|
+\frac{1}{4}\int_0^\varphi
b^{-3/2}(\tilde Q)p(\tilde Q)
|\tilde Q_\varphi||F_s||W+Z|\,ds\\
&\quad+\int_0^\varphi
b^{-1}(\tilde Q)p(\tilde Q)
|\tilde Q_\psi||F_s||Q_\varphi|\,ds\\
&\quad+\frac{1}{2}\int_0^\varphi
\bigg|b^{-1}(\tilde Q)p'(\tilde Q)
-\frac{1}{2}b^{-2}(\tilde Q)p^2(\tilde Q)\bigg|
|\tilde Q_\psi|^2|Q_\varphi|\,ds.
\end{aligned}
\end{equation*}
By \eqref{1-eq2.91}, \eqref{1-eq2.9}, \eqref{1-eq3.17}, \eqref{1-eq3.22}, \eqref{1-eqe}, \eqref{1-beta*}, \eqref{1-qq2}, \eqref{1-eq3.61}, \eqref{1-eqr} and \eqref{1-eqe2}--\eqref{1-M2},
\begin{equation*}
  \begin{aligned}
J_7\leq\,&\bigg(
\frac{13\beta_2^{1/4}}{25\hat\mu}
+\frac{3\beta_2}{20\hat\mu\beta_1^{3/4}}
\bigg)(M_4\beta_6+M_5)\sigma\varphi\\
=\,&\bigg(\frac{13\beta_2^{5/4}}{25\beta_1^{5/4}}
+\frac{3\beta_2^2}{20\beta_1^2}\bigg)
(1+\mu_9\mu_{17}\sigma)\beta_6\sigma\varphi
+\bigg(\frac{13\beta_2^{1/4}}{25\hat\mu}
+\frac{3\beta_2}{20\hat\mu\beta_1^{3/4}}\bigg)M_5\sigma\varphi,\\
J_8\leq\,&\frac12\int_0^\varphi
\frac8{15\beta_1}
\beta_2\beta_6\sigma\,ds=\frac{4\beta_2}{15\beta_1}\beta_6\sigma\varphi,\\
J_9\leq\,&\int_0^\varphi
\frac{8}{15\beta_1}
\beta_3(M_2\beta_6+M_3)\sigma\,ds\\
=\,&\frac{8\beta_3}{15\beta_1}
(M_2\beta_6+M_3)\sigma\varphi
=\frac{8\mu_9}{15\beta_1}
(M_2\beta_6+M_3)\sigma^2\varphi,\\
J_{10}\leq\,&
\frac{\mu_{10}\beta_3}{\beta_1}
\bigg(
\frac{13\beta_2^{1/4}}{25\hat\mu}
+\frac{3\beta_2}{20\hat\mu\beta_1^{3/4}}
\bigg)
\bigg(
\frac{\mu_4\sigma}{\beta_1^{1/4}}
+\frac{8\hat\mu\mu_2\beta_3}{9\beta_1^{5/4}}
\bigg)
(1+\mu_6\beta_3)\varphi\\
&\quad+
\bigg(\frac{8\mu_{10}}{15}+\frac{1}{2}\bigg)
\frac{\beta_2\beta_3^2}{\beta_1^2}\varphi
\leq\mu_{19}\sigma^2\varphi.
\end{aligned}
\end{equation*}
Collecting the four estimates yields
\begin{align*}
|Q_{\psi\psi}|\leq(M_{10}\beta_6+M_{11})\sigma\varphi,
\end{align*}
with
\begin{equation}\label{1-M10}
\begin{cases}
\begin{aligned}
\mu_{19}={}&
\frac{\mu_{10}\mu_9}{\hat\beta_1}
\bigg(
\frac{13\hat\beta_2^{1/4}}{25\hat\mu}
+\frac{3\hat\beta_2}{20\hat\mu\hat\beta_1^{3/4}}
\bigg)
\bigg(
\frac{\mu_4}{\hat\beta_1^{1/4}}
+\frac{8\hat\mu\hat\beta_2\mu_9}{9\hat\beta_1^{5/4}}
\bigg)
(1+\mu_6\mu_9\sigma_1)\\
&\quad+
\bigg(\frac{8\mu_{10}}{15}+\frac{1}{2}\bigg)
\frac{\hat\beta_2\mu_9^2}{\hat\beta_1^2},\\
M_{10}={}&
\bigg(
\frac{13\beta_2^{5/4}}{25\beta_1^{5/4}}
+\frac{3\beta_2^2}{20\beta_1^2}
\bigg)(1+\mu_9\mu_{17}\sigma)
+\frac{4\beta_2}{15\beta_1}
+\frac{8\mu_9M_2}{15\beta_1}\sigma,\\
M_{11}={}&
\bigg(
\frac{13\hat\beta_2^{1/4}}{25\hat\mu}
+\frac{3\hat\beta_2}{20\hat\mu\hat\beta_1^{3/4}}
\bigg)\hat M_5
+\bigg(
\frac{8\mu_9\hat M_3}{15\hat\beta_1}
+\mu_{19}
\bigg)\sigma_1.
\end{aligned}\end{cases}\end{equation}}
{Here $\hat M_i$ $(i=2,3,5)$ are obtained from $M_i$ in \eqref{1-M2} by replacing $\beta_1,\beta_2,\sigma$ with $\hat\beta_1,\hat\beta_2,\sigma_1$, respectively. Then $M_i\leq\hat M_i$ $(i=2,3,5)$. Hence \eqref{1-eq3.66}--\eqref{1-eq3.65} are obtained.}
\end{proof}
\begingroup
By \eqref{1-eq3.35},
$(R_0+\mu_6\mu_9)\sigma\leq1/100$. Then one has
$
\frac{\beta_2}{\beta_1}\leq\frac{41}{40}.
$
Using $M_2\leq\hat M_2$ and $\beta_1\geq\hat\beta_1$, if
\begin{equation}\label{1-sigma1-11}
\sigma\leq
\frac{3\hat\beta_1}{
16\mu_9(
15\mu_{17}\hat\beta_1+8\hat M_2
)},
\end{equation}
one has
\begin{align*}
0<M_{10}\leq\frac{39}{40}
+\bigg(
\mu_9\mu_{17}
+\frac{8\mu_9\hat M_2}{15\hat\beta_1}
\bigg)\sigma\leq\frac{39}{40}+\frac{1}{80}
=\frac{79}{80}<1.
\end{align*}
Combining this with \eqref{1-eqe2} and the assumptions
of Proposition \ref{1-pro3.2}, we choose
\begin{equation}\label{1-sigma1}
\sigma_2=\min\bigg\{
\sigma_1,\frac{N_2}{\mu_9},
\frac{3\hat\beta_1}{
16\mu_9(
15\mu_{17}\hat\beta_1+8\hat M_2
)}
\bigg\}
\end{equation}
and
\begin{equation}\label{1-eq3.87}
\beta_4=M_8\beta_6+M_9,\quad
\beta_5=M_6\beta_6+M_7,\quad
\beta_6=80M_{11}.
\end{equation}
Then if $0< \sigma \leq \sigma_2$, it  follows from \eqref{1-eq3.66}--\eqref{1-eq3.65} that
\begin{align*}
|Q_{\varphi\varphi}(\varphi,\psi)|\leq\beta_4\sigma\varphi^{1/2},\ \
|Q_{\varphi\psi}(\varphi,\psi)|\leq\beta_5\sigma,\ \
|Q_{\psi\psi}(\varphi,\psi)|\leq\beta_6\sigma\varphi,\ \ (\varphi,\psi)\in(0,\zeta)\times(0,m).
\end{align*}
Define
\begin{align*}
\mathscr{S}_2=\Big\{&\tilde{Q}\in C^{1,1}([0,{\zeta}]\times[0,m]): {\tilde{Q}\text{~satisfies~}\eqref{1-eq3.16}\text{--}\eqref{1-eq3.17}\text{~with~}\eqref{1-beta*}},\\
&\quad|\tilde{Q}_{\varphi\varphi}(\varphi,\psi)|\leq\beta_4\sigma\varphi^{1/2},
|\tilde{Q}_{\varphi\psi}(\varphi,\psi)|\leq\beta_5\sigma,|\tilde{Q}_{\psi\psi}(\varphi,\psi)|\leq\beta_6\sigma\varphi,
{\tilde Q_\psi(\varphi,0)=\tilde Q_\psi(\varphi,m)=0}\Big\}.
\end{align*}\endgroup
By an argument similar to that in the proof of  Theorem \ref{1-th3.1}, one can prove that
\begin{theorem}
\label{1-th3.2}
Given  $g\in C^4([-R_0\sin\theta_0,R_0\sin\theta_0])$ satisfying \eqref{1-eq2.1} and \eqref{1-eq2.43}. Then if
$\sigma\le \sigma_2$ with $\sigma_2$ given by \eqref{1-sigma1}, the problem \eqref{1-eq2.221}--\eqref{1-eq2.225} admits at least one weak solution
$Q\in C^{1,1}([0,{\zeta}]\times[0,m])$ satisfying \eqref{1-eq3.44} and
\begin{align}\label{1-eq3.89}
|Q_{\varphi\varphi}(\varphi,\psi)|\leq\beta_4\sigma\varphi^{1/2},\ \
|Q_{\varphi\psi}(\varphi,\psi)|\leq\beta_5\sigma,\ \
|Q_{\psi\psi}(\varphi,\psi)|\leq\beta_6\sigma\varphi,\ \
(\varphi,\psi)\in(0,{\zeta})\times(0,m),
\end{align}
where $\beta _i$ $(i= 4,5,6)$ are positive constants given by \eqref{1-eq3.87}.
\end{theorem}
\section{Global continuous sonic-supersonic flows}\label{3glo} \noindent
\par
\begingroup
In this section, we study the problem \eqref{1-eq2.22}--\eqref{1-eq2.26} for $\varphi \in[{\zeta},+\infty)$. Set
$$Q_\zeta(\psi)=Q_{\rm loc}(\zeta,\psi),\quad
G_\zeta(\psi)=\partial_\varphi Q_{\rm loc}(\zeta,\psi),$$
where $Q_{\rm loc}$ is the local solution obtained in Theorem \ref{1-th3.2}. Then the problem can be written as follows:\endgroup
\begin{align}
&Q_{\varphi\varphi}-(b(Q)Q_{\psi})_{\psi}=0,&&(\varphi,\psi)\in({\zeta},+\infty)\times(0,m),
\label{1-eq2.226}  \\
&{Q(\zeta,\psi)=Q_\zeta(\psi)},&&\psi\in(0,m), \label{1-eq2.227}\\
&{Q_\varphi(\zeta,\psi)=G_\zeta(\psi)},&&\psi\in(0,m), \label{1-eq2.228} \\
&Q_\psi(\varphi,0)=0,&&\varphi\in({\zeta},+\infty),\label{1-eq2.229}\\
&Q_{\psi}(\varphi,m)=0,&&\varphi\in({\zeta},+\infty),\label{1-eq2.2210}
\end{align}
which is equivalent to
\begin{align}
&W_{\varphi}+b^{1/2}(Q)W_{\psi}=\frac{1}{4}b^{-1}(Q)p(Q)W(W+Z),&&(\varphi,\psi)\in({\zeta},+\infty)\times(0,m),\label{1-la1}\\
&Z_{\varphi}-b^{1/2}(Q)Z_{\psi}={-\frac{1}{4}b^{-1}(Q)p(Q)Z(W+Z)},&&(\varphi,\psi)\in({\zeta},+\infty)\times(0,m),\label{1-la2}\\
&{W(\zeta,\psi)=G_\zeta(\psi)-b^{1/2}(Q_\zeta(\psi))Q_\zeta^{\prime}(\psi)},&&\psi\in(0,m),\label{1-la3}\\
&{Z(\zeta,\psi)=-G_\zeta(\psi)-b^{1/2}(Q_\zeta(\psi))Q_\zeta^{\prime}(\psi)},&&\psi\in(0,m),\label{1-la4},\\
&W(\varphi,0)+Z(\varphi,0)=0,&&\varphi\in({\zeta},+\infty),\label{1-la5}\\
&W(\varphi,m)+Z(\varphi,m)=0,&&\varphi\in({\zeta},+\infty),\label{1-la6}\\
&Q_{\varphi}(\varphi,\psi)=\frac{1}{2}\left(W(\varphi,\psi)-Z(\varphi,\psi)\right),&& (\varphi,\psi)\in({\zeta},+\infty)\times(0,m),\label{1-la7}\\
&{Q(\zeta,\psi)=Q_\zeta(\psi)},&&\psi\in(0,m).\label{1-la8}
\end{align}
 Then a direct computation yields that
\begin{align}
&\lim_{s\to-\infty}(-s)^{\gamma+1}b(s)=\check{\mu},
\quad\lim_{s\to-\infty}(-s)^{\gamma+2}p(s)=\check{\mu}(\gamma+1),\label{1-eq2.92}\\
&\lim_{s\to-\infty}s b^{-1}(s)p(s)=-(\gamma+1),
\quad\lim_{s\to-\infty}(-s)^{\gamma+3}p'(s)=\check{\mu}(\gamma+1)(\gamma+2),\label{1-eq2.93}
\end{align}
where $\check{\mu}$ is a positive constant depending only on $\gamma$.
\begin{remark}
\label{1-rem Q1}
{\it
According to \eqref{1-qq1} and \eqref{1-qq2}, it holds that
 {
\begin{align}\label{1-eqlaQ1}
-\beta_2\zeta\leq Q_\zeta(\psi)\leq-\beta_1\zeta,
\quad
-\beta_2\leq G_\zeta(\psi)\leq-\beta_1,
\quad\psi\in(0,m).
\end{align}
}
{while one can get from \eqref{1-eq3.21}, \eqref{1-eqe} and \eqref{1-beta*} that}
\begin{align}
\beta_1
\leq -W({\zeta},\psi),Z({\zeta},\psi)\leq
\beta_2,\quad\psi\in(0,m),\label{1-eqlaw1z1}
\end{align}}
where $\beta_i $ $(i=1,2)$ are defined in \eqref{1-beta*}.
\end{remark}
\par It was shown in \cite[Lemma 2.12, Theorems 2.19 and 3.7]{WX15} that, for a given inlet $\Upsilon(y)$ of the  straight nozzle and prescribed inlet speed $|\nabla \varphi(\Upsilon(y),y)|=q_0(y)$, the condition
\begin{align*}
|q'_0(y)|\le\frac{-\Upsilon''(y)}{1+(\Upsilon'(y))^2}\,
\sqrt{\frac{-q_0^2(y)\rho(q_0^2(y))}{\rho(q_0^2(y))+2q_0^2(y)\rho'(q_0^2(y))}}
\end{align*}
guarantees a global smooth supersonic flow that remains away from the sonic and
vacuum states {in any bounded region} and that no shock forms in the  region. By \cite[Remark 3.11]{WX15},
 this inequality in the potential plane is equivalent to
 \begin{align*}
W(0,\psi)=G_{0}(\psi)-b^{1/2}(Q_{0})Q_{0}'(\psi)\le 0,\ \
Z(0,\psi)=-G_{0}(\psi)-b^{1/2}(Q_{0})Q_{0}'(\psi)\ge 0,\ \ \end{align*}
where
$$Q_{0}(\psi)=A(q_{0}(y))\big|_{y=Y_{\mathrm{in}}(\psi)},\ \
G_{0}(\psi)=\left.
\dfrac{\Upsilon''(y)\bigl(1+(\Upsilon'(y))^{2}\bigr)^
{-3/2}}{q_{0}(y)\rho\!\left(q_{0}^{2}(y)\right)}\right|_{y=Y_{\mathrm{in}}(\psi)}.$$
\par Now, the local sonic-supersonic solution obtained in Section \ref{2se} provides the initial data {at $\varphi=\zeta$ for the} supersonic region. Furthermore, the condition \eqref{1-eqlaw1z1} ensures that the supersonic solution can be extended globally. Consequently, by regarding {$\varphi=\zeta$ as the} initial curve, we can apply the global existence theory developed in \cite{WX15} to obtain a continuous {supersonic solution for $\varphi\ge\zeta$. }In the following two subsections, we establish the $C^{0,1}$ and $C^{1,1}$ estimates for the solution of problem \eqref{1-eq2.226}--\eqref{1-eq2.2210}.
\subsection{The $C^{0,1}$ estimate away from the sonic state}\label{3-1-infty}\noindent
\par
\begin{theorem}
\label{1-thl1}
\begingroup Given $(W(\zeta,\cdot),Z(\zeta,\cdot),Q_\zeta(\cdot),G_\zeta(\cdot))
\in C^{0,1}([0,m])\times C^{0,1}([0,m])\times C^{1,1}([0,m])\times C^{0,1}([0,m])$ obtained from the local solution in Theorem \ref{1-th3.2}, the problem \eqref{1-la1}--\eqref{1-la8} admits at least one weak solution
$(W,Z,Q)\in L^{\infty}((\zeta,+\infty)\times(0,m))\times L^{\infty}((\zeta,+\infty)\times(0,m))\times C_{\rm{loc}}^{1}([\zeta,+\infty)\times[0,m])$ satisfying
\begin{align}
&-\tilde{\beta}_1\sigma+{\mu_0}
\leq -W(\varphi,\psi),Z(\varphi,\psi)\leq
\tilde{\beta}_2\sigma+{\mu_0},
\quad(\varphi,\psi)\in({\zeta},+\infty)\times(0,m),\label{1-lath1}\\
&-\tilde{\beta}_2\sigma\varphi \leq Q(\varphi,\psi)+{\mu_0\varphi}\leq \tilde{\beta}_1\sigma\varphi,\ \ (\varphi,\psi)\in({\zeta},+\infty)\times(0,m),\label{1-lath2}\\
&-\tilde{\beta}_2\sigma\leq Q_{\varphi}(\varphi,\psi)+{\mu_0}\leq \tilde{\beta}_1\sigma,\ \
\left|Q_{\psi}(\varphi,\psi)\right|\leq\beta_7\sigma\varphi^{(\gamma+1)/2},\ \ (\varphi,\psi)\in({\zeta},+\infty)\times(0,m),\label{1-lath3}
\end{align}
where $\tilde{\beta}_i$ $(i=1,2,3)$ are defined in \eqref{1-eq3.34} and $\beta_7$ is a positive constant to be determined below.
\end{theorem}\endgroup
\begin{proof}
{ According to Theorem 3.8 in \cite{WX15}, the problem \eqref{1-la1}--\eqref{1-la8} admits a unique weak solution
$(W,Z,Q)\in L^\infty((\zeta,+\infty)\times(0,m))
\times L^\infty((\zeta,+\infty)\times(0,m))
\times C_{\rm loc}^{1}([\zeta,+\infty)\times[0,m])$.} It remains to derive the estimates \eqref{1-lath1}--\eqref{1-lath3}. As mentioned before,
$$
\beta_1\ge{\mu_0}-\tilde{\beta}_1\sigma,\quad\beta_2\le{\mu_0}+\tilde{\beta}_2\sigma.
$$
It suffices to prove
\begin{align*}
\beta_1
\leq -W(\varphi,\psi),Z(\varphi,\psi)\leq
\beta_2.
\end{align*}
Then
\eqref{1-lath1}--\eqref{1-lath3} can be obtained. It follows from \eqref{1-eqlaw1z1} that
\begin{align}
\label{1-lam}
\beta_1
\leq -W({\zeta},\psi),Z({\zeta},\psi)\leq
\beta_2.
\end{align}
We first prove the upper bound
\begin{align}
\label{1-lag1}
-W(\varphi,\psi),Z(\varphi,\psi)\le \beta_2,
\quad(\varphi,\psi)\in({\zeta},+\infty)\times(0,m).
\end{align}
If \eqref{1-lag1} is invalid, then one of the following two cases occurs:
\begin{enumerate}[ \rm (i)]
\item
{There exist $\check{\varphi}_+\ge\hat{\varphi}_+\ge\zeta$ and $\overline{M}_+>\beta_2$,
as well as a positive characteristic
$$
\Sigma_+:\Psi'_+(\varphi)=b^{1/2}(Q(\varphi,\Psi_+(\varphi))),
\quad\hat{\varphi}_+<\varphi<\check{\varphi}_+
$$
such that $W(\check{\varphi}_+,\Psi_+(\check{\varphi}_+))=-\overline{M}_+$, $W(\cdot,\Psi_+(\cdot))\ge-\overline M_+$ and
$Z(\cdot,\Psi_+(\cdot))\le \overline{M}_+$ on $[\hat{\varphi}_+,\check{\varphi}_+]$.}
\item
{There exist $\check{\varphi}_-\ge\hat{\varphi}_-\ge\zeta$ and $\overline{M}_->\beta_2$,
as well as a negative characteristic
$$
\Sigma_-:\Psi'_-(\varphi)=-b^{1/2}(Q(\varphi,\Psi_-(\varphi))),
\quad\hat{\varphi}_-<\varphi<\check{\varphi}_-
$$
such that $Z(\check{\varphi}_-,\Psi_-(\check{\varphi}_-))=\overline{M}_-$, $Z(\cdot,\Psi_-(\cdot))\leq\overline{M}_-$ and
$W(\cdot,\Psi_-(\cdot))\ge-\overline{M}_-$ on $[\hat{\varphi}_-,\check{\varphi}_-]$.}
\end{enumerate}
Without loss of generality, we only consider case (i), case (ii) can be treated in a similar way.
On $\Sigma_+$, $W$ satisfies
\begin{align}
\label{1-lag2}
\frac{d}{d\varphi}W(\varphi,\Psi_{+}(\varphi))
=R_+(\varphi)W(\varphi,\Psi_{+}(\varphi))
(W(\varphi,\Psi_{+}(\varphi))+Z(\varphi,\Psi_{+}(\varphi))),
\quad\hat{\varphi}_+<\varphi<\check{\varphi}_+,
\end{align}
where
\begin{align*}
R_+(\varphi)=\frac14b^{-1}(Q(\varphi,\Psi_{+}(\varphi)))
p(Q(\varphi,\Psi_{+}(\varphi))),\quad
\hat{\varphi}_+<\varphi<\check{\varphi}_+.
\end{align*}
Denote $\overline{H}_+(\varphi)=W(\varphi,\Psi_{+}(\varphi))+\overline{M}_+$, then ${\overline{H}_+(\check{\varphi}_+)}=0$ and
\begin{align*}
&\frac{d}{d\varphi}{\overline{H}_+(\varphi)}\\
=&\,R_+(\varphi)W(\varphi,\Psi_{+}(\varphi)){\overline{H}_+(\varphi)}+R_+(\varphi)W(\varphi,\Psi_{+}(\varphi))(Z(\varphi,\Psi_{+}(\varphi))-\overline{M}_+),
\quad\hat{\varphi}_+<\varphi<\check{\varphi}_+,
\end{align*}
from which one gets that
\begin{align*}
&\bigg({\overline{H}_+(\varphi)}\exp\bigg\{\int_{\hat{\varphi}_+}^{\varphi}-R_+(s)W(s,\Psi_+(s))ds\bigg\}\bigg)'
\nonumber\\
=&R_+(\varphi)W(\varphi,\Psi_{+}(\varphi))(Z(\varphi,\Psi_{+}(\varphi))-\overline{M}_+)\exp\bigg\{\int_{\hat{\varphi}_+}^{\varphi}-R_+(s)W(s,\Psi_+(s))ds\bigg\},
\quad\hat{\varphi}_+<\varphi<\check{\varphi}_+.
\end{align*}
Since $W \equiv 0$ is a trivial solution to \eqref{1-lag2}, the uniqueness of ODE solutions implies that $W$ cannot cross zero. Thus, $W(\check{\varphi}_+,\Psi_+(\check{\varphi}_+))=-\overline{M}_+<0$ guarantees $W(\cdot,\Psi_+(\cdot))<0$ on $[\hat{\varphi}_+,\check{\varphi}_+]$.
Combining $R_+>0$ and $Z(\cdot,\Psi_+(\cdot))-\overline{M}_+\leq 0$ on $[\hat{\varphi}_+,\check{\varphi}_+]$, we obtain
\begin{align*}
&\bigg({\overline{H}_+(\varphi)}
\exp\bigg\{\int_{{\hat{\varphi}_+}}^{\varphi}-R_+(s)W(s,\Psi_+(s))ds\bigg\}\bigg)'\ge 0,
\quad\hat{\varphi}_+<\varphi<\check{\varphi}_+.
\end{align*}
Therefore,
$$
{\overline{H}_+(\hat{\varphi}_+)}\leq 0, \ \
W(\hat{\varphi}_+,\Psi_+(\hat{\varphi}_+))\leq -\overline{M}_+.
$$
This, together with $W(\cdot,\Psi_+(\cdot))\ge -\overline{M}_+$,  obtains
$$
W(\hat{\varphi}_+,\Psi_+(\hat{\varphi}_+))= -\overline{M}_+.
$$
Denote $\overline{H}_-(\varphi,\Psi_-(\varphi))=\overline{M}_--Z(\varphi,\Psi_-(\varphi))$, by the similar argument, one has
\begin{align*}
Z(\hat{\varphi}_-,\Psi_-(\hat{\varphi}_-))=\overline{M}_-.
\end{align*}
Now we prove the contradiction by the method of characteristics. { Choose $\varphi_0>\zeta$ and $\psi_0\in[0,m]$ such that
\[
\overline M_+=\overline M_-=
\max_{[\zeta,\varphi_0]\times[0,m]}\{-W,Z\}>\beta_2,
\]
and either $W(\varphi_0,\psi_0)=-\overline M_+$
or $Z(\varphi_0,\psi_0)=\overline M_-$.  Then we define $(\varphi_j,\psi_j)$ and $(\varphi_j^*,\psi_j^*)$ on $[\zeta,+\infty)$ similarly as in the proof of
Proposition \ref{1-pro3.1}. There exist nonnegative integer $k$ and $k^*$ such that
$$
\varphi_0\geq\varphi_1\geq\cdots\geq\varphi_{k+1}=\zeta,\quad
\varphi_0\geq\varphi_1^*\geq\cdots\geq\varphi_{k^*+1}^*=\zeta.
$$
For the positive characteristic,
$$
\left\{
\begin{aligned}
&\Psi_j'(\varphi)
=(-1)^{j-1}b^{1/2}(Q(\varphi,\Psi_j(\varphi))),
\quad\varphi_j<\varphi<\varphi_{j-1},
\\
&\Psi_j(\varphi_j)=\psi_j,\quad
\Psi_j(\varphi_{j-1})=\psi_{j-1},
\end{aligned}
\right.
\quad1\le j\le k+1,
$$
and
$$
\psi_j=\left\{
\begin{aligned}
&0,\quad1\le j\le k\mbox{ and $j$ is odd},
\\
&m,\quad1\le j\le k\mbox{ and $j$ is even},
\end{aligned}
\right.
\quad0\le\psi_{k+1}\le m.
$$
For the negative characteristic,
$$
\left\{
\begin{aligned}
&(\Psi_j^*)'(\varphi)
=(-1)^j b^{1/2}(Q(\varphi,\Psi_j^*(\varphi))),
\quad\varphi_j^*<\varphi<\varphi_{j-1}^*,
\\
&\Psi_j^*(\varphi_j^*)=\psi_j^*,\quad
\Psi_j^*(\varphi_{j-1}^*)=\psi_{j-1}^*,
\end{aligned}
\right.
\quad1\le j\le k^*+1,
$$
and
$$
\psi_j^*=\left\{
\begin{aligned}
&m,\quad1\le j\le k^*\mbox{ and $j$ is odd},
\\
&0,\quad1\le j\le k^*\mbox{ and $j$ is even},
\end{aligned}
\right.
\quad0\le\psi_{k^*+1}^*\le m.
$$}
If $k$ {and $k^*$ are} even, it holds that
{
\begin{align*}
-\overline M_+
&=W(\varphi_0,\psi_0)=W(\varphi_1,\psi_1)
=-Z(\varphi_1,\psi_1)=\cdots=W(\zeta,\psi_{k+1}),\\
\overline M_-
&=Z(\varphi_0,\psi_0)=Z(\varphi_1^*,\psi_1^*)
=-W(\varphi_1^*,\psi_1^*)=\cdots=Z(\zeta,\psi_{k^*+1}^*).
\end{align*}}
If $k$ {and $k^*$ are} odd,
{
\begin{align*}
-\overline M_+
&=W(\varphi_0,\psi_0)=W(\varphi_1,\psi_1)
=-Z(\varphi_1,\psi_1)=\cdots=-Z(\zeta,\psi_{k+1}),\\
\overline M_-
&=Z(\varphi_0,\psi_0)=Z(\varphi_1^*,\psi_1^*)
=-W(\varphi_1^*,\psi_1^*)=\cdots=-W(\zeta,\psi_{k^*+1}^*).
\end{align*}
Both cases contradict \eqref{1-lam} and $\overline M_\pm>\beta_2$.}
For the lower bound, one of the following two cases occurs:
\begin{enumerate}[ \rm (i)]
\item
 {There exist $\check{\varphi}_+\ge\hat{\varphi}_+\ge\zeta$} and $\underline{M}_+<\beta_1$,
as well as a positive characteristic
$$
\Sigma_+:\Psi'_+(\varphi)=b^{1/2}(Q(\varphi,\Psi_+(\varphi))),
\quad\hat{\varphi}_+<\varphi<\check{\varphi}_+
$$
such that $W(\check{\varphi}_+,\Psi_+(\check{\varphi}_+))=-\underline{M}_+$ and
$Z(\cdot,\Psi_+(\cdot))\ge \underline{M}_+$ on $[\hat{\varphi}_+,\check{\varphi}_+]$.
\item
{There exist $\check{\varphi}_-\ge\hat{\varphi}_-\ge\zeta$} and $\underline{M}_-<\beta_1$,
as well as a negative characteristic
$$
\Sigma_-:\Psi'_-(\varphi)=-b^{1/2}(Q(\varphi,\Psi_-(\varphi))),
\quad\hat{\varphi}_-<\varphi<\check{\varphi}_-
$$
such that $Z(\check{\varphi}_-,\Psi_-(\check{\varphi}_-))=\underline{M}_-$ and
$W(\cdot,\Psi_-(\cdot))\le-\underline{M}_-$ on $[\hat{\varphi}_-,\check{\varphi}_-]$.
\end{enumerate}
\par
Denote ${\underline{H}_+(\varphi)}=W(\varphi,\Psi_+(\varphi))+\underline{M}_+$
and ${\underline{H}_-(\varphi)}=Z(\varphi,\Psi_-(\varphi))-\underline{M}_-$.  By  the same way, one can get that
\begin{align*}
W(\hat{\varphi}_+,\Psi_+(\hat{\varphi}_+))=-\underline{M}_+,\quad Z(\hat{\varphi}_-,\Psi_-(\hat{\varphi}_-))=\underline{M}_-.
\end{align*}
Hence \eqref{1-lath1} holds.
Then it follows from \eqref{1-la7}, \eqref{1-la8}, \eqref{1-eqlaQ1} and \eqref{1-lath1} that
\begingroup\begin{align}
&-\beta_2\leq Q_\varphi(\varphi,\psi)\leq-\beta_1,
&&(\varphi,\psi)\in[\zeta,+\infty)\times[0,m],\nonumber\\
&-\beta_2\varphi
\leq Q(\varphi,\psi)
=Q_\zeta(\psi)+\int_\zeta^\varphi Q_\varphi(s,\psi)\,ds
\leq-\beta_1\varphi,
&&(\varphi,\psi)\in[\zeta,+\infty)\times[0,m].
\label{1-qqla1}
\end{align}\endgroup
Finally, by \eqref{1-jie1}, \eqref{1-eq2.92}, \eqref{1-lath1} and \eqref{1-qqla1}, it holds that
{\begin{align*}
\left|Q_{\psi}(\varphi,\psi)\right|
\leq\frac{1}{2}b^{-1/2}(Q(\varphi,\psi))\left|W(\varphi,\psi)+Z(\varphi,\psi)\right|
\leq\beta_7\sigma\varphi^{(\gamma+1)/2},\quad(\varphi,\psi)\in(\zeta,+\infty)\times(0,m),
\end{align*}}
where
\begin{align}\label{1-beta7}
\beta_7=\mu_{20}\hat{\beta}_2^{(\gamma+1)/2}\left(\tilde{\beta}_1+\tilde{\beta}_2\right),
\end{align}
  with  $\tilde{\beta}_i$ $(i=1,2)$ and $\hat{\beta}_2$  given by
 \eqref{1-eq3.34} and \eqref{1-jie1}, respectively, {and $\mu_{20}$ is a positive constant depending only on $\gamma$, $\zeta$ and $R_0$. }Therefore, \eqref{1-lath2} and \eqref{1-lath3} can be obtained.
\end{proof}
\subsection{The $C^{1,1}$ estimate away from the sonic state}\label{3-2-infty}\noindent
\par
{In this subsection,  we prove that the global solution
obtained in Theorem \ref{1-thl1} belongs to
$C_{\rm loc}^{1,1}([\zeta,+\infty)\times[0,m])$. To this end, we consider  the following nonlinear problem:
\begin{align}
&W_\varphi+b^{1/2}(Q)W_\psi=-\frac{1}{2}b^{-1/2}(Q)p(Q)Q_\psi W,
&&(\varphi,\psi)\in(\zeta,+\infty)\times(0,m),\label{1-eqla3.7}\\
&Z_\varphi-b^{1/2}(Q)Z_\psi=\frac{1}{2}b^{-1/2}(Q)p(Q)Q_\psi Z,
&&(\varphi,\psi)\in(\zeta,+\infty)\times(0,m),\label{1-eqla3.8}\\
&W(\zeta,\psi)=G_\zeta(\psi)-b^{1/2}(Q_\zeta(\psi))Q_\zeta'(\psi),
&&\psi\in(0,m),\label{1-eqla3.9}\\
&Z(\zeta,\psi)=-G_\zeta(\psi)-b^{1/2}(Q_\zeta(\psi))Q_\zeta'(\psi),
&&\psi\in(0,m),\label{1-eqla3.10}\\
&W(\varphi,0)+Z(\varphi,0)=0,&&\varphi\in(\zeta,+\infty),\label{1-eqla3.11}\\
&W(\varphi,m)+Z(\varphi,m)=0,&&\varphi\in(\zeta,+\infty),\label{1-eqla3.12}\\
&Q_{\varphi}(\varphi,\psi)=\frac{1}{2}\left(W(\varphi,\psi)-Z(\varphi,\psi)\right),&& (\varphi,\psi)\in(\zeta,+\infty)\times(0,m),\label{1-eqla3.13}\\
&Q(\zeta,\psi)=Q_\zeta(\psi),&&\psi\in(0,m).\label{1-eqla3.14}
\end{align}}
{By Theorem \ref{1-th3.2}, the inlet data satisfy
\begin{equation}\label{1-eq3.61l}
Q_\zeta\in C^{1,1}([0,m]),\quad G_\zeta\in C^{0,1}([0,m]),\quad
Q_\zeta'(0)=Q_\zeta'(m)=0.
\end{equation}
By \eqref{1-eq3.61l} and the local theory
for strictly hyperbolic systems, the global solution $Q$
obtained in Theorem \ref{1-thl1} belongs to
$C^{1,1}([\zeta,\hat\zeta]\times[0,m])$
for some $\hat\zeta>\zeta$.
Define
\[
\zeta_{\max}:=
\sup\left\{
\tilde\zeta>\zeta:
Q\in C^{1,1}([\zeta,\tilde\zeta]\times[0,m])
\right\}.
\]
Then $\zeta_{\max}>\zeta$.
We will show that $\zeta_{\max}=+\infty$, and hence
\[
Q\in C_{\rm loc}^{1,1}([\zeta,+\infty)\times[0,m]).
\]
\par 
}
\begin{proposition}\label{1-prola3.2}
{
Assume $0<\sigma\leq\sigma_3$ with $\sigma_3$ to be determined below.
Let $(W,Z,Q)$ be the global solution satisfying
\eqref{1-eqla3.7}--\eqref{1-eqla3.14},
and let $\zeta_{\max}$ be defined above.
Then the following estimates hold almost everywhere:
\begin{align}
&\left|b^{1/2}(Q)(\varphi,\psi)W_{\varphi}(\varphi,\psi)\right|\leq
M_{14}\sigma\varphi^{-(\gamma+3)/2},\quad &&(\varphi,\psi)\in(\zeta,{\zeta_{\max}})\times(0,m),\label{1-eqla3.62}\\
&\left|b^{1/2}(Q)(\varphi,\psi)Z_{\varphi}(\varphi,\psi)\right|\leq
M_{14}\sigma\varphi^{-(\gamma+3)/2},\quad &&(\varphi,\psi)\in(\zeta,{\zeta_{\max}})\times(0,m),\label{1-eqla3.621}\\
&\left|Q_{\varphi\psi}(\varphi,\psi)\right|\leq
M_{15}\sigma\varphi^{(\gamma-1)/2},\ \ \ &&(\varphi,\psi)\in(\zeta,{\zeta_{\max}})\times(0,m),\label{1-eqla3.66}\\
&\left|Q_{\varphi\varphi}(\varphi,\psi)\right|\leq
M_{16}\sigma\varphi^{-1},\quad &&(\varphi,\psi)\in(\zeta,{\zeta_{\max}})\times(0,m),\label{1-eqla3.64}\\
&\left|Q_{\psi\psi}(\varphi,\psi)\right|\leq
M_{17}\sigma\varphi^{\gamma},\quad &&(\varphi,\psi)\in(\zeta,{\zeta_{\max}})\times(0,m).\label{1-eqla3.65}
\end{align}
Here $M_{14}$ is a positive constant depending only on $\sigma$, $\gamma$, $R_0$ and $\zeta$, while $M_i$ $(i=15,16,17)$ are independent of $\sigma$. All these constants are independent of ${\zeta_{\max}}$.}
\end{proposition}

\begin{proof}
 {Set
 $$
 (U_l,V_l)(\varphi,\psi)=\left(b^{1/2}(Q)W_\varphi,b^{1/2}(Q)Z_\varphi\right)(\varphi,\psi),
 \quad {(\varphi,\psi)\in (\zeta,{\zeta_{\max}})\times(0,m).}
 $$
 Then $(U_l,{V_l})$ solves the following problem:
\begin{align}
&\partial_\varphi U_l+b^{1/2}(Q)\partial_\psi U_l
=R_{l1}U_l+R_{l3}WV_l+F_lW,
&&(\varphi,\psi)\in(\zeta,{\zeta_{\max}})\times(0,m),\label{1-daet1}\\
&\partial_\varphi V_l-b^{1/2}(Q)\partial_\psi V_l
=R_{l2}V_l-R_{l3}ZU_l-F_lZ,
&&(\varphi,\psi)\in(\zeta,{\zeta_{\max}})\times(0,m),\label{1-daet2}\\
&U_l(\zeta,\psi)=b(Q_\zeta(\psi))U_s(\zeta,\psi),
&&\psi\in(0,m),\label{1-daet3}\\
&V_l(\zeta,\psi)=b(Q_\zeta(\psi))V_s(\zeta,\psi),
&&\psi\in(0,m),\label{1-daet4}\\
&U_l(\varphi,0)+V_l(\varphi,0)=0,&&\varphi\in(\zeta,{\zeta_{\max}}),\label{1-daet5}\\
&U_l(\varphi,m)+V_l(\varphi,m)=0,&&\varphi\in(\zeta,{\zeta_{\max}}),\label{1-daet6}
\end{align}}
where
{\begin{align*}
&R_{l1}(\varphi,\psi)=b^{-1}(Q(\varphi,\psi))p(Q(\varphi,\psi))
\bigg(Q_\varphi(\varphi,\psi)+\frac{1}{4}W(\varphi,\psi)\bigg),\\
&R_{l2}(\varphi,\psi)=b^{-1}(Q(\varphi,\psi))p(Q(\varphi,\psi))
\bigg(Q_\varphi(\varphi,\psi)-\frac{1}{4}Z(\varphi,\psi)\bigg),\\
&R_{l3}(\varphi,\psi)=\frac{1}{4}b^{-1}(Q(\varphi,\psi))p(Q(\varphi,\psi)),\\
&F_l(\varphi,\psi)=\bigg(\frac{3}{4}b^{-1}(Q(\varphi,\psi))p^2(Q(\varphi,\psi))
-\frac{1}{2}p'(Q(\varphi,\psi))\bigg)
Q_\varphi(\varphi,\psi)Q_\psi(\varphi,\psi),\\
&U_l(\zeta,\psi)=-\frac{1}{2}p(Q_\zeta(\psi))Q_\zeta'(\psi)W(\zeta,\psi)-b(Q_\zeta(\psi))W_\psi(\zeta,\psi),\\
&V_l(\zeta,\psi)=\frac{1}{2}p(Q_\zeta(\psi))Q_\zeta'(\psi)Z(\zeta,\psi)+b(Q_\zeta(\psi))Z_\psi(\zeta,\psi).
\end{align*}}
{By \eqref{1-eq2.93}, one can choose $\zeta_l>\zeta$, depending only on the fixed parameters, sufficiently large such that
$$
sb^{-1}(-s)p(-s)\geq\frac{7\gamma+9}{8},
\quad s\geq\hat\beta_1\zeta_l.
$$}
It follows from \eqref{1-eq3.62}, \eqref{1-eq3.621}, \eqref{1-lath1}, \eqref{1-lath3} and \eqref{1-qqla1} that
{\begin{align}
&\max\{R_{l1}(\varphi,\psi),R_{l2}(\varphi,\psi)\}\leq0,\ \ \varphi\in[\zeta,{\zeta_{\max}}),\nonumber\\
&\max\{R_{l1}(\varphi,\psi),R_{l2}(\varphi,\psi)\}\leq b^{-1}(Q(\varphi,\psi))p(Q(\varphi,\psi))Q_\varphi(\varphi,\psi)\leq-M_{12}\varphi^{-1},\ \ \varphi\in[\zeta_l,{\zeta_{\max}}),\nonumber\\
&R_{l3}(\varphi,\psi)>0,\ \ R_{l3}(\varphi,\psi)\max\{-W(\varphi,\psi),Z(\varphi,\psi)\}
\leq\frac{\mu_{21}\beta_2}{4\beta_1}\varphi^{-1},\ \ (\varphi,\psi)\in[\zeta,{\zeta_{\max}})\times(0,m),\nonumber\\
&|F_l(\varphi,\psi)|
\leq M_{13}\sigma\varphi^{-(\gamma+5)/2},\ \ (\varphi,\psi)\in[\zeta,{\zeta_{\max}})\times(0,m),\nonumber\\
&\max\{|U_l(\zeta,\psi)|,|V_l(\zeta,\psi)|\}
\leq\mu_{21}(M_2\beta_6+M_3)\sigma,\ \ \psi\in(0,m),\label{1-larfu}
\end{align}
where
\begin{align}\label{1-M12}
M_{12}=\frac{(7\gamma+9)\beta_1}{8\beta_2},\qquad
M_{13}=\frac{\mu_{21}\beta_2\beta_7}{\beta_1^{\gamma+3}}.
\end{align}
Here $\mu_{21}$ is a positive constant depending only on $\zeta$, $\hat\beta_1$, $\gamma$ and $R_0$, and $\beta_i$ $(i=1,2)$, $\beta_6$, $\beta_7$ and $M_j$ $(j=2,3)$ are given by \eqref{1-beta*}, \eqref{1-eq3.87}, \eqref{1-beta7} and \eqref{1-M2}, respectively.}
{\par
Let $H_l(\varphi)\geq0$ satisfy
\begin{align}\label{1-char-Hl}
H_l(\zeta)=\mu_{21}(M_2\beta_6+M_3)\sigma, \
H_l'(\varphi)=
\begin{cases}
M_{13}\beta_2\sigma\varphi^{-(\gamma+5)/2},&\zeta<\varphi<\zeta_l,\\
-M_{12}\varphi^{-1}H_l(\varphi)+M_{13}\beta_2\sigma\varphi^{-(\gamma+5)/2},&\varphi>\zeta_l.
\end{cases}
\end{align}
The function $H_l$ is continuous at $\zeta_l$. By \eqref{1-larfu}, the following inequality holds almost everywhere:
\begin{align*}
H_l'(\varphi)
&\geq b^{-1}(Q(\varphi,\psi))p(Q(\varphi,\psi))Q_\varphi(\varphi,\psi)H_l(\varphi)\\
&\quad+\max\{|F_l(\varphi,\psi)W(\varphi,\psi)|,
|F_l(\varphi,\psi)Z(\varphi,\psi)|\},\ \
 (\varphi,\psi)\in(\zeta,{\zeta_{\max}})\times(0,m).
\end{align*}
We next prove that
\begin{equation}\label{1-eqHUU}
\max\{|U_l(\varphi,\psi)|,|V_l(\varphi,\psi)|\}\leq H_l(\varphi),
\quad \text{a.e. }(\varphi,\psi)\in[\zeta,{\zeta_{\max}})\times[0,m],
\end{equation}
and then use the defining equation of $H_l$ to obtain
\eqref{1-eqla3.62} and \eqref{1-eqla3.621}. We first prove the estimates for $U_l$ and $-V_l$, and then
apply the same argument to $-U_l$ and $V_l$.
\par We prove \eqref{1-eqHUU} by contradiction. For $\varepsilon>0$, let $H_{l,\varepsilon}$ satisfy
\begin{align*}
H_{l,\varepsilon}(\zeta)=H_l(\zeta)+\varepsilon,\
H_{l,\varepsilon}'(\varphi)=
\begin{cases}
M_{13}\beta_2\sigma\varphi^{-(\gamma+5)/2}+\varepsilon,
&\zeta<\varphi<\zeta_l,\\
-M_{12}\varphi^{-1}H_{l,\varepsilon}(\varphi)
+M_{13}\beta_2\sigma\varphi^{-(\gamma+5)/2}+\varepsilon,
&\varphi>\zeta_l.
\end{cases}
\end{align*}
The function $H_{l,\varepsilon}$ is continuous at $\zeta_l$. Suppose that \eqref{1-eqHUU} does not hold.
For sufficiently small $\varepsilon>0$, the strict
inequality at $\varphi=\zeta$ and continuity imply
that there exists $\varphi_0\in(\zeta,{\zeta_{\max}})$
such that
\begin{align}\label{1-char-first-hit}
&\max\bigg\{
\|U_l(\varphi_0,\cdot)\|_{L^\infty(0,m)},
\|V_l(\varphi_0,\cdot)\|_{L^\infty(0,m)}
\bigg\}=H_{l,\varepsilon}(\varphi_0),\nonumber\\
&\max\bigg\{
\|U_l(\varphi,\cdot)\|_{L^\infty(0,m)},
\|V_l(\varphi,\cdot)\|_{L^\infty(0,m)}
\bigg\}<H_{l,\varepsilon}(\varphi),\quad\zeta\leq\varphi<\varphi_0.
\end{align}
\par Assume that
\begin{align*}
\Sigma_\pm:\quad&\Psi_\pm'(\varphi)=\pm b^{1/2}(Q(\varphi,\Psi_\pm(\varphi))),\\
&0<\Psi_\pm(\varphi)<m,\quad\hat\varphi_\pm<\varphi<\check\varphi_\pm,
\quad\zeta\leq\hat\varphi_\pm<\check\varphi_\pm\leq\varphi_0,
\end{align*}
are positive and negative characteristics. On $\Sigma_+$, if $\hat\varphi_+\leq\varphi\leq\check\varphi_+$, $U_l$ satisfies
\begin{align*}
\frac{d}{d\varphi}U_l(\varphi,\Psi_+(\varphi))
=\,&R_{l1}(\varphi,\Psi_+(\varphi))
U_l(\varphi,\Psi_+(\varphi))\\[-3pt]
&\quad+W(\varphi,\Psi_+(\varphi))\bigg(R_{l3}(\varphi,\Psi_+(\varphi))
V_l(\varphi,\Psi_+(\varphi))+F_l(\varphi,\Psi_+(\varphi))\bigg),
\end{align*}
which yields
\begin{align*}
\frac{d}{d\varphi}\left(U_l(\varphi,\Psi_+(\varphi))-H_{l,\varepsilon}(\varphi)\right)
=\,&R_{l1}(\varphi,\Psi_+(\varphi))
(U_l(\varphi,\Psi_+(\varphi))-H_{l,\varepsilon}(\varphi))\\
&\quad-R_{l3}(\varphi,\Psi_+(\varphi))W(\varphi,\Psi_+(\varphi))
(-V_l(\varphi,\Psi_+(\varphi))-H_{l,\varepsilon}(\varphi))\\
&\quad+b^{-1}(Q(\varphi,\Psi_+(\varphi)))p(Q(\varphi,\Psi_+(\varphi)))
Q_\varphi(\varphi,\Psi_+(\varphi))H_{l,\varepsilon}(\varphi)\\
&\quad+F_l(\varphi,\Psi_+(\varphi))W(\varphi,\Psi_+(\varphi))
-H_{l,\varepsilon}'(\varphi).
\end{align*}
By \eqref{1-char-first-hit}, we have $-V_l(\varphi,\Psi_+(\varphi))-H_{l,\varepsilon}(\varphi)\leq0$, while
\begin{align*}
&b^{-1}(Q(\varphi,\Psi_+(\varphi)))p(Q(\varphi,\Psi_+(\varphi)))
Q_\varphi(\varphi,\Psi_+(\varphi))H_{l,\varepsilon}(\varphi)\\
&\quad+F_l(\varphi,\Psi_+(\varphi))W(\varphi,\Psi_+(\varphi))
-H_{l,\varepsilon}'(\varphi)\leq-\varepsilon.
\end{align*}
Then for any $\hat{\varphi}_+\leq\varphi\leq\check{\varphi}_+$,
\begin{align}\label{1-pos-arc-bound}
&U_l(\check\varphi_+,\Psi_+(\check\varphi_+))-H_{l,\varepsilon}(\check\varphi_+)\nonumber\\
\leq\,&\big(U_l(\hat\varphi_+,\Psi_+(\hat\varphi_+))-H_{l,\varepsilon}(\hat\varphi_+)\big)
\exp\bigg\{\int_{\hat\varphi_+}^{\check\varphi_+}R_{l1}(s,\Psi_+(s))\,ds\bigg\}.
\end{align}
Similarly, on $\Sigma_-$, if $\hat\varphi_-\leq\varphi\leq\check\varphi_-$, $V_l$ satisfies
\begin{align*}
\frac{d}{d\varphi}V_l(\varphi,\Psi_-(\varphi))
&=R_{l2}(\varphi,\Psi_-(\varphi))V_l(\varphi,\Psi_-(\varphi))\\[-3pt]
&\quad-Z(\varphi,\Psi_-(\varphi))
\bigg(R_{l3}(\varphi,\Psi_-(\varphi))U_l(\varphi,\Psi_-(\varphi))
+F_l(\varphi,\Psi_-(\varphi))\bigg),
\end{align*}
which yields
\begin{align*}
\frac{d}{d\varphi}(-V_l(\varphi,\Psi_-(\varphi))-H_{l,\varepsilon}(\varphi))
&=R_{l2}(\varphi,\Psi_-(\varphi))
(-V_l(\varphi,\Psi_-(\varphi))-H_{l,\varepsilon}(\varphi))\\
&\quad+R_{l3}(\varphi,\Psi_-(\varphi))Z(\varphi,\Psi_-(\varphi))
\bigg(U_l(\varphi,\Psi_-(\varphi))-H_{l,\varepsilon}(\varphi)\bigg)\\
&\quad+b^{-1}(Q(\varphi,\Psi_-(\varphi)))p(Q(\varphi,\Psi_-(\varphi)))
Q_\varphi(\varphi,\Psi_-(\varphi))H_{l,\varepsilon}(\varphi)\\
&\quad+F_l(\varphi,\Psi_-(\varphi))Z(\varphi,\Psi_-(\varphi))
-H_{l,\varepsilon}'(\varphi).
\end{align*}
Similarly, by \eqref{1-char-first-hit}, for any $\hat{\varphi}_-\leq\varphi\leq\check{\varphi}_-$, one obtains
\begin{align}\label{1-neg-arc-bound}
-V_l(\check\varphi_-,\Psi_-(\check\varphi_-))-H_{l,\varepsilon}
(\check\varphi_-)
\leq\big(-V_l(\hat\varphi_-,\Psi_-(\hat\varphi_-))-H_{l,\varepsilon}(\hat\varphi_-)\big)
\exp\bigg\{\int_{\hat\varphi_-}^{\check\varphi_-}R_{l2}(s,\Psi_-(s))\,ds\bigg\}.
\end{align}
\par Now, we estimate $(U_l,-V_l)$ by the method of characteristics. Define functions $r_3(s)$ $(\zeta\leq s<\varphi)$ and $r_3^*(s)$ $(\zeta\leq s<\varphi)$ as follows:
\begin{equation*}
\begin{aligned}
r_3(s)=\left\{
\begin{aligned}
&R_{l1}(s,\Psi_j(s)),\quad\varphi_{j}\le s<\varphi_{j-1},\, 1\le j\le k+1\mbox{ and $j$ is odd},
\\[2.5 mm]
&R_{l2}(s,\Psi_j(s)),\quad\varphi_{j}\le s<\varphi_{j-1},\,1\le j\le k+1\mbox{ and $j$ is even},
\end{aligned}
\right.\\
r_3^*(s)=\left\{
\begin{aligned}
&R_{l2}(s,\Psi_j^*(s)),\quad\varphi_{j}^*\le s<\varphi_{j-1}^*,\, 1\le j\le {k^*}+1\mbox{ and $j$ is odd},
\\[2.5 mm]
&R_{l1}(s,\Psi_j^*(s)),\quad\varphi_{j}^*\le s<\varphi_{j-1}^*,\,1\le j\le {k^*}+1\mbox{ and $j$ is even},
\end{aligned}
\right.
\end{aligned}
\end{equation*}
where $(\varphi_j,\Psi_j)$ and $(\varphi_j^*,\Psi_j^*)$ are defined in Theorem \ref{1-thl1}.
It follows from \eqref{1-pos-arc-bound} that
\begin{align}\label{1-dau1}
&U_l(\varphi_{j-1},\psi_{j-1})-H_{l,\varepsilon}(\varphi_{j-1})\nonumber\\
\leq\,&(U_l(\varphi_j,\psi_j)-H_{l,\varepsilon}(\varphi_j))
\exp\bigg\{\int_{\varphi_j}^{\varphi_{j-1}}r_3(s)\,ds\bigg\},\quad
\mbox{if $j$ is odd},\\
&-V_l(\varphi_{j-1},\psi_{j-1})-H_{l,\varepsilon}(\varphi_{j-1})\nonumber\\
\leq\,&(-V_l(\varphi_j,\psi_j)-H_{l,\varepsilon}(\varphi_j))
\exp\bigg\{\int_{\varphi_j}^{\varphi_{j-1}}r_3(s)\,ds\bigg\},\quad
\mbox{if $j$ is even}.
\label{1-dav1}
\end{align}
If $k$ and $k^*$ are even, it follows from \eqref{1-dau1} and \eqref{1-dav1} that
\begin{align}
U_l(\varphi_0,\psi_0)-H_{l,\varepsilon}(\varphi_0)
&\leq\bigg(U_l(\zeta,\psi_{k+1})-H_{l,\varepsilon}(\zeta)\bigg)
\exp\bigg\{\int_\zeta^{\varphi_0}r_3(s)\,ds\bigg\}<0,\label{1-char-U-end}\\
-V_l(\varphi_0,\psi_0)-H_{l,\varepsilon}(\varphi_0)
&\leq\bigg(-V_l(\zeta,\psi_{k^*+1}^*)-H_{l,\varepsilon}(\zeta)\bigg)
\exp\bigg\{\int_\zeta^{\varphi_0}r_3^*(s)\,ds\bigg\}<0.
\label{1-char-V-end}
\end{align}
Similarly, if $k$ and $k^*$ are odd, one obtains
\begin{align}
U_l(\varphi_0,\psi_0)-H_{l,\varepsilon}(\varphi_0)
&\leq\bigg(-V_l(\zeta,\psi_{k+1})-H_{l,\varepsilon}(\zeta)\bigg)
\exp\bigg\{\int_\zeta^{\varphi_0}r_3(s)\,ds\bigg\}<0,\label{1-char-U-odd}\\
-V_l(\varphi_0,\psi_0)-H_{l,\varepsilon}(\varphi_0)
&\leq\bigg(U_l(\zeta,\psi_{k^*+1}^*)-H_{l,\varepsilon}(\zeta)\bigg)
\exp\bigg\{\int_\zeta^{\varphi_0}r_3^*(s)\,ds\bigg\}<0.
\label{1-char-V-odd}
\end{align}
Combining these estimates yields
\begin{align}\label{1-char-contradiction}
0=\,&\max\bigg\{
\|U_l(\varphi_0,\cdot)\|_{L^\infty(0,m)},
\|V_l(\varphi_0,\cdot)\|_{L^\infty(0,m)}
\bigg\}-H_{l,\varepsilon}(\varphi_0)<0,
\end{align}
which contradicts \eqref{1-char-first-hit}. Letting $\varepsilon\to0$, we obtain
$$
U_l(\varphi,\psi),-V_l(\varphi,\psi)\leq H_l(\varphi),\quad\text{a.e. }(\varphi,\psi)\in[\zeta,{\zeta_{\max}})\times[0,m].
$$
Similarly, the estimate of $(-U_l,V_l)$ can be obtained as follows:
$$
-U_l(\varphi,\psi),\ V_l(\varphi,\psi)\leq H_l(\varphi),\quad\text{a.e. }
(\varphi,\psi)\in[\zeta,{\zeta_{\max}})\times[0,m].
$$
Thus, using \eqref{1-char-Hl}, we obtain that for
$(\varphi_0,\psi_0)$ with $\zeta<\varphi_0\leq\zeta_l$, $\varphi_0<{\zeta_{\max}}$ and $\psi_0\in(0,m)$,
\begin{align}\label{1-dau2}
|U_l(\varphi_0,\psi_0)|,\ |V_l(\varphi_0,\psi_0)|\leq \,&H_l(\varphi_0)
\leq\mu_{21}(M_2\beta_6+M_3)\sigma
+M_{13}\beta_2\sigma\int_\zeta^{\varphi_0}
\varphi^{-(\gamma+5)/2}\,d\varphi\nonumber\\
\leq\,&\bigg(\mu_{21}(M_2\beta_6+M_3)
+\frac{2M_{13}\beta_2}{\gamma+3}\zeta^{-(\gamma+3)/2}\bigg)\sigma.
\end{align}
If ${\zeta_{\max}}>\zeta_l$, then for almost every
$(\varphi_0,\psi_0)\in[\zeta_l,{\zeta_{\max}})\times(0,m)$, we similarly obtain
\begin{align}\label{1-dau4}
|U_l(\varphi_0,\psi_0)|,\ |V_l(\varphi_0,\psi_0)|&\leq H_l(\varphi_0)
\leq H_l(\zeta_l)\bigg(\frac{\zeta_l}{\varphi_0}\bigg)^{M_{12}}
+M_{13}\beta_2\sigma\int_{\zeta_l}^{\varphi_0}
\varphi^{-(\gamma+5)/2}\bigg(\frac{\varphi}{\varphi_0}\bigg)^{M_{12}}\,d\varphi\nonumber\\
&\leq\bigg(\mu_{21}(M_2\beta_6+M_3)
+\frac{2M_{13}\beta_2}{\gamma+3}\zeta^{-(\gamma+3)/2}\bigg)
\sigma\bigg(\frac{\zeta_l}{\varphi_0}\bigg)^{M_{12}}\nonumber\\
&\quad+M_{13}\beta_2\sigma\int_{\zeta_l}^{\varphi_0}
\varphi^{-(\gamma+5)/2}\bigg(\frac{\varphi}{\varphi_0}\bigg)^{M_{12}}\,d\varphi.
\end{align}
If $\beta_i$ $(i=1,2)$ given by \eqref{1-beta*} satisfy
\begin{align}\label{1-eqe3}
\frac{\beta_2}{\beta_1}\leq\frac{7\gamma+9}{6\gamma+10},
\end{align}
there holds
\begin{align}\label{1-m7}
M_{12}=\frac{(7\gamma+9)\beta_1}{8\beta_2}
\geq\frac{3\gamma+5}{4}>\frac{\gamma+3}{2}.
\end{align}
Then one gets from \eqref{1-dau2}--\eqref{1-m7} that
\begin{align*}
&|U_l(\varphi_0,\psi_0)|,\ |V_l(\varphi_0,\psi_0)|\\
\leq\,&\bigg(\zeta_l^{(\gamma+3)/2}
\bigg(\mu_{21}(M_2\beta_6+M_3)
+\frac{2M_{13}\beta_2}{\gamma+3}\zeta^{-(\gamma+3)/2}\bigg)
+\frac{4M_{13}\beta_2}{\gamma-1}\bigg)
\sigma\varphi_0^{-(\gamma+3)/2}\\
=&:M_{14}\sigma\varphi_0^{-(\gamma+3)/2},\quad
\text{a.e. }(\varphi_0,\psi_0)\in
(\zeta,{\zeta_{\max}})\times(0,m).
\end{align*}
Therefore, \eqref{1-eqla3.62} and \eqref{1-eqla3.621} hold.
\par Next, we set
\begin{equation}\label{1-M14}
\begin{aligned}
\hat M_{14}={}&\zeta_l^{(\gamma+3)/2}
\bigg(\mu_{21}(\hat M_2\beta_6+\hat M_3)
+\frac{2\hat M_{13}\hat\beta_2}{\gamma+3}\zeta^{-(\gamma+3)/2}\bigg)
+\frac{4\hat M_{13}\hat\beta_2}{\gamma-1}.
\end{aligned}
\end{equation}
Here $\hat M_i$ $(i=2,3)$ and $\hat M_{13}$ are defined in the same way as $M_i$ and $M_{13}$ in \eqref{1-M2} and \eqref{1-M12} by replacing $\beta_j$ $(j=1,2)$ and $\sigma$ with $\hat\beta_j$ and $\sigma_1$. Then $M_{14}\leq\hat M_{14}$. Combining the estimates with \eqref{1-jie1},
\eqref{1-eq2.92}, \eqref{1-eq2.93},
\eqref{1-lath1}, \eqref{1-lath3}, \eqref{1-qqla1}
and \eqref{1-eqla3.62}--\eqref{1-eqla3.621}
shows that for {almost every}
$(\varphi,\psi)\in(\zeta,{\zeta_{\max}})\times(0,m)$,
\begin{equation*}
\begin{aligned}
&|Q_{\varphi\psi}(\varphi,\psi)|
=\left|-\frac{1}{2}b^{-1}(Q)p(Q)
Q_{\psi}Q_{\varphi}-\frac{1}{2}b^{-1}(Q)(U_l+V_l)\right|\\
&\leq \left|\frac{1}{2}b^{-1}(Q)p(Q)Q_{\psi}Q_{\varphi}\right|
+\frac{1}{2}b^{-1}(Q)|U_l+V_l|
\leq M_{15}\sigma\varphi^{(\gamma-1)/2},\\
&|Q_{\varphi\varphi}(\varphi,\psi)|
=\frac{1}{2}\left|b^{-1/2}(Q)(U_l-V_l)\right|
\leq M_{16}\sigma\varphi^{-1},\\
&|Q_{\psi\psi}(\varphi,\psi)|
=\left|-b^{-1}(Q)p(Q)Q_{\psi}^2
+\frac{1}{2}b^{-3/2}(Q)(U_l-V_l)\right|\\
&\leq \left|b^{-1}(Q)p(Q)Q_{\psi}Q_{\psi}\right|
+\frac{1}{2}\left|b^{-3/2}(Q)(U_l-V_l)\right|
\leq M_{17}\sigma\varphi^{\gamma},
\end{aligned}
\end{equation*}
where
\begin{align}\label{1-M15}
M_{15}=\frac{\mu_{22}\hat{\beta}_2\beta_7}{\hat{\beta}_1}
+\mu_{22}\hat M_{14}\hat{\beta}_2^{\gamma+1},\ \
M_{16}=\mu_{22}\hat M_{14}\hat{\beta}_2^{(\gamma+1)/2},\ \
M_{17}=\frac{\mu_{22}\beta_7^2\sigma_1}{\hat{\beta}_1}
+\mu_{22}\hat M_{14}\hat{\beta}_2^{3(\gamma+1)/2}.
\end{align}
Here $\mu_{22}$ is a positive constant depending only on $\zeta$, $\gamma$ and $R_0$, and $\hat{\beta}_j$ $(j=1,2)$, $\beta_7$  and  $\hat{M}_{14}$ are given by \eqref{1-jie1}, \eqref{1-beta7} and \eqref{1-M14}, respectively. Therefore, \eqref{1-eqla3.66}--\eqref{1-eqla3.65} hold almost everywhere in $(\zeta,{\zeta_{\max}})\times(0,m)$.
}
\end{proof}
\par
It follows
from  \eqref{1-mu1}, \eqref{1-beta*} and \eqref{1-eqe3} that we set
\begin{align}\label{1-sigma3}
\sigma_3=\min\left\{\sigma_1,\sigma _2,\frac{(7\gamma+1)(\gamma-1)}{\mu_6\mu_9(259\gamma^2 + 642\gamma + 379)} \right\}.
\end{align}
Furthermore, under the assumptions of Proposition \ref{1-prola3.2}, we choose
\begin{align}
\beta_8=M_{16},\quad
\beta_9=M_{15},\quad
\beta_{10}=M_{17}.\label{1-eqla3.87}
\end{align}
{Then if $0< \sigma \leq \sigma_3$, one can get from \eqref{1-eqla3.62}--\eqref{1-eqla3.65} that
for almost every $(\varphi,\psi)\in(\zeta,{\zeta_{\max}})\times(0,m)$
\begin{align*}
|Q_{\varphi\varphi}(\varphi,\psi)|\leq\beta_8\sigma\varphi^{-1},\ \
|Q_{\varphi\psi}(\varphi,\psi)|\leq\beta_9\sigma\varphi^{(\gamma-1)/2}, \ \
|Q_{\psi\psi}(\varphi,\psi)|\leq\beta_{10}\sigma\varphi^{\gamma}.
\end{align*}}
We now establish the global existence of the {$C^{1,1}$ solution} in the following theorem.
\begin{theorem}
\label{1-thla3.2}
Given  $g\in C^4([-R_0\sin\theta_0,R_0\sin\theta_0])$ satisfying \eqref{1-eq2.1} and \eqref{1-eq2.43}. Then if $\sigma\le \sigma_3$ with $\sigma_3$ given by \eqref{1-sigma3}, the problem \eqref{1-eq2.226}--\eqref{1-eq2.2210} admits at least one weak solution
$Q\in C_{\rm{loc}}^{1,1}([{\zeta},+\infty)\times[0,m])$ satisfying \eqref{1-lath2}, \eqref{1-lath3} and {the following estimates almost everywhere:}
{
\begin{align}\label{1-eqla3.89}
|Q_{\varphi\varphi}(\varphi,\psi)|\leq\beta_8\sigma\varphi^{-1},\ \
|Q_{\varphi\psi}(\varphi,\psi)|\leq\beta_9\sigma\varphi^{(\gamma-1)/2},\ \
|Q_{\psi\psi}(\varphi,\psi)|\leq\beta_{10}\sigma\varphi^{\gamma},\ \
(\varphi,\psi)\in(\zeta,+\infty)\times(0,m).
\end{align}}
where $\beta _i$ $(i= 8,9,10)$ are positive constants given by \eqref{1-eqla3.87}.
\end{theorem}
\begin{proof}
{
If ${\zeta_{\max}}<+\infty$, Theorem \ref{1-thl1}
and Proposition \ref{1-prola3.2} yield that for
$\zeta<{\tilde\zeta}<{\zeta_{\max}}$,
\begin{align*}
&-\beta_2{\tilde\zeta}\leq Q(\varphi,\psi)\leq-\beta_1\zeta<0,\quad
(\varphi,\psi)\in[\zeta,{\tilde\zeta}]\times[0,m],\\
&|Q_\varphi(\varphi,\psi)|\leq\beta_2,\quad
|Q_\psi(\varphi,\psi)|\leq\beta_7\sigma {\tilde\zeta}^{(\gamma+1)/2},\quad
(\varphi,\psi)\in[\zeta,{\tilde\zeta}]\times[0,m],\\
&|Q_{\varphi\varphi}(\varphi,\psi)|
\leq\beta_8\sigma\zeta^{-1},\
|Q_{\varphi\psi}(\varphi,\psi)|
\leq\beta_9\sigma {\tilde\zeta}^{(\gamma-1)/2},\
|Q_{\psi\psi}(\varphi,\psi)|
\leq\beta_{10}\sigma {\tilde\zeta}^\gamma,\ \
(\varphi,\psi)\in(\zeta,{\tilde\zeta})\times(0,m).
\end{align*}
Consequently, one has
\[
\sup_{\zeta<{\tilde\zeta}<{\zeta_{\max}}}
\|Q\|_{C^{1,1}([\zeta,{\tilde\zeta}]\times[0,m])}<+\infty,
\]
and
\begin{align*}
&(Q({\zeta_{\max}},\cdot),Q_\varphi({\zeta_{\max}},\cdot))
\in C^{1,1}([0,m])\times C^{0,1}([0,m]),\\
&Q({\zeta_{\max}},\psi)\leq-\beta_1{\zeta_{\max}}<0,
\ \ \psi\in[0,m],\ \
Q_\psi({\zeta_{\max}},0)=Q_\psi({\zeta_{\max}},m)=0.
\end{align*}
Then the preceding estimates and the local existence theory
for strictly hyperbolic systems allow the solution
to be extended to $[\zeta,{\zeta_{\max}}+\delta]$
for some $\delta>0$ with the same regularity.
This contradicts the definition of ${\zeta_{\max}}$.
Therefore, ${\zeta_{\max}}=+\infty$ and
\[
Q\in C^{1,1}_{\rm loc}([\zeta,+\infty)\times[0,m]).
\]
The estimates \eqref{1-lath2}, \eqref{1-lath3} and \eqref{1-eqla3.89} follow from uniqueness and Proposition \ref{1-prola3.2}. This completes the proof.
}
\end{proof}
\par Now, by Theorems  \ref{1-th3.1}, \ref{1-th3.2}, and \ref{1-thla3.2}, we set \begin{align*}
\sigma_1^*=\sigma_3 \end{align*}
and
\begin{align*}
  C_1=\tilde{\beta}_1,\quad C_2=\tilde{\beta}_2,\quad C_3=\tilde{\beta}_3,\ \ C_4=\beta_7,\ \ C_5=\max\left\{\beta_4,\beta_5,\beta_6\right\},\ \ C_6=\max\left\{\beta_8,\beta_9,\beta_{10}\right\},
\end{align*}
where $\sigma_3$, $\tilde{\beta}_i$ $(i=1,2,3)$ and $\beta_j$ $(j=4,\cdots,10)$ are given by \eqref{1-sigma3}, \eqref{1-eq3.34}, \eqref{1-eq3.87}, \eqref{1-beta7} and \eqref{1-eqla3.87}, respectively. Then if $ \sigma\leq \sigma_1^*$, the problem \eqref{1-eq2.22}--\eqref{1-eq2.26} admits at least one weak solution $Q\in C_{\rm{loc}}^{1,1}([0,+\infty)\times[0,m])$ satisfying the estimates \eqref{1-eq2.44}--\eqref{1-eq2.461}. Therefore, the proof of Theorem
\ref{1-th2.1} is completed.

\section{Uniqueness}\label{un}\noindent
\par
In this section, we study the uniqueness of sonic-supersonic flows, which can be established by using the precise estimates obtained in the existence theory.
\begin{theorem}
\label{1-th4.1}
 Given  $g\in C^4([-R_0\sin\theta_0,R_0\sin\theta_0])$ satisfying \eqref{1-eq2.1} and \eqref{1-eq2.43}. There exists a positive constant $\sigma_4$ depending only on $\gamma$, $m$ and $R_0$ such that if  $\sigma\le \sigma_4$, the problem \eqref{1-eq2.22}--\eqref{1-eq2.26} admits at most one weak solution $Q\in C_{\rm{loc}}^{0,1}((0,+\infty)\times[0,m])$ satisfying
    \begin{equation}\label{1-eq4.1}
        \begin{aligned}
            -\tilde{\beta}_2\sigma\varphi &\le Q(\varphi, \psi) + \frac{\varphi}{c_*^{1+2/(\gamma-1)}R_0} \le \tilde{\beta}_1\sigma\varphi, \\
            -\tilde{\beta}_2\sigma &\le Q_\varphi(\varphi, \psi) + \frac{1}{c_*^{1+2/(\gamma-1)}R_0} \le \tilde{\beta}_1\sigma,
        \end{aligned}
        \quad  (\varphi, \psi) \in (0, +\infty) \times (0, m).
    \end{equation}
    Furthermore,  $Q_\psi$ satisfies
    \begin{equation}\label{1-eq4.2}
        |Q_\psi(\varphi, \psi)| \le \tilde{\beta}_3\sigma\varphi, \quad  (\varphi, \psi) \in (0, {\zeta}] \times (0, m)
    \end{equation}
    and
    \begin{equation}\label{1-eq4.2*}
        |Q_\psi(\varphi, \psi)| \le \beta_7\sigma\varphi^{(\gamma+1)/2}, \quad (\varphi, \psi) \in ({\zeta}, +\infty) \times (0, m),
    \end{equation}
where ${\tilde{\beta}_i}$ $(i=1,2,3)$ and $\beta_7$ are given by \eqref{1-eq3.34} and \eqref{1-beta7}, respectively.
\end{theorem}
\begin{proof}
Note that the problem \eqref{1-eq2.22}--\eqref{1-eq2.26}
is equivalent to the problem
\eqref{1-eq2.10}--\eqref{1-eq2.15}.
Furthermore, since system \eqref{1-eq2.10} is strictly
hyperbolic for $\varphi\in[\zeta,+\infty)$, the uniqueness
of its weak solutions follows directly from standard
hyperbolic theory \cite{Li85}. Consequently, the proof
of uniqueness for problem
\eqref{1-eq2.22}--\eqref{1-eq2.26}
can be reduced to the region $\varphi\in(0,\zeta)$.
Thanks to Theorem \ref{1-th3.2}, the problem admits a
solution $\tilde Q\in C^{1,1}([0,\zeta]\times[0,m])$
satisfying \eqref{1-eq3.44} and \eqref{1-eq3.89}.
It remains to prove that $\tilde Q$ is the unique weak
solution. Assume that
$\hat Q\in C^{0,1}([0,\zeta]\times[0,m])$
is a weak solution satisfying
\eqref{1-eq4.1} and \eqref{1-eq4.2}. Note that
$$
\div\left(\hat Q_\varphi,-b(\hat Q)\hat Q_\psi\right)
=
\div\left(\hat Q_\psi,-\hat Q_\varphi\right)=0
$$
in $(0,\zeta)\times(0,m)$ in the sense of distributions.
On each strip $[\varepsilon,\zeta]\times[0,m]$,
the normal traces are well defined.

Set
$$
Q(\varphi,\psi)
=\tilde Q(\varphi,\psi)-\hat Q(\varphi,\psi),
\quad(\varphi,\psi)\in[0,\zeta]\times[0,m].
$$
Then $Q\in C^{0,1}([0,\zeta]\times[0,m])$ satisfies
$$
b(\tilde Q(\varphi,\psi))-b(\hat Q(\varphi,\psi))
=h(\varphi,\psi)Q(\varphi,\psi),
\quad(\varphi,\psi)\in(0,\zeta)\times(0,m),
$$
where
$$
h(\varphi,\psi)
=\int_0^1p(t\tilde Q(\varphi,\psi)
+(1-t)\hat Q(\varphi,\psi))\,dt.
$$
Using \eqref{1-eq2.9}, \eqref{1-qq1},
\eqref{1-qq2} and \eqref{1-eq3.89} on $\tilde Q$,
and \eqref{1-eq4.1} and \eqref{1-eq4.2} on $\hat Q$,
direct computations show that
\begin{align*}
&|Q_\varphi(\varphi,\psi)|\leq2\hat\beta_2,\quad
|Q_\psi(\varphi,\psi)|
\leq2\tilde\beta_3\sigma\varphi,\quad
|\tilde Q_{\varphi\psi}(\varphi,\psi)|
\leq\beta_5\sigma,\\
&0<h(\varphi,\psi)
\leq\frac{\mu_{23}}{\hat\beta_1^{3/2}}\varphi^{-3/2},
\quad
|h_\varphi(\varphi,\psi)|
\leq\frac{\mu_{23}\hat\beta_2}{\hat\beta_1^{5/2}}
\varphi^{-5/2},
\quad(\varphi,\psi)\in(0,\zeta)\times(0,m),
\end{align*}
where $\mu_{23}$ is a positive constant depending only
on $\hat\mu$.
Furthermore, the   characteristic estimates of Section \ref{2se}
 and
$\|W_0'\|_{L^\infty(0,m)}\leq C\sigma$ yield that
\begin{equation*}
\|Q_\varphi(\varphi,\cdot)\|_{L^\infty(0,m)}
\leq C\sigma\varphi^{3/4},
\ \ \varphi\in(0,\zeta).
\end{equation*}
Here $C$ is independent of $\varphi$ and $\varepsilon$.
Consequently,
\begin{equation*}
\varepsilon^{-1/4}\int_0^m
Q_\varphi^2(\varepsilon,\psi)\,d\psi
\leq Cm\sigma^2\varepsilon^{5/4}\longrightarrow0.
\end{equation*}
It follows from H\"older's inequality that
\begin{align}\label{1-eq4.3}
\sup_{(0,m)}Q^2(\phi,\cdot)
\leq\,&
\bigg(
\frac1m\int_0^m|Q(\phi,\psi)|\,d\psi
+\int_0^m|Q_\psi(\phi,\psi)|\,d\psi
\bigg)^2\nonumber\\
\leq\,&
\bigg(
\frac1m\int_0^\phi\int_0^m
|Q_\varphi(\varphi,\psi)|\,d\psi\,d\varphi
+\int_0^m|Q_\psi(\phi,\psi)|\,d\psi
\bigg)^2\nonumber\\
\leq\,&
\frac{8}{9m}\phi^{9/4}
\int_0^\phi\int_0^m
\varphi^{-5/4}Q_\varphi^2\,d\psi\,d\varphi
+2m\int_0^mQ_\psi^2(\phi,\psi)\,d\psi,
\quad\phi\in(0,\zeta).
\end{align}

We prove Theorem \ref{1-th4.1} by a weighted energy estimate.
The following identities are justified by regularization
on $(\varepsilon,\phi)\times(0,m)$.
After combining them, we let $\varepsilon\to0^+$. Multiplying the equation of
$\tilde Q$ by $\varphi^{-1/4}\tilde Q_\varphi$
and $-\varphi^{-1/4}\hat Q_\varphi$, respectively,
and integrating gives
\begin{equation}\label{1-eq4.4}
\begin{aligned}
0={}&
\int_0^\phi\int_0^m
\varphi^{-1/4}\tilde Q_\varphi
\tilde Q_{\varphi\varphi}\,d\psi\,d\varphi
-\int_0^\phi\int_0^m
\varphi^{-1/4}(b(\tilde Q)\tilde Q_\psi)_\psi
\tilde Q_\varphi\,d\psi\,d\varphi\\
={}&
\frac12\left.\int_0^m
\varphi^{-1/4}\tilde Q_\varphi^2\,d\psi
\right|_{\varphi=0}^{\varphi=\phi}
+\frac18\int_0^\phi\int_0^m
\varphi^{-5/4}\tilde Q_\varphi^2\,d\psi\,d\varphi
+\frac12\left.\int_0^m
\varphi^{-1/4}b(\tilde Q)\tilde Q_\psi^2\,d\psi
\right|_{\varphi=0}^{\varphi=\phi}\\
&-\frac12\int_0^\phi\int_0^m
(\varphi^{-1/4}b(\tilde Q))_\varphi
\tilde Q_\psi^2\,d\psi\,d\varphi,
\end{aligned}
\end{equation}
and
\begin{align}\label{1-eq4.5}
0={}&
-\int_0^\phi\int_0^m
\varphi^{-1/4}\hat Q_\varphi
\tilde Q_{\varphi\varphi}\,d\psi\,d\varphi
+\int_0^\phi\int_0^m
\varphi^{-1/4}(b(\tilde Q)\tilde Q_\psi)_\psi
\hat Q_\varphi\,d\psi\,d\varphi.
\end{align}
Then a direct computation shows that
\begin{align}\label{1-eq4.6}
&\int_0^\phi\int_0^m
\varphi^{-1/4}(b(\tilde Q)\tilde Q_\psi)_\psi
\hat Q_\varphi\,d\psi\,d\varphi
=
-\left.\int_0^m
\varphi^{-1/4}b(\tilde Q)
\tilde Q_\psi\hat Q_\psi\,d\psi
\right|_{\varphi=0}^{\varphi=\phi}\nonumber\\
&+\int_0^\phi\int_0^m
(\varphi^{-1/4}b(\tilde Q))_\varphi
\tilde Q_\psi\hat Q_\psi\,d\psi\,d\varphi+\int_0^\phi\int_0^m
\varphi^{-1/4}b(\tilde Q)
\tilde Q_{\varphi\psi}\hat Q_\psi\,d\psi\,d\varphi.
\end{align}
Similarly,
\begin{equation}\label{1-eq4.7}
\begin{aligned}
0={}&
\frac12\left.\int_0^m
\varphi^{-1/4}\hat Q_\varphi^2\,d\psi
\right|_{\varphi=0}^{\varphi=\phi}
+\frac18\int_0^\phi\int_0^m
\varphi^{-5/4}\hat Q_\varphi^2\,d\psi\,d\varphi\\
&+\frac12\left.\int_0^m
\varphi^{-1/4}b(\hat Q)\hat Q_\psi^2\,d\psi
\right|_{\varphi=0}^{\varphi=\phi}
-\frac12\int_0^\phi\int_0^m
(\varphi^{-1/4}b(\hat Q))_\varphi
\hat Q_\psi^2\,d\psi\,d\varphi,
\end{aligned}
\end{equation}
and
\begin{align}\label{1-eq4.8}
0&=
-\left.\int_0^m
\varphi^{-1/4}\tilde Q_\varphi
\hat Q_\varphi\,d\psi
\right|_{\varphi=0}^{\varphi=\phi}\nonumber
-\frac14\int_0^\phi\int_0^m
\varphi^{-5/4}\tilde Q_\varphi
\hat Q_\varphi\,d\psi\,d\varphi\nonumber\\
&\quad+\int_0^\phi\int_0^m
\varphi^{-1/4}\hat Q_\varphi
\tilde Q_{\varphi\varphi}\,d\psi\,d\varphi-\int_0^\phi\int_0^m
\varphi^{-1/4}b(\hat Q)
\tilde Q_{\varphi\psi}\hat Q_\psi\,d\psi\,d\varphi.
\end{align}
Integrating by parts leads to
\begin{align}\label{1-eq4.9}
\int_0^\phi\int_0^m
\varphi^{-1/4}hQ\tilde Q_{\varphi\psi}
\tilde Q_\psi\,d\psi\,d\varphi=
\frac12\left.\int_0^m
\varphi^{-1/4}hQ\tilde Q_\psi^2\,d\psi
\right|_{\varphi=0}^{\varphi=\phi}
-\frac12\int_0^\phi\int_0^m
(\varphi^{-1/4}hQ)_\varphi
\tilde Q_\psi^2\,d\psi\,d\varphi.
\end{align}
Combining \eqref{1-eq4.4}--\eqref{1-eq4.8}
with \eqref{1-eq4.9} yields
$$
0=I_3+I_4+I_5,
$$
where
\begin{equation*}
\begin{aligned}
I_3={}&
\frac18\int_0^\phi\int_0^m
\varphi^{-5/4}
\bigg(
\tilde Q_\varphi^2+\hat Q_\varphi^2
-2\tilde Q_\varphi\hat Q_\varphi
\bigg)\,d\psi\,d\varphi\\
={}&
\frac18\int_0^\phi\int_0^m
\varphi^{-5/4}Q_\varphi^2\,d\psi\,d\varphi,\\
I_4={}&
-\frac12\int_0^\phi\int_0^m
\bigg(
(\varphi^{-1/4}b(\tilde Q))_\varphi\tilde Q_\psi^2
+(\varphi^{-1/4}b(\hat Q))_\varphi\hat Q_\psi^2
-2(\varphi^{-1/4}b(\tilde Q))_\varphi
\tilde Q_\psi\hat Q_\psi
\bigg)\,d\psi\,d\varphi\\
&-\frac12\int_0^\phi\int_0^m
(\varphi^{-1/4}hQ)_\varphi
\tilde Q_\psi^2\,d\psi\,d\varphi\\
&-\int_0^\phi\int_0^m
\varphi^{-1/4}
\bigg(
b(\hat Q)\tilde Q_{\varphi\psi}\hat Q_\psi
-b(\tilde Q)\tilde Q_{\varphi\psi}\hat Q_\psi
+hQ\tilde Q_{\varphi\psi}\tilde Q_\psi
\bigg)\,d\psi\,d\varphi\\
={}&
-\frac12\int_0^\phi\int_0^m
(\varphi^{-1/4}b(\tilde Q))_\varphi
Q_\psi^2\,d\psi\,d\varphi
-\frac12\int_0^\phi\int_0^m
\varphi^{-1/4}h(\tilde Q_\psi+\hat Q_\psi)
Q_\varphi Q_\psi\,d\psi\,d\varphi\\
&-\frac12\int_0^\phi\int_0^m
(\varphi^{-1/4}h)_\varphi
(\tilde Q_\psi+\hat Q_\psi)
QQ_\psi\,d\psi\,d\varphi
-\int_0^\phi\int_0^m
\varphi^{-1/4}h\tilde Q_{\varphi\psi}
QQ_\psi\,d\psi\,d\varphi,
\end{aligned}
\end{equation*}
and
\begin{equation*}
\begin{aligned}
I_5={}&
\frac12\left.\int_0^m
\varphi^{-1/4}
\bigg(
\tilde Q_\varphi^2+\hat Q_\varphi^2
-2\tilde Q_\varphi\hat Q_\varphi
\bigg)\,d\psi
\right|_{\varphi=0}^{\varphi=\phi}\\
&+\frac12\left.\int_0^m
\varphi^{-1/4}
\bigg(
b(\tilde Q)\tilde Q_\psi^2
-2b(\tilde Q)\tilde Q_\psi\hat Q_\psi
+b(\hat Q)\hat Q_\psi^2
+hQ\tilde Q_\psi^2
\bigg)\,d\psi
\right|_{\varphi=0}^{\varphi=\phi}\\
={}&
\frac12\left.\int_0^m
\varphi^{-1/4}Q_\varphi^2\,d\psi
\right|_{\varphi=\phi}
+\frac12\left.\int_0^m
\varphi^{-1/4}b(\tilde Q)Q_\psi^2\,d\psi
\right|_{\varphi=\phi}\\
&+\frac12\left.\int_0^m
\varphi^{-1/4}h(\tilde Q_\psi+\hat Q_\psi)
QQ_\psi\,d\psi
\right|_{\varphi=\phi}.
\end{aligned}
\end{equation*}
Therefore,
\begin{align*}
&\frac18\int_0^\phi\int_0^m
\varphi^{-5/4}Q_\varphi^2\,d\psi\,d\varphi
-\frac12\int_0^\phi\int_0^m
(\varphi^{-1/4}b(\tilde Q))_\varphi
Q_\psi^2\,d\psi\,d\varphi\\
&\quad
+\frac12\left.\int_0^m
\varphi^{-1/4}Q_\varphi^2\,d\psi
\right|_{\varphi=\phi}
+\frac12\left.\int_0^m
\varphi^{-1/4}b(\tilde Q)Q_\psi^2\,d\psi
\right|_{\varphi=\phi}\\
={}&
\frac12\int_0^\phi\int_0^m
\varphi^{-1/4}h(\tilde Q_\psi+\hat Q_\psi)
Q_\varphi Q_\psi\,d\psi\,d\varphi
+\frac12\int_0^\phi\int_0^m
(\varphi^{-1/4}h)_\varphi
(\tilde Q_\psi+\hat Q_\psi)
QQ_\psi\,d\psi\,d\varphi\\
&+\int_0^\phi\int_0^m
\varphi^{-1/4}h\tilde Q_{\varphi\psi}
QQ_\psi\,d\psi\,d\varphi
-\frac12\left.\int_0^m
\varphi^{-1/4}h(\tilde Q_\psi+\hat Q_\psi)
QQ_\psi\,d\psi
\right|_{\varphi=\phi}.
\end{align*}
Using the preceding estimates, \eqref{1-jie1}
and H\"older's inequality gives
\begin{align}\label{1-eq4.10}
&\frac18\int_0^\phi\int_0^m
\varphi^{-5/4}Q_\varphi^2\,d\psi\,d\varphi
+\frac{\mu_{24}\hat\beta_1}{\hat\beta_2^{3/2}}
\int_0^\phi\int_0^m
\varphi^{-7/4}Q_\psi^2\,d\psi\,d\varphi\nonumber\\
&\quad
+\frac12\left.\int_0^m
\varphi^{-1/4}Q_\varphi^2\,d\psi
\right|_{\varphi=\phi}
+\frac{\mu_{24}}{\hat\beta_2^{1/2}}
\left.\int_0^m
\varphi^{-3/4}Q_\psi^2\,d\psi
\right|_{\varphi=\phi}\nonumber\\
\leq\,&
M_{18}\sigma\int_0^\phi\int_0^m
\varphi^{-3/4}|Q_\varphi Q_\psi|\,d\psi\,d\varphi
+M_{19}\sigma\int_0^\phi\int_0^m
\varphi^{-7/4}|QQ_\psi|\,d\psi\,d\varphi\nonumber\\
&+M_{18}\sigma\left.\int_0^m
\varphi^{-3/4}|QQ_\psi|\,d\psi
\right|_{\varphi=\phi}\nonumber\\
\leq\,&
\frac{M_{18}\sigma}{2}\int_0^\phi\int_0^m
\varphi^{-3/4}Q_\varphi^2\,d\psi\,d\varphi
+\frac{M_{18}\sigma}{2}\int_0^\phi\int_0^m
\varphi^{-3/4}Q_\psi^2\,d\psi\,d\varphi\nonumber\\
&+\frac{M_{19}\sigma}{2}\int_0^\phi\int_0^m
\varphi^{-7/4}Q_\psi^2\,d\psi\,d\varphi
+\frac{M_{19}\sigma}{2}\int_0^\phi\int_0^m
\varphi^{-7/4}Q^2\,d\psi\,d\varphi\nonumber\\
&+\frac{M_{18}\sigma}{2}\left.\int_0^m
\varphi^{-3/4}Q_\psi^2\,d\psi
\right|_{\varphi=\phi}
+\frac{M_{18}\sigma}{2}\left.\int_0^m
\varphi^{-3/4}Q^2\,d\psi
\right|_{\varphi=\phi},
\end{align}
where $\mu_{24}$ is a positive constant depending only
on $\hat\mu$, and
\begin{align}\label{1-M18}
M_{18}
=\frac{\mu_{23}\tilde\beta_3}{\hat\beta_1^{3/2}},
\quad
M_{19}
=\frac{\mu_{23}\beta_5}{\hat\beta_1^{3/2}}
+\frac{5\mu_{23}\hat\beta_2\tilde\beta_3}
{4\hat\beta_1^{5/2}}.
\end{align}
Using \eqref{1-eq4.3}, one obtains
\begin{align*}
&\frac18\int_0^\phi\int_0^m
\varphi^{-5/4}Q_\varphi^2\,d\psi\,d\varphi
+\frac{\mu_{24}\hat\beta_1}{\hat\beta_2^{3/2}}
\int_0^\phi\int_0^m
\varphi^{-7/4}Q_\psi^2\,d\psi\,d\varphi\\
&\quad
+\frac12\left.\int_0^m
\varphi^{-1/4}Q_\varphi^2\,d\psi
\right|_{\varphi=\phi}
+\frac{\mu_{24}}{\hat\beta_2^{1/2}}
\left.\int_0^m
\varphi^{-3/4}Q_\psi^2\,d\psi
\right|_{\varphi=\phi}\\
\leq\,&
\bigg(
\frac{M_{18}\sigma}{2}\phi^{1/2}
+\bigg(\frac{4M_{18}\sigma}{9}
+\frac{8M_{19}\sigma}{27}\bigg)\phi^{3/2}
\bigg)
\int_0^\phi\int_0^m
\varphi^{-5/4}Q_\varphi^2\,d\psi\,d\varphi\\
&+\bigg(
\frac{M_{18}\sigma}{2}\phi
+\frac{(1+2m^2)M_{19}\sigma}{2}
\bigg)
\int_0^\phi\int_0^m
\varphi^{-7/4}Q_\psi^2\,d\psi\,d\varphi\\
&+\frac{(1+2m^2)M_{18}\sigma}{2}
\left.\int_0^m
\varphi^{-3/4}Q_\psi^2\,d\psi
\right|_{\varphi=\phi}.
\end{align*}
Since $\phi<\zeta<1$, it follows that
\begin{align*}
&\bigg(
\frac18-\frac{17M_{18}\sigma}{18}
-\frac{8M_{19}\sigma}{27}
\bigg)
\int_0^\phi\int_0^m
\varphi^{-5/4}Q_\varphi^2\,d\psi\,d\varphi\\
&+\bigg(
\frac{\mu_{24}\hat\beta_1}{\hat\beta_2^{3/2}}
-\frac{M_{18}\sigma+(1+2m^2)M_{19}\sigma}{2}
\bigg)
\int_0^\phi\int_0^m
\varphi^{-7/4}Q_\psi^2\,d\psi\,d\varphi\\
&+\frac12\left.\int_0^m
\varphi^{-1/4}Q_\varphi^2\,d\psi
\right|_{\varphi=\phi}
+\bigg(
\frac{\mu_{24}}{\hat\beta_2^{1/2}}
-\frac{(1+2m^2)M_{18}\sigma}{2}
\bigg)
\left.\int_0^m
\varphi^{-3/4}Q_\psi^2\,d\psi
\right|_{\varphi=\phi}
\leq0.
\end{align*}
Choose
\begin{equation}\label{1-sigma4}
\begin{aligned}
\sigma_4=\min\bigg\{
&\sigma_1,\sigma_2,\sigma_3,
\frac{27}{204M_{18}+64M_{19}},\\
&\frac{2\mu_{24}\hat\beta_1}
{\hat\beta_2^{3/2}
\bigg(M_{18}+(1+2m^2)M_{19}\bigg)},
\frac{2\mu_{24}}
{\hat\beta_2^{1/2}(1+2m^2)M_{18}}
\bigg\}.
\end{aligned}
\end{equation}
For $\sigma\leq\sigma_4$, all the coefficients in the
preceding inequality are nonnegative. In particular,
\[
\int_0^m Q_\varphi^2(\phi,\psi)\,d\psi=0
\quad\text{for } \ \ \phi\in(0,\zeta).
\]
Since $Q(0,\cdot)=0$, we conclude that
$$
\tilde Q(\varphi,\psi)=\hat Q(\varphi,\psi),
\quad(\varphi,\psi)\in[0,\zeta]\times[0,m].
$$
\end{proof}

 {\bf Acknowledgement.} The research of Chunpeng Wang is partially supported by the National Natural Science Foundation of China (Grant No. 12471221) and the Fundamental and Interdisciplinary Disciplines Breakthrough Plan of the Ministry of Education of China (Grant No. JYB2025XDXM308).  The research of Zihao Zhang is  supported by the China Scholarship Council (Grant No. 202606170019).
\par {\bf Data availability.} No data was used for the research described in the article.
\par {\bf Conflict of interest.} This work does not have any conflicts of interest.
\par {\bf Declaration of AI-assisted tools.}
During the preparation of the present paper, 
the authors independently developed the proof framework, performed all mathematical calculations underlying the proofs, and wrote the initial draft. AI assistance was used primarily during drafting and revision to improve grammar and exposition and to identify potential errors. The authors take full responsibility for the accuracy, originality, and integrity of the paper.


\end{document}